\documentclass{coulonpaper}

\usepackage{mathtools}
\usepackage{soul}
\renewcommand*{\dist}[3][]{
	\ifthenelse{\equal{#1}{}}
		{\mathchoice%
			{d\!\left(#2,#3\right)}%
			{d(#2,#3)}%
			{d(#2,#3)}%
			{d(#2,#3)}%
		}
		{\mathchoice%
			{d_{#1}\left(#2,#3\right)}%
			{d_{#1}(#2,#3)}%
			{d_{#1}(#2,#3)}%
			{d_{#1}(#2,#3)}%
		}
}

\renewcommand*{\distV}[1][]{
	\ifthenelse{\equal{#1}{}}
		{\left| \ . \ \right|}
		{
			\ifthenelse{\equal{#1}{SC}}
			{\left\| \ . \ \right\|}
			{\left| \ . \ \right|_{#1}}
		}
}

\renewcommand{\phi}{\oldphi}

\title{Growth gap in hyperbolic groups and amenability}
\author{Rémi Coulon, Françoise Dal'Bo, Andrea Sambusetti}

\begin{document} 
\selectlanguage{english}


\maketitle
\begin{abstract}
We prove a general version of the  amenability conjecture in the unified setting of a Gromov hyperbolic group $G$ acting properly cocompactly either on its Cayley graph, or on a  CAT(-1)-space. 
Namely, for any subgroup $H$ of $G$, we show that $H$ is co-amenable in $G$   if and only if their exponential growth rates  (with respect to the prescribed action) coincide.
For this, we prove a quantified, representation-theoretical version of Stadlbauer's amenability criterion for group  extensions of  a  topologically transitive subshift of finite type, in terms of  the spectral radii of the classical Ruelle transfer operator and its   corresponding extension.
As a consequence, we are able to show that, in our enlarged context, there is a gap between the exponential growth rate of a group with Kazhdan's property (T) and the ones of its infinite index subgroups. This also generalizes a well-known theorem of Corlette for lattices of the quaternionic hyperbolic space or the Cayley hyperbolic plane.
\end{abstract}

\tableofcontents

%
\section{Introduction}
%

Amenability  has a large number of equivalent formulations. 
In a seminal work dating back to 1959, Kesten proved that a finitely generated group $Q$ is amenable if and only if $1$ is the spectral radius of the Markov operator associated to a symmetric random walk on $Q$ whose support generates $Q$ \cite{Kesten:1959wx}.
Given a finite generating set $S$ of $Q$, Grigorchuk \cite{Grigorchuk:1980wx} and Cohen \cite{Cohen:1982gt} independently related  the spectral radius $\rho$ for the random walk with uniform probability measure supported on $S\cup S^{-1}$, to the exponential growth rates of the free group $\F(S)$ and the kernel $N$ of the canonical projection $\F(S) \twoheadrightarrow Q$.
Recall that the \emph{exponential growth rate} of a  discrete  group $G$ of isometries of a proper  metric space $X$, denoted by $\omega(G,X)$, (or simply $\omega_G$ if there is no ambiguity) is
\begin{displaymath}
	\omega(G,X) = \limsup_{r \to \infty} \frac 1r \ln \card{\set{g \in G}{\dist{gx}x \leq r}},
\end{displaymath}
where $x$ is any point of $X$.
The Cohen-Grigorchuck formula states that 
\begin{equation}
\label{eqn: formula grigorchuck} 
	\rho = \frac{\sqrt{e^{\omega_{\F(S)}}}}{1+ e^{\omega_{\F(S)}}} \left( \frac {\sqrt{e^{\omega_{\F(S)}}}}{e^{\omega_N}} + \frac {e^{\omega_N}}{\sqrt{e^{\omega_{\F(S)}}}}\right),
\end{equation} 
where $\omega_N$ and $\omega_{\F(S)}$ are the respective growth rates of $N$ and $\F(S)$ acting on the Cayley graph of $\F(S)$ with respect to $S$.
This immediately  yielded, by Kesten's criterion,   a characterization of  amenability for  a group $Q$ generated by a subset $S$, in terms of the growth of the relator subgroup:  the quotient $Q= \mathbf{F}(S) /N$ is amenable if and only if   $\omega_N=\omega_{\mathbf{F}(S)}=\log(2|S|-1)$.

\medskip
Almost at the same time, a geometric version of Kesten's criterion, in terms of the  bottom of the $L^2$-spectrum of the Laplace operator,  was discovered in Riemannian geometry  by Brooks 
\cite{Brooks:1981jp,Brooks:1985ky} -- see also \cite{Burger:1988ib}. 
For any normal covering  $\hat M \rightarrow M$ of a  Riemannian manifold $M$ of {\em finite topological type},  if the automorphism group of $\hat M$ is amenable, then   $\lambda_0 (\hat M) = \lambda_0 (M)$.
The converse implication also holds, whenever $M$ satisfies in addition a Cheeger-type condition \cite[Theorem~2]{Brooks:1985ky}. 
For instance, if $M=\H^n/G$ is a convex-cocompact hyperbolic manifold with $\lambda_0(M) < (n − 1)^2/4$, this Cheeger-type condition automatically holds.
In this case, realizing the cover as $\hat M = \H^n/N$, with $N$ normal in $G$, Brooks' theorem precisely says that $\lambda_0(\hat M) = \lambda_0(M)$ if and only if the quotient group $Q = G/N$ is amenable.
Coupling this with Sullivan's formula relating 
$\lambda_0 (M)$ to  $\omega_G$  for discrete subgroups of isometries of the hyperbolic space $\H^n$,
\begin{equation} 
 \label{eqn: formula sullivan}
 \lambda_0 (M) =  \left\{
	\begin{array}{ll}
	  \omega_G (n-1 - \omega_G)  &  \mbox{  if  } \omega_G \geq \displaystyle\frac{n-1}{2}    \\
	  \lambda_0  (\mathbf{H}^n)   &          \mbox{  if  } \omega_G \leq  \displaystyle\frac{n-1}{2}  
	\end{array},
	\right.
\end{equation}
(where $\omega_G$ is computed, this time,  with respect to the action of $G$ on the hyperbolic space), Brook's theorem gives an analogue of Cohen-Grigorchuk statement  for particular, discrete groups of isometries of $\H^n$:  for any normal Riemannian covering $\hat M=  \tilde M/N$ of a convex-cocompact hyperbolic manifold $M= \tilde M/G$ with $\omega_G> (n-1)/2 $, the automorphism group  $Q=G/N$  of $\hat M$   is amenable if and only if $\omega_G = \omega_N$. This holds, for instance, for uniform lattices. 
The result was also  extended to hyperbolic, non-uniform lattices in \cite{Roblin:2013bh}. 

\medskip
Beyond the evident formal analogies of these  results -- negatively curved cocompact groups look like free groups at large scale, while the bottom of the Laplacian of $M$ can be related to  the spectral radius of the Markov operator associated to particular random walks on $G=\pi_1(M)$, induced by the heat kernel of $\tilde M$ \cite{Brooks:1981jp} -- the two statements did not live on a common ground. 
This opened the door to intensive research for  a unifying, generalized setting, and a deeper understanding of the relations between these results in dynamical terms.

\medskip
The first  statement  in a general setting was given in 2005  by Roblin \cite{Roblin:2005fn}.
Given a discrete group $G$ of isometries of a $\operatorname{CAT}(-1)$ space, he proved, using  Patterson-Sullivan theory,  that amenability of the quotient $G/N$ always implies the relation  $\omega_G=\omega_N$.
In this generality, it is worth to remark that the reciprocal is not true.
Indeed, there exist Kleinian, geometrically finite groups $G$ -- even lattices in pinched, variable negative curvature \cite{Dalbo:2000eh,DalBo:2017gq} -- admitting a parabolic subgroup $P$ with $\omega_P = \omega_G$.
Such groups give  easy counterexamples to the converse implication, by taking for $N$ the normal closure of $P$ in $G$.
Indeed in most of these cases $G/N$ contains free subgroups and is not amenable. 
The most accredited version of  the {\em Amenability Problem} in the last decade can be stated as follows. 
{\em Given a group $G$ acting properly on a hyperbolic space $X$ and a normal subgroup $N$ of $G$,  under which circumstances does the equality $\omega_N=\omega_G$ imply that the quotient group $Q=G/N$ is amenable?}

\medskip
Clearly, for a group $G$ acting on a general space $X$, an exact formula as (\ref{eqn: formula grigorchuck}) or  (\ref{eqn: formula sullivan}) is hopeless. 
Rather, one should expect that the equality $\omega_N=\omega_G$  reflects  the qualitative behavior of the dynamics of $G/N$ on the space  $X/N$.
Nevertheless an exact relation, in terms of the asymptotic behavior of the spectral radii of  random walks on $G/N$ with probability measure supported by large spheres, resists to this general setting, allowing to show the ``easy part'' of the implication above even in the generality of cocompact group actions on Gromov-hyperbolic spaces (see for instance Propositions~\ref{res: roblin - maj spec rad} and \ref{res: roblin - equ spec rad} in \autoref{sec: roblin}).

\medskip
A  substantial step forward in the solution of the Amenability Problem is due to  Stadlbauer  \cite{Stadlbauer:2013dg}, who generalized Kesten's amenability criterion in terms of group extensions of topological Markov chains.
More precisely, he considered a topologically mixing  subshift of finite type $(\Sigma, \sigma)$ together with a topologically transitive extension $(\Sigma_\theta, \sigma_\theta) $ of this system by a locally constant evaluation map $\theta: \Sigma \rightarrow Q$ into a discrete group $Q$.
He proved that $Q$ is amenable if and only if the \emph{Gurevi\v{c} pressures} of the two systems (with respect to a weakly symmetric potential with H\"older variations) coincide%
\footnote{Actually Stadlbauer's criterion works in a slightly more general context which allow to consider symbolic dynamical systems over an infinite alphabet.
Nevertheless in the context of hyperbolic groups a subshift of finite type is sufficient to conclude.}.
As an application, Stadlbauer solved the Amenability Problem for the class of {\em essentially free} groups $G$ of isometries  of  $\H^n$,  for the first time without assuming that $\omega_G>   (n-1)/2$.  \linebreak
This result was  recently generalized by Dougall and Sharp, using Stadbauer's criterion, to the class of  convex-cocompact  groups of isometries of  pinched, negatively curved  Cartan-Hadamard manifolds \cite{Dougall:2014wo}.
 
 \medskip
The first result of this paper solves the Amenability Problem in an enlarged context encompassing two very different cases.
 The first one, of algebraic nature, concerns the growth of groups with respect to the word metric.
 The second case, coming from the geometry, focuses on the action of a group on a negatively curved Riemannian manifold or a $\operatorname{CAT}(-1)$ space.
 The aim is  two-fold: to give a self-contained proof of all these results in a unified setting,  and to make clear the minimum algebraic and geometric  structure needed.

\begin{theo}[see \autoref{res: main theo amenability}]
\label{res: main theo amenability - intro}
	Let $G$ be a group acting properly co-compactly by isometries on a Gromov hyperbolic space $X$.
	We assume that one of the following holds.
	Either
	\begin{enumerate}
		\item $X$ is the Cayley graph of $G$ with respect to a finite generating set, or
		\item $X$ is a $\operatorname{CAT}(-1)$ space.
	\end{enumerate}	
	Let $H$ be a subgroup of $G$, and let $\omega_G$ and $\omega_H$ denote the exponential growth rates of  $G$ and $H$ acting on $X$.
	The subgroup $H$ is co-amenable in $G$ if and only if $\omega_H = \omega_G$.
\end{theo}

\noindent Recall that the action of a group $G$ on a space $X$ is  {\em amenable} if $X$ admits  a $G$-invariant mean, and that $H$ is called {\em co-amenable} in $G$ if the action of $G$ on the left coset space $H\backslash G$ is amenable.
When the subgroup $H$ is normal in $G$ then $H$ is co-amenable in $G$ if and only if $G/H$ is an amenable group. 
Notice however that in the above theorem we do not assume that $H$ is a normal subgroup.

\medskip
Note that the $\operatorname{CAT}(-1)$ case  in the above theorem extends the Riemannian convex-compact situation studied by Dougall and Sharp \cite{Dougall:2014wo}.
Nevertheless, going from negatively curved manifolds to $\operatorname{CAT}(-1)$ spaces is a substantial generalization.
Indeed, Dougall and Sharp explicitly use the Riemannian structure to encode the geodesic flow via Markov sections.
To the best of our knowledge there is no such coding for the geodesic flow on $\operatorname{CAT}(-1)$ spaces.
We explain at the end of the introduction our strategy to bypass this difficulty.

\medskip
The easy part of \autoref{res: main technical theorem - intro} is the ``only if'' implication.
As we mentioned before, this direction was proved for  normal subgroups of   discrete groups acting on $\operatorname{CAT}(-1)$ spaces by Roblin \cite{Roblin:2005fn}.
In \cite{Roblin:2013bh} Roblin and Tapie sketched how to extend the argument to the case of groups acting on a Gromov hyperbolic space.
Nevertheless, we decided to report in \autoref{sec: roblin} a complete proof of this fact via random walks, for general subgroups of Gromov hyperbolic groups, since this  also gives an exact formula which is similar to Sullivan's one for the bottom of the Laplacian of hyperbolic quotients   (see \autoref{res: roblin}).

\medskip
On the other hand, our proof of the converse implication is strongly  inspired by  Stadlbauer's work  \cite{Stadlbauer:2013dg} and relies on a variation of his amenability criterion.
However, our approach makes an explicit use of representation theory and operator algebra, which was somehow hidden in \cite{Stadlbauer:2013dg}.
We hope that this point of view can enlighten the key conceptual arguments and clarify the exposition.
More precisely we take advantage of the Hulanicki–Reiter criterion for amenable actions: the action of a discrete group $G$ on a set $Y$ is amenable if and only if the induced regular representation $\lambda \colon G \to \mathcal U(\ell^2(Y))$ admits almost invariant vectors \cite[Theorem~G.3.2]{BekHarVal08}.
Assume now that $(\Sigma_\theta, \sigma_\theta)$ is the extension of a subshift of finite type $(\Sigma, \sigma)$ by a locally constant map $\theta \colon \Sigma \to G$.
We associate the classical Ruelle transfer operator $\mathcal L$ to the original system $(\Sigma, \sigma)$.
On the other hand, given an action of $G$ on a set $Y$, we endow the extended system $(\Sigma_\theta,\sigma_\theta)$ with a \emph{twisted transfer operator} $\mathcal L_\lambda$ which is naturally related to the induced unitary representation $\lambda \colon G \to \mathcal U(\ell^2(Y))$ (see \autoref{sec: twisted transfer operator}).
The twisted transfer operator acts on a subspace of the space of continuous functions $C( \Sigma , \ell^2(Y))$  (the appropriate, Hölder regularity will be described in \autoref{sec: function spaces}). 
Using the uniform convexity of Hilbert spaces, we relate the difference between the spectral radii $\rho$ and $\rho_\lambda$ of $\mathcal L$ and  $\mathcal L_\lambda$ respectively, to the existence of almost invariant vectors for the representation $\lambda$. 

\begin{theo}[see Theorems~\ref{res: main technical theorem} and \ref{res: generalization stadlbauer}]
\label{res: main technical theorem - intro}
Let $(\Sigma,\sigma)$ be a  topologically transitive subshift of finite type.
Let $F\colon \Sigma \rightarrow \mathbf{R}_+^*$ be a potential with $\alpha$-bounded H\"older variations (for some $\alpha \in \mathbf{R}_+^*$), and $\mathcal L$ be the Ruelle transfer operator associated with $F$.
Let $G$ be a finitely generated group and $\theta: \Sigma \rightarrow G$ a locally constant map.
Assume that the extension $(\Sigma_\theta,\sigma_\theta)$ of $(\Sigma, \sigma)$ by $\theta$ has the visibility property. 
Then the following holds.
\begin{enumerate}
	\item For every finite subset $S$ of $G$ and every $\epsilon \in \mathbf{R}_+^*$, there exists $\eta \in \mathbf{R}_+^*$ with the following property: 
if $G$ acts on a set $Y$, and $\mathcal L_\lambda$ is the corresponding twisted transfer operator,  the condition $\rho_\lambda > (1 - \eta)\rho$ implies that the representation $\lambda \colon G \to \mathcal U( \ell^2(Y))$ admits an $(S,\epsilon)$-invariant vector.
	\item In particular, if $\rho_\lambda =\rho$ then the action of $G$ on $Y$ is amenable.
\end{enumerate}
\end{theo}

The second statement of this theorem easily follows from the first point, and is very similar to the one obtained by Stadlbauer \cite{Stadlbauer:2013dg} when $(\Sigma,\sigma)$ is a subshift of finite type. 
Let us highlight a few important differences though.
Unlike Stadlbauer's proof, our approach does not really use the Gibbs measure but simply the Perron-Frobenius theorem.
Therefore we do not ask the original system $(\Sigma, \sigma)$ to be \emph{topologically mixing}, but just \emph{topologically transitive} (this is much weaker, as in many situations one can always reduce to an irreducible component of the system). 
Secondly, we consider an extension $(\Sigma_\theta, \sigma_\theta)$ of the initial system by the whole group $G$, and only assume that it has the \emph{visibility property} (which means that the extended flow visits almost the whole group $G$), 
whereas  Stadlbauer extends the initial system by the quotient $Q$, and assumes that this extension is \emph{topologically transitive}.
This is one of the key points which allows us to consider any subgroup of a hyperbolic group and not only normal subgroups.
Moreover, as we state our result in terms of spectral radius instead of pressure, we do not need any kind of symmetry for the potential $F$ (this was already observed by Jaerisch \cite{Jaerisch:2015cs}).
\medskip

More importantly, our approach provides a \emph{quantitative} version%
\footnote{In the proof of \autoref{res: main technical theorem - intro}~(i), we choose to work with ultra-limit of Banach spaces: this has the advantage of  simplifying the arguments involving almost invariant vectors. As a consequence, we do not provide a precise formula for $\eta$ in terms of $S$ and $\epsilon$; 
nevertheless, a careful reader could go through the arguments and make the relation between these quantities explicit.}
 of Stadlbauer's statement. \linebreak
In this perspective, the first statement in the above theorem is close to some results of Dougall in \cite{Dougall:2017va}, which also includes  more concrete representation theory (nevertheless she assumes mixing of the initial system, and considers only normal subgroups to ensure the transitivity of the extended  system, as well as a condition called {\em linear visibility with reminders}, a bit stronger than ours, to control the return times of the flow in a fixed cylinder).
The quantitative version of the amenability criterion (see  \autoref{res: Kazhdan - spectral gap}) makes apparent the following consequence for groups satisfying Kazhdan's property (T).

\begin{theo}[see \autoref{res: main theo property T}]
\label{res: main theo property T - intro}
	Let $G$ be a group with Kazhdan's property (T) acting properly co-compactly by isometries on a hyperbolic space $X$.
	We assume that one of the following holds:
	\begin{enumerate}
		\item either $X$ is the Cayley graph of $G$ with respect to a finite generating set, 
		\item or $X$ is a CAT($-1$) space.
	\end{enumerate}
	Then, there exists $\eta > 0$ with the following property.
	Let $H$ be a subgroup of $G$, and let $\omega_G$, $\omega_H$ denote the exponential growth rates of $G$ and $H$ acting on $X$.
	If $\omega_H > \omega_G - \eta$, then $H$ is a finite index subgroup of $G$.
\end{theo}

We stress the fact that also in this statement $H$ is not assumed to be normal. 
So, this gives the following generalization of Corlette's celebrated growth gap  theorem \cite[Corollay~2]{Corlette:1990br} for subgroups of lattices of  rank one symmetric spaces of negative curvature  possessing Kazhdan's Property (T), i.e. the quaternionic hyperbolic space $\H^n_{\mathbb{H}}$ or the Cayley hyperbolic plane $\H^2_{\mathbb{O}}$.
Corlette's theorem was generalized by Dougall \cite{Dougall:2017va} for convex-cocompact  groups of isometries of  pinched, negatively curved  Cartan-Hadamard manifolds.
Our statement is an even further generalization which unifies the combinatorial and the geometric point of views.
Recall that, for a Gromov hyperbolic space, the \emph{visual dimension} $\dim_{vis} (\partial X) $  is defined analogously to the Hausdorff dimension, but with respect to the natural visual measures of $\partial X$, and it coincides with the exponential growth rate of any cocompact group $G$ of isometries of $X$ -- see for instance \cite{Paulin:1996jl,Paulin:1997jp}.
We then have:

\begin{coro} 
Let $X$ be a CAT($-1$) metric space and $G$ a uniform lattice in the isometry group of $X$ with Kazhdan's property (T).
There exists  $\eta>0$ such that for any subgroup $H$ of $G$, either $H$ is itself a lattice, or  the exponential growth rate  of $H$ is at most $\dim_{vis} (\partial X) - \eta$.
\end{coro}

\noindent 
A similar statement holds if $X$ is the Cayley graph of a hyperbolic group.
This result shall be added to the list of geometric and dynamical rigidity consequences of property (T), such as Serre's fixed point-edge property for actions on trees \cite{Margulis:1981uu,Watatani:1982uj}, or the local $C^{\infty}$-conjugacy rigidity of  isometric actions on compact Riemannian manifolds \cite{Fisher:2005jg}.
\bigskip

Let us now give an overview of the proof of Theorems~\ref{res: main theo amenability - intro} and \ref{res: main theo property T - intro}.
The main idea is to apply our amenability criterion (\autoref{res: main technical theorem - intro}) to the geodesic flow on $X$.
However the criterion requires a coding of this dynamical system, which may not exist for CAT($-1$) spaces.
To bypass this difficulty we consider a geodesic flow not on the space $X$ but rather on the Cayley graph $\Gamma$ of the group $G$ with respect to a finite generating set $S$.
More precisely, if $\mathcal G \Gamma$ stands for all bi-infinite geodesics $\gamma \colon \R \to \Gamma$, the flow $\phi_t$ acts on $\mathcal G\Gamma$ by shifting the time parameter by $t$.
The issue is that this dynamical system is rather pathological.
Generally, any two points in the boundary at infinity $\partial \Gamma$ of $\Gamma$ are joined by infinitely many orbits of the flow.

\medskip
As Gromov explained in \cite{Gro87} -- see also Coornaert and Papadopoulos \cite{Coornaert:2001ff,Coornaert:2002fh} -- one can restrict our attention to an invariant subset of $\mathcal G\Gamma$: roughly speaking, all the bi-infinite geodesics whose labels are minimal for the lexicographic order induced by some fixed,   arbitrary order on the symmetric, generating set $S$ of $G$.
Formally the system that we consider is the following.
One introduces a space $\mathfrak H_0(\Gamma)$ of horofunctions on $\Gamma$ (which generalize the Busemann functions) on which the group $G$ acts.
Any horofunction $h \in \mathfrak H_0(\Gamma)$ naturally comes with a \emph{preferred} gradient line starting at $1$, i.e. whose labelling is minimal for the fixed lexicographic order. Calling $\theta(h)$   the first letter of our preferred gradient line, 
the transformation $T \colon \mathfrak H_0(\Gamma) \to \mathfrak H_0(\Gamma)$ is defined by sending $h$ to $\theta(h)^{-1}h$.
Remarkably, the system $(\mathfrak H_0(\Gamma), T)$ is conjugated to a subshift of finite type $(\Sigma,\sigma)$.
Geometrically, the suspension of $(\Sigma,\sigma)$ should be thought of as analogue of the geodesic flow on the unit tangent bundle of compact, negatively curved manifold $M$, while the suspension of the extension $(\Sigma_\theta,\sigma_\theta)$ plays the role of the geodesic flow on the unit tangent bundle of its universal cover $\tilde M$.
Nevertheless, unlike in the Riemannian setting, this flow is neither mixing nor, a-priori, topologically transitive.
This reflects the fact that two points in the boundary at infinity of $\Gamma$ can still be joined by finitely many orbits of the flow.

\medskip
Actually, the dynamical properties of $(\mathfrak H_0(\Gamma), T)$ are very sensitive to the choice of the order on $S$.
For instance, if $G$ is the direct product of the free group with a finite group, then $\mathfrak H_0(\Gamma)$ naturally splits into several disjoint ``layers'' $\mathfrak I_0, \dots , \mathfrak I_n$, where $\mathfrak I_0$ is invariant under the action of $T$.
Moreover, depending on the choice of the order on $S$, the other layers $\mathfrak I_i$ are either invariant under the action of $T$, or mapped into $\mathfrak I_0$.
For more details, we refer the reader to \autoref{exa: horoboundary free by finite}.

\medskip
To circumvent  this difficulty, we are forced to restrict our study to an irreducible component $\mathfrak I$ of the system $(\mathfrak H_0(\Gamma), T)$.
The price to pay though, is that the extension of $(\mathfrak I,T)$ by the map $\theta \colon \mathfrak H_0(\Gamma) \to G$ may not ``visit'' the whole Cayley graph $\Gamma$ of $G$.
If it misses a large portion of $\Gamma$, then this system will be useless for counting purposes.
However, we show that there exists an irreducible component $\mathfrak I$ whose extension has the visibility property  (see \autoref{def: visibility property}). \linebreak
Our strategy to produce such an irreducible component is inspired by an idea of Constantine, Lafont and Thompson \cite{Constantine:2016tia}, and based on a construction of Gromov. Namely, in \cite{Gro87} Gromov build from $(\mathcal G\Gamma, \phi_t)$ a new flow $(\widehat{\mathcal G}\Gamma, \psi_s)$ with enhanced properties -- see also \cite{Matheus:1991gt,Cha94}.
The space $\widehat{\mathcal G}\Gamma$ is quasi-isometric to $\Gamma$, hence its boundary at infinity is homeomorphic to $\partial \Gamma$; every two points in $\partial \Gamma$ are joined by a \emph{unique} orbit of the flow $\psi_s$; there is a natural projection $\mathcal G \Gamma \to \widehat{\mathcal G}\Gamma$ which send every $\phi_t$-orbit homeomorphically onto a $\psi_s$-orbit.
It turns out that the new flow $(\widehat{\mathcal G}\Gamma, \psi_s)$ is topologically transitive.
In particular, it admits a dense orbit.
The irreducible component $\mathfrak I$ is, roughly speaking, the closure of a lift of this dense orbit.
The transitivity of $(\widehat{\mathcal G}\Gamma, \psi_s)$ tells us that the the extension of $(\mathfrak I, T)$ passes uniformly near  every point of $\Gamma$, hence ensuring the visibility property.

\medskip
In order to apply our criterion (\autoref{res: main technical theorem - intro}) to the system $(\mathfrak I, T)$ we finally need to define a potential $F \colon \mathfrak I \to \R_+^*$ with bounded Hölder variations.
Recall that the dynamical system $(\mathfrak I,T)$ was not build directly from the metric space $X$ we are interested in: thus, 
the role of the potential is to reflect the geometry of $X$.
If $X$ coincides with the Cayley graph of $\Gamma$ (which corresponds to the first case of \autoref{res: main theo amenability - intro}), we simply take for $F$ the constant map equal  to $1$.
In this situation we prove that the spectral radius of the corresponding Ruelle transfer operator $\mathcal L$ is $\rho = e^{\omega_G}$, whereas the one of the twisted transfer operator satisfies $\rho_\lambda \geq e^{\omega_H}$.
Hence the conclusion of \autoref{res: main theo amenability - intro} directly follows from our amenability criterion.
If $X$ is a CAT($-1$) space (which corresponds to the second case of \autoref{res: main theo amenability - intro}) we use a quasi-isometry between $\Gamma$ and $X$ to define a potential $F$ describing the geometry of $X$.
In this situation the CAT($-1$) geometry is crucial to ensure that $F$ has bounded Hölder variations.
Once this is done, we provide as before  estimates of the spectral radii $\rho$ and $\rho_\lambda$ in terms of $\omega_G$ and $\omega_H$ and conclude by the amenability criterion.

\paragraph{Acknowledgment.}
The authors would like to thank the organizers of the \emph{Ventotene International Workshops} (2015) where this work was initiated. 
The first author acknowledges support from the Agence Nationale de la Recherche under Grant \emph{Dagger} ANR-16-CE40-0006-01.
The first and second authors are grateful to the \emph{Centre Henri Lebesgue} ANR-11-LABX-0020-01 for creating an attractive mathematical environment.
Many thanks also go to the referees for their helpful comments and pertinent corrections.

%
\section{Preliminaries}
%

%
\subsection{Subshift of finite type}
%
\label{sec: subshift finite type}

\paragraph{Vocabulary.}
Let $\mathcal A$ be a finite set.
We write $\mathcal A^*$ for the set of all finite words over the alphabet $\mathcal A$.
The length of a word $w \in \mathcal A^*$ is denoted by $\abs w$.
Given $n \in \N$, the set of all words of length $n$ is denoted by $\mathcal A^n$.
The set $\mathcal A^\N$ is is endowed with a distance $d$ defined as follows: given $x,y \in \Sigma$, we let $\dist xy = e^{-n}$ where $n$ is the length of the longest common prefix of $x$ and $y$.\label{def: distance shift}
Let $\sigma \colon \mathcal A^\N \to \mathcal A^\N$ be the shift operator, i.e. the map sending the sequence $(x_i)_{i \in \N}$ to $(x_{i+1})_{i \in \N}$.
Let $\Sigma$ be a \emph{subshift} i.e. a closed $\sigma$-invariant subset $\mathcal A^\N$.
A word $w \in \mathcal A^*$ is \emph{admissible} if it is the prefix of a sequence $x = (x_i)$ in $\Sigma$.
We denote by $\mathcal W$ the set of all finite admissible words.
For every $n \in \N$, we write $\mathcal W^n = \mathcal W \cap \mathcal A^n$ for the set of admissible words of length $n$.
Given $w \in \mathcal W$, the \emph{cylinder of $w$}, denoted by $[w]$, is the set of sequences $x \in \Sigma$ such that the prefix of length $\abs w$ of $x$ is exactly $w$.
We refer to $\abs w$ as the \emph{length} of the cylinder.
From now on we assume that $\Sigma$ is a \emph{subshift of finite type}, i.e. there exists $N \in \N$ with the following property:
a sequence $x \in \mathcal A^\N$ belongs to $\Sigma$ if and only if every subword of length $N$ of $\Sigma$ belongs to $\mathcal W^N$.
We say that $(\Sigma,\sigma)$ is \emph{irreducible} or \emph{topologically transitive} if $(\Sigma, \sigma)$ admits a dense orbit.
As $(\Sigma, \sigma)$ is a subshift of finite type it is equivalent to ask that one of the following holds.
\begin{enumerate}
	\item For every $x,y \in \Sigma$, for every $\epsilon \in \R_+^*$, there exists $k \in \N$ and $x' \in \Sigma$ such that $d(x,x') \leq \epsilon$ and $\sigma^kx' = y$
	\item For every $w, w' \in \mathcal W$, there exists $w_0 \in \mathcal W$ such that the concatenation $ww_0w'$ is admissible
\end{enumerate}

\paragraph{Irreducible components.}
We associate to $(\Sigma,\sigma)$ an oriented graph $\Gamma = (V,E)$ labeled by $\mathcal A$ describing the dynamics of the shift.
We still denote by $N$ the integer given by the definition of a subshift of finite type.
The vertex set $V$ of $\Gamma$ is simply $\mathcal W^N$.
Given two words $w_1,w_2 \in \mathcal W^N$ and a letter $a \in \mathcal A$, there is an edge from $w_1$ to $w_2$ labelled by $a$ if $w_1$ is the prefix of length $N$ of $aw_2$.
It follows from the definition of $N$, that the labelling of $\Gamma$ induces a one-to-one correspondance between $\Sigma$ and the set of infinite oriented paths in $\Gamma$.
We now define an equivalence relation on $\mathcal W^N$ seen as the vertex set of $\Gamma$.
Two vertices $w,w' \in \mathcal W^N$ are \emph{communicating} and we write $w \sim w'$ if there exists an oriented loop in $\Gamma$ passing through $w$ and $w'$.
The corresponding equivalence classes $V_1, \dots, V_m$ are called \emph{communicating classes}.

\medskip
Given $i \in \intvald 1m$, we write $\Gamma_i$ for the full subgraph of $\Gamma$ associated to $V_i$.
The subspace $\Sigma_i \subset \Sigma$ is defined as the set of all sequences of $\mathcal A^\N$ labeling an infinite path in $\Gamma_i$.
Observe that $\Sigma_i$ is a closed $\sigma$-invariant subset of $\Sigma$ and $(\Sigma_i, \sigma)$ is a subshift of finite type.
It follows from the construction that $(\Sigma_i, \sigma)$ is irreducible.

\medskip
We observe that for every sequence $x \in \Sigma$, there exists $n_0 \in \N$, and a unique $i \in \intvald 0m$ such that for all integers $n \geq n_0$, the sequence $\sigma^nx$ belongs to $\Sigma_i$.
Indeed one can derive from the relation $\sim$ a new graph $\Gamma/\!\!\sim$ defined as follows.
Its vertex set is the set of communicating classes $V_1, \dots, V_m$.
There is an edge in $\Gamma/\!\!\sim$ from $V_i$ to $V_j$, if $\Gamma$ contains an edge joining a vertex in $V_i$ to a vertex in $V_j$.
The graph $\Gamma/\!\!\sim$ comes with a natural projection $\Gamma \to \Gamma/\!\!\sim$.
It follows from the definition of $\sim$ that $\Gamma/\sim$ does not contain any oriented loop.
Hence if $c$ is an infinite path in $\Gamma$, its projection in $\Gamma/\!\!\sim$ is ultimately constant.
This means that there exists $i \in \intvald 1m$ such that the path obtained from $c$ by removing its first edges is contained in $\Gamma_i$.
Hence if $x$ is the sequence labelling $c$, there exists $n_0 \in \N$, such that for all integers $n \geq n_0$, the sequence $\sigma^nx$ belongs to $\Sigma_i$.
We say that $\Sigma_i$ is the \emph{asymptotic irreducible component of $x$}.

\paragraph{Group extension.}
Let $G$ be a finitely generated group.
Let $\theta \colon \Sigma \to G$ be a continuous map.
We denote by $\Sigma_\theta$ the direct product $\Sigma_\theta = \Sigma \times G$ endowed with the product topology.
We define a continuous map $\sigma_\theta \colon \Sigma_\theta \to \Sigma_\theta$ by 
\begin{displaymath}
	\sigma_\theta (x,g) = \left(\sigma x ,g \theta(x)\right), \quad \forall (x,g) \in \Sigma_\theta.
\end{displaymath}
The dynamical system $(\Sigma_\theta, \sigma_\theta)$ is called the \emph{extension of $(X,\sigma)$ by $\theta$}.
For every $n \in \N$, for every $x \in \Sigma$ we define the cocycle $\theta_n(x)$ by 
\begin{equation}
\label{eqn: def cocycle}
	\theta_n(x) = \theta(x) \theta(\sigma x) \cdots \theta\left(\sigma^{n-1}x\right).
\end{equation}
By convention $\theta_0 \colon \Sigma \to G$ is the constant map equal to $1$ (the identity of $G$).
Hence we have
\begin{displaymath}
	\sigma_\theta^n(x,g) = \left(\sigma^nx,g\theta_n(x)\right).
\end{displaymath}

\begin{defi}
\label{def: visibility property}
	We say that the extension $(\Sigma_\theta, \sigma_\theta)$ has the \emph{visibility property} if there exists a finite subset $U$ of $G$ with the following property: for every $g \in G$, there is a point $x \in \Sigma$, an integer $n \in \N$ and two elements $u_1,u_2 \in U$ such that $g = u_1 \theta_n(x)u_2$.
\end{defi}

\paragraph{Remark.}
This definition is somewhat reminiscent of Dougall's \emph{linear visibility with remainder} property \cite[Definition~3.1]{Dougall:2017va}.
Nevertheless our notion is more flexible as we do not ask to control the value of $n$ in terms of the word length of $g$.

%
\subsection{Hyperbolic geometry}
%

\paragraph{The four point inequality.}
Let $(X,d)$ be a geodesic metric space.
The Gromov product of three points $x,y,z \in X$  is defined by 
\begin{displaymath}
	\gro xyz = \frac 12 \left\{  \dist xz + \dist yz - \dist xy \right\}.
\end{displaymath}
We assume that the space $X$ is \emph{$\delta$-hyperbolic}, i.e. for every $x,y,z,t \in X$,
\begin{equation}
\label{eqn: hyperbolicity condition 1}
	\gro xzt \geq \min\left\{ \gro xyt, \gro yzt \right\} - \delta,
\end{equation}

\paragraph{The boundary at infinity.} 
Let $x_0$ be a base point of $X$.
A sequence $(y_n)$ of points of $X$ \emph{converges to infinity} if $\gro {y_n}{y_m}{x_0}$ tends to infinity as $n$ and $m$ approach to infinity.
The set $\mathcal S$ of such sequences is endowed with a binary relation defined as follows.
Two sequences $(y_n)$ and $(z_n)$ are related if 
\begin{displaymath}
	\lim_{n \rightarrow + \infty} \gro {y_n}{z_n}{x_0} = + \infty.
\end{displaymath}
If follows from (\ref{eqn: hyperbolicity condition 1}) that this relation is actually an equivalence relation.
The \emph{boundary at infinity} of $X$ denoted by $\partial X$ is the quotient of $\mathcal S$ by this relation.
If the sequence $(y_n)$ is an element in the class of $\xi \in \partial X$ we say that $(y_n)$  \emph{converges} to $\xi$ and  write
\begin{displaymath}
	\lim_{n \rightarrow + \infty} y_n = \xi.
\end{displaymath}
Note that the definition of $\partial X$ does not depend on the base point $x_0$.
For more details about the Gromov boundary, we refer the reader to \cite[Chapitre 2]{CooDelPap90}.
The notation $\partial^2X$ stands for 
\begin{displaymath}
	\partial^2X = \partial X \times \partial X \setminus \operatorname{diag}(\partial X),
\end{displaymath}
where $\operatorname{diag}(\partial X)=\set{(\xi, \xi)}{\xi \in \partial X}$ is the diagonal.

\paragraph{Quasi-geodesics.}
One major feature of hyperbolic spaces is the stability of quasi-geodesics also known as Morse's Lemma.

\begin{defi}
\label{def: qi}
	Let $\kappa \in \R_+^*$ and $\ell \in \R_+$.
	Let $f \colon X_1 \to X_2$ be a map between two metric spaces.
	We say that $f$ is a \emph{$(\kappa, \ell)$-quasi-isometric embedding}, if for every $x,x' \in X_1$ we have 
	\begin{displaymath}
		\kappa^{-1}\dist x{x'} - \ell \leq \dist{f(x)}{f(x')} \leq \kappa \dist x{x'} + \ell.
	\end{displaymath}
	A \emph{$(\kappa, \ell)$-quasi-geodesic} of $X$, is a $(\kappa, \ell)$-quasi-isometric embedding of an interval of $\R$ into $X$.
\end{defi}

Given a $(\kappa,\ell)$-quasi-geodesic $\gamma \colon \R_+ \to X$, there exists a point $\xi \in \partial X$ such that for every sequence $(t_n)$ diverging to infinity, we have
\begin{displaymath}
	\lim_{n \rightarrow + \infty}\gamma(t_n) = \xi,
\end{displaymath}
see for instance \cite[Chapitre~3, Théorème~2.2]{CooDelPap90}.
In this situation we consider $\xi$ as the endpoint at infinity of $\gamma$ and write $\gamma(\infty) = \xi$.
If $\gamma \colon \R \to X$ is a bi-infinite $(\kappa,\ell)$-quasi-geodesic, we define $\gamma(-\infty)$ in the same way.

\begin{prop}[{\cite[Chapitre~3, Théorèmes~1.2 et 3.1]{CooDelPap90}}]
\label{res: morse lemma}
	Let $\kappa \in \R_+^*$ and $\ell \in \R_+$.
	There exists $D = D(\kappa, \ell, \delta)$ in $\R_+$ such that the Hausdorff distance between any two $(\kappa, \ell)$-quasi-geodesics with the same endpoints (possibly in $\partial X$) is at most $D$.
\end{prop}

\paragraph{Group action.}
Let $x_0$ be a base point of $X$.
Let $G$ be a group acting properly by isometries on $X$.
The \emph{exponential growth rate of $G$ acting on $X$} is the quantity
\begin{equation}
\label{eqn: def exp growth rate}
	\omega(G,X) = \limsup_{r \to \infty} \frac 1r \ln \card{\set{g \in G}{\dist {gx_0}{x_0}\leq r}}.
\end{equation}
Note that $\omega(G,X)$ does not depend on $x_0$.
It can also be interpreted as the critical exponent of the Poincaré series
\begin{displaymath}
	\mathcal P_G(s) = \sum_{g \in G} e^{-s\dist{gx_0}{x_0}}.
\end{displaymath}
If there is no ambiguity, we simply write $\omega_G$ instead of $\omega(G,X)$.

%
\section{Horofunctions}
%

In this section we recall the definition of horofunctions and gradient lines introduced by Gromov in \cite[Section~7.5]{Gro87}.
We follow the exposition given by Coornaert and Papadopoulos in \cite{Coornaert:1993jz,Coornaert:2001ff}.
Let $(X,x_0)$ be a pointed geodesic $\delta$-hyperbolic space and $G$ be a group acting by isometries on $X$.

\paragraph{Horofunctions.}
Let $\epsilon \in \R_+$.
A map $f \colon X \to \R$ is \emph{$\epsilon$-quasi-convex} if for every geodesic $\gamma \colon I \to X$, for every $a,b \in I$, for every $t \in \intval 01$ we have
\begin{displaymath}
	f\circ \gamma \left(ta + (1-t)b \right) \leq tf\circ \gamma(a) + (1-t)f\circ \gamma(b) + \epsilon.
\end{displaymath}
The map $f$ is \emph{quasi-convex} if there exists $\epsilon  \in \R_+$ such that $f$ is $\epsilon$-quasi-convex.

\begin{defi}
\label{def: horofunction}
	A \emph{horofunction} is a quasi-convex map $h \colon X \to \R$ satisfying the following distance-like property:
	for every $x \in X$, for every $r \in \R$, if $r \leq h(x)$, then
\begin{equation}
\label{eqn: distance like condition}
	h(x) = r + \dist{x}{h^{-1}(r)}.
\end{equation}
	A \emph{cocycle} is a map $c \colon X \times X \to \R$ of the form $c(x,y) = h(x) - h(y)$ where $h$ is a horofunction.
	In this situation, $h$ is called a \emph{primitive} of $c$.
\end{defi}

\medskip
Let $C(X)$ be the set of all continuous function $f \colon X \to \R$ endowed with the topology of uniform convergence on compact subsets.
We denote by $C_*(X)$ the quotient of $C(X)$ by the $1$-dimensional closed subspace of constant functions.
The space $C_*(X)$ is endowed with the quotient topology.
It can be identified with the set of all continuous functions $f \colon X \to \R$ vanishing at $x_0$.
This is the point of view that we adopt here.
The action of $G$ on $X$ induces an action on $C(X)$ and thus on $C_*(X)$.
More precisely $g \in G$, and $f \in C_*(X)$ the map $g \cdot f$ is defined by 
\begin{displaymath}
	\left[g\cdot f\right](x) = f(g^{-1}x) - f(g^{-1}x_0),\quad \forall x \in X.
\end{displaymath}
A horofunction is automatically $1$-Lipschitz, hence continuous \cite[Proposition~2.2]{Coornaert:2001ff}.
Moreover, given a cocycle $c$, any two primitives $h_1$ and $h_2$ of $c$ differ by a constant, i.e. there exists $a \in \R$, such that for every $x \in X$, we have $h_2(x) - h_1(x) = a$.
Hence the set of all cocycles, or equivalently all horofunctions vanishing at $x_0$, embeds in $C_*(X)$.
We denote it by $\mathfrak H(X)$.
It is a compact subspace of $C_*(X)$ \cite[Proposition~3.9]{Coornaert:2001ff}.

\paragraph{Gradient lines.}
The gradient lines are an important tool to track the behavior of horofunctions.

\begin{defi}
	Let $h \in \mathfrak H (X)$ be a horofunction.
	A \emph{gradient line for $h$} or an \emph{$h$-gradient line} is a path $\gamma \colon I \to X$ parametrized by arc length such that for every $t,t' \in I$ we have 
	\begin{displaymath}
		h(\gamma(t)) - h(\gamma(t')) = t' - t.
	\end{displaymath}
	If the interval $I$ has the form $I = \R_+$ we say that $\gamma$ a \emph{$h$-gradient ray}.
\end{defi}

Let us recall a few properties of gradient lines.
Let $h \in \mathfrak H (X)$.
Every $h$-gradient line is a geodesic \cite[Proposition~2.10]{Coornaert:2001ff}.
If $\gamma \colon I \to X$ be an $h$-gradient line, then for every $g \in G$, the path $g\gamma$ is a gradient line for $gh$ \cite[Proposition~2.14]{Coornaert:2001ff}.
For every $x \in X$, there exists an $h$-gradient ray starting at $x$ \cite[Proposition~2.13]{Coornaert:2001ff}.

\medskip
Let $\gamma \colon \R_+ \to X$ be a geodesic ray.
The \emph{Busemann function} associated to $\gamma$ is the map $b_\gamma \colon X \to \R$ defined by 
\begin{displaymath}
	b_\gamma(x)  
	= \lim_{t \to \infty} \left[\fantomB\;\dist x{\gamma(t)} - t\;\right], \quad \forall x \in X.
\end{displaymath}
The map $c_\gamma \colon X\times X \to \R$ defined by $c_\gamma(x,y) = b_\gamma(x) - b_\gamma(y)$ is a cocycle \cite[Chapter~3, Proposition~3.6]{Coornaert:1993jz}.
Moreover $\gamma$ is a gradient ray starting at $\gamma(0)$ for this cocycle.
However, one has to be careful that $c_\gamma$ may admit other gradient lines starting at $\gamma(0)$.
Similarly, given a horofunction $h \in \mathfrak H(X)$, the Busemann function associated to an $h$-gradient ray starting at $x_0$ is not necessarily $h$.

\paragraph{Comparison with the Gromov boundary.}
Recall that $\partial X$ is the Gromov boundary of $X$.

\begin{prop}[Coornaert-Papadopoulos {\cite[Proposition~3.1]{Coornaert:2001ff}}]
\label{res: endpoints of gradient line}
	Let $h \in \mathfrak H (X)$ be a horofunction.
	Let $\gamma_1 \colon \R_+ \to X$ and $\gamma_2 \colon \R_+ \to X$ be two $h$-gradient rays.
	Then $\gamma_1(\infty) = \gamma_2(\infty)$.
\end{prop}

This gives rise to a map $\pi \colon \mathfrak H(X) \to \partial X$ which associates to any $h \in \mathfrak H (X)$ the endpoint at infinity of any $h$-gradient ray.

\begin{prop}[Coornaert-Papadopoulos {\cite[Proposition~3.3 and Corollary~3.8]{Coornaert:2001ff}}]
\label{res: projection horofunction to boundary}
	The map $\pi \colon \mathfrak H(X)\to \partial X$ is continuous, $G$-equivariant and onto.
	More precisely for every geodesic ray $\gamma \colon \R_+ \to X$ starting at $x_0$ the corresponding Busemann function $b_\gamma$ is a preimage of $\gamma(\infty)$ in $\mathfrak H (X)$.
	Two horofunctions $h, h' \in \mathfrak H(X)$ have the same image in $\partial X$ if and only if $\norm[\infty]{h - h'} \leq 64 \delta$.
\end{prop}

%
\section{Dynamics in a hyperbolic group}
%

In this section we introduce a few dynamical systems to describe the ``geodesic flow'' of a hyperbolic Cayley graph.
Let $G$ be a hyperbolic group and $A$ a finite generating set of $G$.
For simplicity we assume that $A$ is symmetric, i.e. $A^{-1} = A$.
We denote by $\Gamma(G,A)$ or simply $\Gamma$ the Cayley graph of $G$ with respect to $A$.
We identify $G$ with the vertex set of $\Gamma$.
We consider $1$ as a base point in $\Gamma$.
For every $n \in \N$, we write $S(n)$ for the sphere of radius $n$ in $\Gamma$ centered at the identity, i.e. the set of all elements $g \in G$ such that $\dist[\Gamma] 1g = n$.
Similarly we write $B(n)$ for the closed ball of radius $n$.

%
\subsection{Transitivity of Gromov's geodesic flow}
%

We denote by $\mathcal G\Gamma$ the set of all \emph{parametrized} bi-infinite geodesic $\gamma \colon \R \to \Gamma$ of $\Gamma$.
This set is endowed with a distance defined as follows: given two bi-infinite geodesics $\gamma_1, \gamma_2 \colon \R \to \Gamma$ we let
\begin{equation}
\label{eqn: def distance geo space}
	\dist{\gamma_1}{\gamma_2} = \int_{-\infty}^{\infty} e^{-\abs t}\dist {\gamma_1(t)}{\gamma_2(t)}dt.
\end{equation}
The action of $G$ on $\Gamma$ induces an action by isometries of $G$ on $\mathcal G\Gamma$.
One checks easily that the map $\mathcal G\Gamma \to \Gamma$ sending $\gamma$ to $\gamma(0)$ is a $G$-equivariant quasi-isometry.
The space $\mathcal G\Gamma$ also comes with a flow $\phi = (\phi_s)_{s \in \R}$ defined as follows: for every $\gamma \in \mathcal G\Gamma$, for every $s \in \R$, the geodesic $\phi_s(\gamma)\colon \R \to \Gamma$ is given by 
\begin{displaymath}
	\phi_s(\gamma) (t) = \gamma(s+t), \quad \forall t \in \R.
\end{displaymath}
Starting from $\mathcal G\Gamma$, Gromov build a new hyperbolic space $\widehat{\mathcal G}\Gamma$ that is quasi-isometric to $\Gamma$.
In particular it is hyperbolic and its boundary at infinity is homeomorphic to $\partial \Gamma$.
Moreover $\widehat{\mathcal G}\Gamma$ comes with a flow so that any two distinct points of $\partial \Gamma$ are joined by a \emph{unique} orbit of the flow.
The construction is given in \cite[Section~8.3]{Gro87}, the details can be found in \cite{Cha94,Matheus:1991gt}.
We recall here the main properties of this space.
\begin{labelledenu}[F]
	\item \label{enu: geoflow - space action and flow}
	The space $\widehat{\mathcal G}\Gamma$ is geodesic and proper.
	It is endowed with a proper co-compact action by isometries of $G$ as well as a flow $\psi = (\psi_s)_{s \in \R}$.
	The flow and the action of $G$ commute, i.e. for every $\widehat \gamma \in \widehat{\mathcal G}\Gamma$, for every $s \in \R$ and $g \in G$ we have $\psi_s(g\widehat \gamma) = g \psi_s(\widehat \gamma)$.
	\item \label{enu: geoflow - projection}
	There exists a continuous $G$-equivariant quasi-isometric projection $p \colon \mathcal G\Gamma \twoheadrightarrow \widehat{\mathcal G}\Gamma$.
	In particular $p$ induces a homeomorphism $p_\infty$ from $\partial \Gamma$ onto the boundary at infinity of $\widehat{\mathcal G}\Gamma$.
	In addition, for every geodesic $\gamma \in \mathcal G\Gamma$, the projection $p$ maps the $\phi$-orbit of $\gamma \in \mathcal G\Gamma$ homeomorphically onto the $\psi$-orbit of $\widehat \gamma = p(\gamma)$.
	\item \label{enu: geoflow - point at infinity}
	For every point $\widehat \gamma \in \widehat{\mathcal G}\Gamma$, the map $\R \to \widehat{\mathcal G}\Gamma$ sending $s$ to $\psi_s(\widehat \gamma)$ is a quasi-isometric embedding of $\R$ into $\widehat{\mathcal G}\Gamma$.
	Hence for every point $\widehat \gamma \in \widehat{\mathcal G}\Gamma$ one can associate two distinct points in the boundary at infinity of $\widehat{\mathcal G}\Gamma$ defined by 
	\begin{displaymath}
		\widehat\gamma(\infty) = \lim_{s \to \infty}\psi_s(\widehat \gamma)
		\quad \text{and}\quad
		\widehat\gamma(-\infty) = \lim_{s \to -\infty}\psi_s(\widehat \gamma).
	\end{displaymath}
	By construction for every geodesic $\gamma \in \mathcal G \Gamma$, the homeomorphism $p_\infty$ maps $\gamma(\infty)$ and $\gamma(-\infty)$ to $\widehat\gamma(\infty)$ and $\widehat\gamma(-\infty)$, where $\widehat \gamma = p(\gamma)$.
	\item \label{enu: geoflow - decomposition}
	The map $\widehat{\mathcal G}\Gamma \to \partial^2\Gamma$ sending $\widehat \gamma$ to $(\widehat \gamma(-\infty), \widehat \gamma(\infty))$ induces a homeomorphism from $\widehat{\mathcal G}\Gamma/\R$ onto $\partial^2\Gamma$.
	Actually $\widehat{\mathcal G}\Gamma$ is homeomorphic to $\partial^2\Gamma \times \R$.
\end{labelledenu}

It is important to notice that in general $p$ does not conjugate the flow, i.e. $p \circ \phi_s \neq \psi_s \circ p$.
Since the flow $\psi$ and the action of $G$ commute, the flow $\psi$ induces a flow on $\widehat{\mathcal G}\Gamma/G$ that we denote $\overline \psi = (\overline \psi_s)_{s \in \R}$.

\begin{prop}
\label{res: transitivity Gromov geodesic flow}
	The flow $\overline\psi$ on $\widehat{\mathcal G}\Gamma/G$ is topologically transitive, i.e. given any two non-empty open subsets $U$ and $V$ of $\widehat{\mathcal G}\Gamma/G$, there exists $s \in \R$ such that $\overline\psi_s(U) \cap V$ is non-empty.
\end{prop}

The remainder of the section is dedicated to the proof of this proposition.
We follow mostly the strategy used in \cite[Chapter~III]{Ballmann:1995kf}.

\begin{lemm}
\label{res: transitivity Gromov geodesic flow - prelim}
	Let $\gamma, \gamma' \in \mathcal G\Gamma$ such that $\gamma(\infty) \neq \gamma'(-\infty)$.
	There are sequences $(\nu_n)$ of geodesics in $\mathcal G\Gamma$, $(g_n)$ of elements in $G$ and $(t_n)$ of numbers in $\R$ diverging to infinity with the following properties.
	\begin{enumerate}
		\item $(\nu_n)$ converges to a geodesic with the same endpoints at $\gamma$.
		\item $(g_n\phi_{t_n}(\nu_n))$ converges to a geodesic with the same endpoints as $\gamma'$.
	\end{enumerate}
\end{lemm}

\begin{proof}
	For simplicity we write $\xi_+$ and $\xi_-$ for $\gamma(\infty)$ and $\gamma(-\infty)$.
	Similarly we define $\xi'_+$ and $\xi'_-$.
	Since $\xi_+ \neq \xi'_-$, there exists a sequence $(g_n)$ of elements of $G$ such that for some (hence any) $x \in \Gamma$ the sequence $(g_nx)$ and $(g_n^{-1}x)$ respectively converge to $\xi'_-$ and $\xi_+$ -- see for instance \cite[Chapter~III, Lemma~2.2.]{DalBo:2011di}.
	Nevertheless, the action of $G$ on $\Gamma \cup \partial \Gamma$ is a convergence action.
	Up to replacing $(g_n)$ by a subsequence, we can assume that for every for every $\xi\in\partial \Gamma \setminus\{\xi_+\}$, the sequence $(g_n\xi)$ converges to $\xi'_-$ and for every $\xi\in\partial \Gamma \setminus\{\xi'_-\}$, the sequence $(g_n^{-1}\xi)$ converges to $\xi_+$ \cite[Theorem~3A]{Tukia:1994uk}.
	For every $n \in \N$, we denote by $\nu_n \colon \R \to \Gamma$ a bi-infinite geodesic joining $\xi_-$ to $g_n^{-1}\xi'_+$.
	We observe that 
	\begin{itemize}
		\item $\nu_n(-\infty)$ and $\nu_n(\infty)$ respectively converge to $\xi_-$ and $\xi_+$, whereas
		\item $g_n\nu_n(-\infty)$ and $g_n\nu_n(\infty)$ respectively converge to $\xi'_-$ and $\xi'_+$.
	\end{itemize}
	Geodesic triangles in $\Gamma \cup \partial \Gamma$ are $24\delta$-thin \cite[Chapitre~2 Proposition~2.2]{CooDelPap90}.
	Up to passing again to a subsequence we may assume that $\dist{\gamma(0)}{\nu_n} \leq 24 \delta$ and $\dist{\gamma'(0)}{g_n \nu_n} \leq 24 \delta$.
	By shifting if necessary the origin of $\nu_n$ we can assume that for every $n \in \N$, there exists $t_n \in \R$, such that  $\dist{\gamma(0)}{\nu_n(0)} \leq 24\delta$ and $\dist{\gamma'(0)}{g_n\nu_n(t_n)} \leq 24 \delta$.
	According to the Azelà-Ascoli theorem $(\nu_n)$ converges to a geodesic $\nu$.
	It follows from our choice of $(g_n)$ that $\nu$ has the same endpoints as $\gamma$.
	Similarly we obtain that $(g_n\phi_{t_n}(\nu_n))$ converges to a geodesic with the same endpoints as $\gamma'$.
	Note also that $\dist{\nu_n(t_n)}{g_n^{-1}\gamma'(0)}$ and $\dist{\nu_n(0)}{\gamma(0)}$ are uniformly bounded.
	As $(g_n^{-1}\gamma'(0))$ converges to $\xi_+$, the sequence  $(t_n)$ has to diverge to $\infty$.
\end{proof}

\begin{proof}[Proof of \autoref{res: transitivity Gromov geodesic flow}]
	It suffices to show that for every non-empty open subset $\widehat U$ and $\widehat V$ of $\widehat{\mathcal G}\Gamma$, there exists $s \in \R$, and $g \in G$ such that $g\psi_s(\widehat U)\cap\widehat V \neq \emptyset$.
	Let $\widehat U$ and $\widehat V$ be two non-empty open subsets of $\widehat{\mathcal G}\Gamma$.
	We denote by $\widehat U(\infty)$ the following subset of $\partial \Gamma$
	\begin{displaymath}
		\widehat U(\infty) = \set{\widehat \gamma(\infty)}{\widehat \gamma \in \widehat U}.
	\end{displaymath}
	The set $\widehat V(-\infty)$ is defined in a similar way.
	It follows from \ref{enu: geoflow - decomposition} that $\widehat U(\infty)$ and $\widehat V(-\infty)$ are open non-empty subsets of $\partial \Gamma$.
	In particular they are not reduced to a point.
	Hence there exists $\widehat \gamma \in \widehat U$ and $\widehat \gamma' \in \widehat V$ such that $\widehat \gamma(\infty) \neq \widehat\gamma'(-\infty)$.
	Let $\gamma, \gamma' \in \mathcal G\Gamma$ be respective preimages of $\widehat \gamma$ and $\widehat \gamma'$.
	According to \ref{enu: geoflow - point at infinity}, $\gamma(\infty) \neq \gamma'(- \infty)$.
	Applying \autoref{res: transitivity Gromov geodesic flow - prelim}, there  are sequences $(\nu_n)$ of geodesics in $\mathcal G\Gamma$, $(g_n)$ of elements in $G$ and $(t_n)$ of numbers in $\R$ diverging to infinity with the following properties: the geodesic $\nu_n$ converges to a geodesic $\nu$ with the same endpoints as $\gamma$; 
	the geodesic $\nu'_n= g_n\phi_{t_n}(\nu_n)$ converges to a geodesic $\nu'$ with the same endpoints as $\gamma'$.
	
	\medskip
	We now push these data in $\widehat{\mathcal G}\Gamma$ using the projection $p \colon \mathcal G \Gamma \twoheadrightarrow \widehat{\mathcal G}\Gamma$.
	For every $n \in \N$, we let $\widehat \nu_n = p(\nu_n)$ and $\widehat \nu'_n = p(\nu'_n)$.
	We define $\widehat \nu$ and $\widehat\nu'$ in the same way.
	Since $p$ maps homeomorphically $\phi$-orbits onto $\psi$-orbits, there exists a sequence $(s_n)$ of numbers in $\R$, diverging to $\infty$ such that for every $n \in \N$, we have
	\begin{displaymath}
		p \left(\phi_{t_n}(\nu_n)\right) = \psi_{s_n}\left(\widehat\nu_n\right).
	\end{displaymath}
	By construction the endpoints of $\widehat \nu$ are the same as those of $\widehat \gamma$.
	However there is a unique orbit of the flow $\psi$ joining two distinct point of $\partial \Gamma$ -- see \ref{enu: geoflow - decomposition}.
	Thus there exists $s \in \R$ such that $\widehat \gamma = \psi_s(\widehat \nu)$.
	Recall that $\nu_n$ converges to $\nu$.
	Since the projection $p$ is continuous, $\widehat \nu_n$ converges to $\widehat \nu$, hence $\psi_s(\widehat \nu_n)$ converges to $\widehat \gamma$.
	Similarly we prove that there exists $s' \in \R$ such that $\psi_{s'}(\widehat \nu'_n) = g_n \psi_{s_n + s'}(\widehat \nu_n)$ converges to $\widehat \gamma'$.
	As $\widehat U$ and $\widehat V$ are open, there exists $n_0 \in \N$ such that for every $n \geq n_0$, the point $\psi_s(\widehat \nu_n)$ belongs to $\widehat U$ while $g_n \psi_{s_n + s'}(\widehat \nu_n)$ belongs to $\widehat V$.
	Hence 
	\begin{displaymath}
		g_n \psi_{s_n + s' - s}(\widehat U) \cap \widehat V \neq \emptyset, \quad \forall n \geq n_0,
	\end{displaymath}
	which completes the proof of the proposition.
\end{proof}

\begin{coro}[Compare Dal'bo {\cite[Chapter~III, Theorem~4.2]{DalBo:2011di}}]
\label{res: dense orbit Gromov geodesic flow}
	There exists a point $\widehat \gamma \in \widehat{\mathcal G}\Gamma$ such that the image of $\set{\psi_s(\widehat\gamma)}{s \in \R_+}$ in $\widehat{\mathcal G}\Gamma/G$ is dense.
\end{coro}

%
\subsection{Dynamics on the space of horofunctions}
%
\label{sec: dyn on horoboundary}

\begin{defi}
\label{res: integral horofunction}
	A horofunction $h \in \mathfrak H (\Gamma)$ is \emph{integral} if its restriction to the vertex set $G$ takes integer values.
	We write $\mathfrak H_0(\Gamma)$ (or simply $\mathfrak H_0$) for the subset of $\mathfrak H(\Gamma)$ consisting of all integral horofunctions.
\end{defi}

Note that $\mathfrak H_0$ is closed, $G$-invariant subset of $\mathfrak H(\Gamma)$.
Moreover, the Busemann function of any geodesic ray starting at $1$ is an integral horofunction.
Consequently \autoref{res: projection horofunction to boundary} leads to the following statement
\begin{prop}
\label{res: projection integral horofunction to boundary}
	The projection $\pi \colon \mathfrak H(\Gamma) \to \partial \Gamma$ maps $\mathfrak H_0$ onto $\partial \Gamma$.
\end{prop}

We now define a map $T \colon \mathfrak H_0 \to \mathfrak H_0$ which can be thought as an analogue of the first return map for a section of the geodesic flow.
To that end we endow the generating set $A$ with an arbitrary total order.
The map $\theta \colon \mathfrak H_0 \to G$ is the one sending $h$ to the smallest element $a \in A$ such that $h(a) = h(1) - 1$.
The existence of such a generator $a \in A$ is ensured by \cite[Lemma~5.1]{Coornaert:2001ff}.
\begin{defi}
\label{def: shift horoboundary}	
	The map $T \colon \mathfrak H_0 \to \mathfrak H_0$ is the one sending $h$ to $\theta(h)^{-1}h$.
\end{defi}

This map is continuous \cite[Proposition~5.6]{Coornaert:2001ff}.
Before speaking of coding, let us recall a few properties of the dynamical system $(\mathfrak H_0, T)$.
Recall that the edges of $\Gamma$ are labelled by elements of $A$.
Let $h \in \mathfrak H_0$ be a horofunction and $x$ a vertex of $\Gamma$.
The lexicographic order on $A^\N$ induces a total order on the set of $h$-gradient rays starting at $x$.
This set admits a smallest element \cite[Proposition~5.2]{Coornaert:2001ff} that we call the \emph{minimal $h$-gradient ray starting at $x$}.
Assume that $\gamma \colon \R_+ \to \Gamma$ is the minimal $h$-gradient ray starting at $x$.
We have the following properties:
\begin{enumerate}
	\item for every $n \in \N$, the path $\gamma_n \colon \R_+ \to \Gamma$ defined by $\gamma_n(t) = \gamma(t+n)$ is the minimal $h$-gradient ray starting at $\gamma(n)$ \cite[Proposition~5.3]{Coornaert:2001ff};
	\item for every $g \in G$, $g\gamma$ is the minimal $gh$-gradient ray starting at $gx$ \cite[Proposition~5.4]{Coornaert:2001ff}.
\end{enumerate}
As usual we define for every $n \in \N$, the cocycle $\theta_n \colon \mathfrak H_0 \to G$ by
\begin{equation}
\label{eqn: def cocycle horofunction}
	\theta_n(h) = \theta(h)\theta(Th) \cdots \theta(T^{n-1}h).
\end{equation}
It follows from the previous two observations, that this cocycle has the following geometric interpretation:
if $\gamma \colon \R \to \Gamma$ is the minimal $h$-gradient line starting at $1$, then $\gamma(n) = \theta_n(h)$.

\paragraph{Connection with Gromov's geodesic flow.}
We relate here $(\mathfrak H_0, T)$ to Gromov's geodesic flow on $\Gamma$.
Let $h \in \mathfrak H_0$ be an integral horofunction.
A bi-infinite $h$-gradient line $\gamma \colon \R \to \Gamma$ is called \emph{primitive} if the following holds
\begin{enumerate}
	\item $\gamma(0)$ is a vertex of $\Gamma$, and thus $\gamma(\Z)$ is contained in $G$.
	\item For every integer $m \in \Z$, the path $\gamma_m \colon \R_+ \to \Gamma$ sending $t$ to $\gamma(m+t)$ is the minimal $h$-gradient ray starting at $\gamma(m)$.
\end{enumerate}
Observe that in this case, for every $k \in \Z$, the path $\phi_k(\gamma)$ is also a primitive $h$-gradient line.
Similarly for every $g \in G$, the path $g\gamma$ is a primitive $gh$-gradient line \cite[Lemma~2.5]{Coornaert:2002fh}.
Given $h \in \mathfrak H_0$ and $\xi \in \partial \Gamma$, there is always a primitive $h$-gradient line such that $\gamma(-\infty) = \xi$ \cite[Proposition~5.2]{Coornaert:2002fh}.

\medskip
Let $\mathfrak H_\R$ be the subset of $\mathfrak H_0 \times \mathcal G\Gamma$ that consists of all pairs $(h,\gamma)$ where $h \in \mathfrak H_0$ is an integral horofunction and $\gamma \in \mathcal G \Gamma$ a geodesic in the $\phi$-orbit of a primitive $h$-gradient line \cite[Definition~4.1]{Coornaert:2002fh}.
We endow $\mathfrak H_\R$ with the topology induced by the product topology on $\mathfrak H_0 \times \mathcal G\Gamma$.
The set $\mathfrak H_\R$ comes with a flow, that we still denote $\phi = (\phi_t)_{t \in \R}$, defined as follows:
for every $(h,\gamma) \in \mathfrak H_\R$,  
\begin{displaymath}
	\phi_t(h, \gamma) = (h, \phi_t(\gamma)),\quad \forall t \in \R.
\end{displaymath}
We call the dynamical system $(\mathfrak H_\R, \phi)$ the \emph{horoflow} of $\Gamma$.
This system shall not be confused with the horocyclic flow on a hyperbolic surface.
The group $G$ acts on $\mathfrak H_\R$ as follows: for every $(h,\gamma) \in \mathfrak H_\R$ we let
\begin{displaymath}
	g(h,\gamma) = (gh, g\gamma), \quad \forall g\in G.
\end{displaymath}
Note that the horoflow and the action of $G$ commutes.
In order to compare $(\mathfrak H_\R,\phi)$ with $(\widehat{\mathcal G}\Gamma, \psi)$, we define a $G$-equivariant continuous map 
\begin{displaymath}
	q \colon \mathfrak H_\R \to \widehat{\mathcal G}\Gamma,
\end{displaymath} by composing the map $\mathfrak H_\R \to \mathcal G\Gamma$ and the projection $p \colon \mathcal G\Gamma \to \widehat{\mathcal G}\Gamma$.
By construction, given any $(h,\gamma) \in \mathfrak H_\R$, the map $q$ maps homeomorphically the $\phi$-orbit of $(h,\gamma)$ onto the $\psi$-orbit of $q(h,\gamma) = p(\gamma)$.

\begin{prop}
\label{res: proj two sided shift onto gromov flow}
	The map $q \colon \mathfrak H_\R \to \widehat{\mathcal G}\Gamma$ is onto.
\end{prop}

\begin{proof}
	Let $\eta$ and $\xi$ be two distinct points of $\partial \Gamma$.
	According to \autoref{res: projection horofunction to boundary} there exists an integral horofunction $h \in \mathfrak H_0$ such that $\pi(h) = \xi$.
	On the other hand, there exists a primitive $h$-gradient line $\gamma \colon \R \to \Gamma$ such that $\gamma(-\infty) = \eta$ \cite[Proposition~5.2]{Coornaert:2002fh}.
	Being an $h$-gradient line, $\gamma$ is such that $\gamma(\infty) = \xi$.
	Hence $q$ maps the $\phi$-orbit of $(h,\gamma)$ to the (unique) $\psi$-orbit in $\widehat{\mathcal G}\Gamma$ joining $\eta$ to $\xi$.
	This works for any two distinct points $\eta, \xi \in \partial \Gamma$.
	Hence $q$ is onto.
\end{proof}

\paragraph{Discretization of the flow.}
We denote by $\mathfrak H_\Z$ the closed subset of $\mathfrak H_\R$ containing all the pairs $(h,\gamma)$ where $\gamma$ is a primitive $h$-gradient line.
Observe that $\mathfrak H_\Z$ is $G$-invariant.
Moreover the time $1$ flow $\phi_1$ on $\mathfrak H_\R$ induces a homeomorphism of $\mathfrak H_\Z$ onto itself \cite[Definition~2.6]{Coornaert:2002fh}.
The system $(\mathfrak H_\Z, \phi_1)$ is called the \emph{discrete horoflow} of $\Gamma$.
Let $\bar {\mathfrak H}_\Z$ and $\bar {\mathfrak H}_\R$ the quotients $\mathfrak H_\Z/G$ and $\mathfrak H_\R/G$ respectively.
As the flow $\phi$ and the action of $G$ commute, $\phi$ induces a flow $\bar \phi = (\bar \phi_t)_{t \in \R}$ on the space $\bar {\mathfrak H}_\R = \mathfrak H_\R/G$.
Moreover $\bar \phi_1$ induces a homeomorphism of $\bar {\mathfrak H}_\Z$ onto itself $\bar {\mathfrak H}_\Z$.
One can check that $(\bar{\mathfrak H}_\R, \bar \phi)$ is the suspension of the system $(\bar{\mathfrak H}_\Z, \bar \phi_1)$ \cite[Proposition~4.8]{Coornaert:2002fh}.
Let $r \colon \mathfrak H_\Z \to \mathfrak H_0$ be the map sending $(h,\gamma)$ to $\gamma(0)^{-1}h$ (recall that $\gamma$ is a primitive gradient line, hence $\gamma(0)$ is a vertex of $\Gamma$ which corresponds to a unique element of $G$).
We observe that $r$ induces a map $\bar r \colon \bar{\mathfrak H}_\Z \to \mathfrak H_0$ such that $\bar r \circ \bar \phi_1 = T \circ \bar r$.
Actually $(\bar{\mathfrak H}_\Z, \bar \phi_1)$ is conjugated to the canonical two sided shift induced by $(\mathfrak H_0, T)$ \cite[Propositions~2.5 and 2.22]{Coornaert:2002fh}.

%
\subsection{Gromov's coding}
%
\label{sec: gromov coding}

In \cite[Theorem~ 8.4.C]{Gro87} Gromov explains that $(\mathfrak H_0,T)$ is conjugated to a subshift of finite type.
We recall here Gromov's coding as it is detailed by Coornaert and Papadopoulos in \cite{Coornaert:2001ff}.

\paragraph{The alphabet.}
Fix a real number $R_0 \geq 100\delta +1$ and an integer $L_0 \geq 2R_0 +32 \delta + 1$.
Given a subset $S$ of $\Gamma$ and a number $r \in \R_+$, we denote by $\mathcal N_r(S)$ the $r$-neighborhood of $S$, i.e. the set
\begin{displaymath}
	\mathcal N_r(S) = \set{x \in \Gamma}{\dist xS \leq r}.
\end{displaymath}
Let $h \in \mathfrak H_0$ be a horofunction and $\gamma \colon \R_+ \to \Gamma$ the minimal $h$-gradient line starting at $1$.
The set $V(h)$ is the $R_0$-neighborhood of $\gamma$ restricted to $\intval 0{L_0}$.
In addition we define the map 
\begin{displaymath}
	b(h) \colon V(h) \to \R,
\end{displaymath}
to be the restriction of $h$ to $V(h)$ (recall that our horofunctions vanish at $1$).
The alphabet $\mathcal B$ is the set of functions $b(h) \colon V(h) \to \R$ where $h$ runs over $\mathfrak H_0$.
It is a finite set \cite[Proposition~6.2]{Coornaert:2001ff}.

\paragraph{The coding.}
Let $\sigma \colon \mathcal B^\N \to \mathcal B^\N$ the shift map, i.e. the map sending the sequence $(b_n)_{n \in \N}$ to $(b_{n+1})_{n \in \N}$.
We define a map $\jmath \colon \mathfrak H_0 \to \mathcal B^\N$ by sending a horofunction $h \in \mathfrak H_0$ to the sequence $(b_n)$ defined by 
\begin{displaymath}
	b_n = b(T^nh), \quad \forall n \in \N.
\end{displaymath}
Let $\Sigma$ be the image of $\jmath$.
One observes that $\jmath \circ T = \sigma \circ \jmath$ \cite[Lemma~6.3]{Coornaert:2001ff}.
Moreover $\jmath \colon \mathfrak H_0 \to \mathcal B^\N$ induces a homeomorphism from $\mathfrak H_0$ onto $\Sigma$, which is a subshift of finite type of $\mathcal B^\N$ \cite[Theorem~7.18]{Coornaert:2001ff}.

\begin{rema}
\label{rem: cocycle horofunction loc cst}
From now on we implicitly identify $\mathfrak H_0$ with its image $\Sigma$.
In particular, we say that $h_1,h_2 \in \mathfrak H_0$ belong to the same cylinder of length $n$, if $\jmath(h_1)$ an $\jmath(h_2)$ coincide on the first $n$ letters.
We endow $\mathfrak H_0$ with the canonical distance on $\mathcal B^\N$: for every $h_1,h_2 \in \mathfrak H_0$, we let $\dist{h_1}{h_2} = e^{-n}$ where $n$ is the largest integer such that $h_1$ and $h_2$ belong to the same cylinder of length $n$.

\medskip
	Let $h \in \mathfrak H_0$.
	By construction $b(h)$ completely determines the restriction of $h$ to the ball of radius $R_0$ centered at $1$.
	Hence $\theta(h)$ only depends on $b(h)$, i.e. the first letter of $\jmath(h)$.
	Consequently, if $h_1,h_2 \in \mathfrak H_0$ belong to the same cylinder of length $n$, then $\theta_n(h_1) = \theta_n(h_2)$, or said differently the minimal $h_1$- and $h_2$-gradient line starting at $1$ coincide on $\intval 0n$.
\end{rema}

\paragraph{Choice of an irreducible component.}
Unlike the geodesic flow on a negatively curved compact surface the dynamical system $(\mathfrak H_0,T)$ is a priori not topologically mixing and even not topologically transitive.
This can be a major issue to study its properties.
The difficulty comes from the fact that two points in $\Gamma \cup \partial \Gamma$ may be joined by multiple geodesics.
This pathology can be illustrated by the following simple example.

\medskip

\begin{exam}
\label{exa: horoboundary free by finite}
Let $G = \free 2 \times B$ be the direct product of the free group generated by $\{a_1,a_2\}$ and a non-trivial finite group $B$.
We choose for the generating set $A = \{a_1,a_1^{-1}, a_2,a_2^{-1}\} \cup B$ and write $\Gamma$ for the corresponding Cayley graph.
One can check easily that $\mathfrak H_0(\Gamma)$, contains one copy of $\partial \free 2$ (the usual Gromov boundary of $\free 2$) for each element $b \in B$.
Said differently there is an embedding of $\partial \free 2 \times B$ into $\mathfrak H_0(\Gamma)$.
This subset is invariant under the action of $G$.
More precisely, for every $(h,b) \in \partial \free 2 \times B$, for every $g = (f,u)$ in $G$, we have 
\begin{displaymath}
	g\cdot (h,u) = (f \cdot h, ub).
\end{displaymath}
Assume now that the order on $A$ is such that the letters $a_1,a_1^{-1}, a_2,a_2^{-1}$ are smaller that the one of $B$.
Then for every $b \in B$, the \og layer\fg\ $\partial \free 2 \times \{b\}$ is invariant under $T$.
On the contrary if every letter of $B$ is smaller than $a_1,a_1^{-1}, a_2,a_2^{-1}$, then $T$ maps $\partial \free 2 \times B$ onto $\partial \free 2 \times \{1\}$.
\end{exam}

\medskip
Nevertheless for our purpose, one does not need to work with the whole system $(\mathfrak H_0, T)$. 
It is sufficient to restrict our attention to an irreducible component of the system, as long as it visits almost all the group $G$.
This is formalized by the visibility property (see \autoref{def: visibility property}).
The goal of this section is to prove that such an irreducible component exists (see \autoref{res: irreducible component w/ visibility}).
Our main tool is the space of the geodesic flow introduced by Gromov in \cite[Section~8.3]{Gro87}.
From now on, $\Gamma$ is the Cayley graph of any hyperbolic group, as in the previous section.

\medskip
We have seen that $(\mathfrak H_0,T)$ is conjugated to a subshift of finite type.
We write $\mathfrak I_1, \dots, \mathfrak I_m$ for the irreducible components of $(\mathfrak H_0,T)$ (see \autoref{sec: subshift finite type}).
Recall that for every $h \in \mathfrak H_0$, there exists $i \in \intvald 1m$ and $n_0 \in \N$ such that for every $n \geq n_0$, the horofunction $T^nh$ belongs to $\mathfrak I_i$ (which we call the asymptotic irreducible component of $h$).
We can now state the main result of this section:

\begin{prop}
\label{res: irreducible component w/ visibility}
	There exists an irreducible component $\mathfrak I_i$ of $\mathfrak H_0$ such that the extension of $(\mathfrak I_i,T)$ by $\theta$ has the visibility property.
\end{prop}

\begin{proof}
	According to \autoref{res: dense orbit Gromov geodesic flow} there exists $\widehat \gamma$ in $\widehat{\mathcal G}\Gamma$ such that the positive orbit defined by
	\begin{displaymath}
		\set{\psi_s(\widehat \gamma)}{s \in \R_+}
	\end{displaymath}
	has a dense image in $\widehat{\mathcal G}\Gamma/G$.
	Recall that $(\mathfrak H_\R, \phi)$ is the horoflow of $\Gamma$ introduced in \autoref{sec: dyn on horoboundary}.
	The map $q \colon \mathfrak H_\R \to \widehat{\mathcal G}\Gamma$ being onto (\autoref{res: proj two sided shift onto gromov flow}), we can fix a pre-image $(h,\gamma) \in \mathfrak H_\R$ of $\widehat \gamma$ by $q$.
	Without loss of generality we can assume that $(h,\gamma)$ actually belongs to the space of the discrete horoflow $\mathfrak H_\Z$.
	We denote by $h_0$ the image of $(h, \gamma)$ by the map $r \colon \mathfrak H_\Z \to \mathfrak H_0$, i.e. $h_0 = \gamma(0)^{-1}h$.
	We choose for $\mathfrak I_i$ the asymptotic irreducible component of $h_0$.
	In particular there exists $K\in \N$, such that for every integer $k \geq K$, we have $T^k(h_0) \in \mathfrak I_i$.
	
	\medskip 
	We now study the properties of the map $\theta\colon \mathfrak H_0 \to G$ restricted to $\mathfrak I_i$.
	It is an exercise of hyperbolic geometry to prove that there exists $R_0 \in \R_+$ with the following property: given any two points $y,y' \in \Gamma$, there exists a bi-infinite geodesic $\nu \in \mathcal G\Gamma$ such that $\dist y\nu \leq R_0$ and $\dist{y'}\nu \leq R_0$.
	Recall that the map $\mathcal G \Gamma \to \Gamma$ sending $\nu$ to $\nu(0)$ as well as the projection $p \colon \mathcal G \Gamma \to \widehat{\mathcal G}\Gamma$ are quasi-isometries.
	Hence there exists $\kappa \geq 1$ and $\epsilon \geq 0$ such that for every $\nu, \nu' \in \mathcal G\Gamma$ we have
	\begin{equation}
	\label{eqn: irreducible component w/ visibility}
		\dist{\nu(0)}{\nu'(0)}\leq \kappa\dist{p(\nu)}{p(\nu')} + \epsilon.
	\end{equation}
	We define the finite set $U$ by
	\begin{displaymath}
		U = \set{u \in G}{\dist 1u \leq \kappa + \epsilon + R_0 +1 + 50\delta}.
	\end{displaymath}
	We are going to prove that for every $g \in G$, there exists a horofunction $h \in \mathfrak I_i$, an integer $n \in \N$ and two elements $u_1,u_2 \in U$ such that $g = u_1 \theta_n(h)u_2$.
	
	\medskip
	Let $g \in G$.
	There exists a geodesic $\nu \in \mathcal G\Gamma$ such that $\dist 1\nu \leq R_0$ and $\dist g\nu \leq R_0$.
	Up to changing the parametrization of $\nu$, we can assume that $\dist 1{\nu(0)}\leq R_0$ and $\dist g{\nu(m)} \leq R_0 + 1$ for some integer $m \in \N$.
	We denote by $\widehat \nu$ the image of $\nu$ in $\widehat{\mathcal G}\Gamma$.
	According to our choice of $\widehat \gamma$, there exist a sequence $(g_n)$ of elements of $G$ and a sequence $(s_n)$ of times diverging to infinity, such that $(g_n\psi_{s_n}(\widehat \gamma))$ converges to $\widehat \nu$.

	\medskip
	The projection $p \colon \mathcal G \Gamma \to \widehat{\mathcal G}\Gamma$ maps homeomorphically $\phi$-orbits onto $\psi$-orbits.
	Hence there exists a sequence $(t_n)$ of times diverging to infinity such that for every $n \in \N$, the map $p$ sends $g_n\phi_{t_n}(\gamma)$ to $g_n\psi_{s_n}(\widehat \gamma)$.
	Combining (\ref{eqn: irreducible component w/ visibility}) with the fact that $(g_n\psi_{s_n}(\widehat \gamma))$ converges to $\widehat \nu$ we get that for sufficiently large integer $n$
	\begin{displaymath}
		\dist{g_n\gamma(t_n)}{\nu(0)} \leq \kappa + \epsilon.
	\end{displaymath}
	The convergence taking place in $\widehat{\mathcal G}\Gamma$ also tells us that $(g_n\gamma(\infty))$ and $(g_n\gamma(-\infty))$ converge to $\nu(\infty)$ and $\nu(-\infty)$ respectively.
	Recall that the metric on $\Gamma$ is $8\delta$-quasi-convex \cite[Chapitre~10, Corollaire~5.3]{CooDelPap90}
	Combined with the previous inequality we get that for sufficiently large integer $n$
	\begin{displaymath}
		\dist{g_n\gamma(t_n + m)}{\nu(m)} \leq \kappa + \epsilon + 50\delta.
	\end{displaymath}
	For every $n \in \N$ we denote by $k_n$ the integer the closest to $t_n$ so that $\abs{t_n - k_n} \leq 1$.
	Recall that $\nu$ has been chosen to pass close by $1$ and $g$.
	Combining all these facts together we finally get that there exists $n_1 \in \N$ such that for all integer $n \geq n_1$ we have
	\begin{displaymath}
		\dist 1{g_n\gamma(k_n)} \leq \kappa + \epsilon + R_0 +1
		\quad \text{and} \quad
		\dist g{g_n\gamma(k_n+m)} \leq \kappa + \epsilon + R_0 +1 + 50\delta.
	\end{displaymath}
	Let $n \geq n_1$.
	Recall that $\gamma$ is a \emph{primitive} $h$-gradient line, while $k_n$ and $m$ are integers.
	Hence $g_n\gamma(k_n)$ and $g_n\gamma(k_n+m)$ are elements of $G$.
	We let $u_n = g_n\gamma(k_n)$ and $u'_n = \gamma(k_n+m)^{-1}g_n^{-1}g$.
	The last inequalities tell us that $u_n$ and $u'_n$ belong to $U$.
	In addition we claim that 
	\begin{displaymath}
		g = u_n \theta_m\left(T^{k_n}(h_0)\right)u'_n.
	\end{displaymath}
	As we explained before, the map $r \colon \mathfrak H_\Z \to \mathfrak H_0$ induces a map $\bar r \colon \bar{\mathfrak H}_\Z \to \mathfrak H_0$ such that $\bar r \circ \bar \phi_1 = T \circ \bar r$.
	It follows that the ray $\rho_n \colon \R_+ \to \Gamma$ defined by $\rho_n(t) = \gamma(k_n)^{-1}\gamma(k_n+t)$ is the (unique) minimal $T^{k_n}(h_0)$-gradient ray starting at $1$.
	Consequently
	\begin{displaymath}
		\theta_m\left(T^{k_n}(h_0)\right) = \rho_n(m) = \gamma(k_n)^{-1}\gamma(k_n+m) = u_n^{-1}g{u'_n}^{-1},
	\end{displaymath}
	which completes the proof of our claim.
	This decomposition of $g$ holds for any integer $n \geq n_1$.
	However $k_n$ is diverging to infinity.
	Hence if $n$ is sufficiently large $k_n \geq K$, thus $T^{k_n}(h_0)$ belongs to the irreducible component $\mathfrak I_i$.
	Thus $g$ can be decomposed as announced in the proposition.
\end{proof}

The next lemmas formalize the fact that the irreducible component ``visits almost'' all the group $G$.
The statements are quite technical.
Nevertheless the idea one has to keep in mind is the following.
Given a horofunction $h \in \mathfrak I$, one can follows its trajectory in $G$ by looking at the corresponding minimal $h$-gradient line $\gamma \colon \R_+ \to \Gamma$.
More precisely, given $n \in \N$, the point $\gamma(n)$ can be thought at the position in $G$ at time $n$ of the orbit of $h$.
The next lemmas say that at time $n$ the positions reached by the flow almost embed in (\resp almost cover) the sphere $S(n) \subset G$.
These computations will be needed to estimate the spectral radii of various transfer operators (see Lemmas~\ref{res: computing spec radius - step 1}, \ref{res: computing spec radius - step 2} and \ref{res: computing twisted spec radius - step 1}).
Before stating the lemmas we introduce a few notations.
Let $n \in \N$.
For every $h_0 \in \mathfrak I$, the set $S_{\mathfrak I}(h_0,n)$ is
\begin{displaymath}
	S_{\mathfrak I}(h_0,n) = \set{h \in \mathfrak I}{T^nh = h_0} = T^{-n}(h_0) \cap \mathfrak I.
\end{displaymath}
In addition we define a map
\begin{displaymath}
	\begin{array}{rccc}
		p_n \colon & \mathfrak H_0 \times G \times G & \to & G \\
		& (h,u_1,u_2) & \to & u_1\theta_n(h)u_2
	\end{array}
\end{displaymath}

\begin{lemm}
\label{res: position vs sphere - embedding}
	Let $h_0 \in \mathfrak I$.
	For every $n \in \N$, the map $\mathfrak H_0 \to G$ sending $h$ to $\theta_n(h)$ induces an embedding from $S_{\mathfrak I}(h_0,n)$ into $S(n)$.
\end{lemm}

\begin{proof}
	Let $h \in \mathfrak I$ such that $T^nh = h_0$.
	Let $\gamma \colon \R_+ \to \Gamma$ be the minimal $h$-gradient line starting at $1$.
	As we observed before $\theta_n(h) = \gamma(n)$.
	In particular $\theta_n(h)$ belongs to $S(n)$.
	On the other hand, it follows from the definition of the map $T$ that $h = \theta_n(h)T^nh = \theta_n(h)h_0$.
	Hence $h$ is completely determined by $\theta_n(h)$ which completes the proof of the lemma.
\end{proof}

\begin{lemm}
\label{res: position vs sphere - cover}
	There exist $R, N \in \N$ with the following property.
	For every $h_0 \in \mathfrak I$, for every $n \in \N$, the sphere $S(n) \subset G$ is contained in the image of the map 
	\begin{displaymath}
		\bigsqcup_{k = n}^{n+N} \left(\fantomB S_{\mathfrak I}(h_0,k) \times B(R)\times B(R) \right) \to G
	\end{displaymath}
	induced by the sequence $(p_k)$.
\end{lemm}

\begin{proof}
	We start by defining the constants $R$ and $N$.
	Recall that the extension of $(\mathfrak I,T)$ by $\theta \colon \mathfrak I \to G$ has the visibility property, i.e. there is a finite set $U$ such that for every $g \in G$, there exist a horofunction $h \in \mathfrak I$, an integer $n \in \N$, and two elements $u_1,u_2 \in U$ satisfying $g = u_1\theta_n(h)u_2$ (\autoref{res: irreducible component w/ visibility}).
	We denote by $L$ the maximal length (in $\Gamma$) of an element of $U$.
	As the system $(\mathfrak I, T)$ is an irreducible subshift of finite type, there exists $K$ with the following property:
	for every $h_1,h_2 \in \mathfrak I$, for every $n \in \N$, there exists $h'_1 \in \mathfrak I$ and $k \in \intvald 0K$ such that $T^{n+k}h'_1 = h_2$ and $h_1$ and $h_1'$ belong to the same cylinder of length $n$.
	Finally we let 
	\begin{displaymath}
		R = 5L+K
		\quad \text{and} \quad
		N = 2L + K.
	\end{displaymath}
	We now fix a horofunction $h_0 \in \mathfrak I$ and an integer $n \in \N$. 
	Let $g \in S(n)$.
	According to the visibility property there exist a horofunction $h \in \mathfrak I$, an integer $m \in \N$ and two elements $u_1,u_2 \in U$ such that $g = u_1\theta_m(h)u_2$.
	Recall that the length (in $\Gamma$) of $\theta_m(h)$ is $m$, hence $\abs{n-m} \leq 2L$.
	As $(\mathfrak I, T)$ is irreducible, there exists $h' \in \mathfrak I$ and $k \in \intvald 0K$ such that $h$ and $h'$ belongs to the same cylinder of length $n+2L$ and 
	\begin{displaymath}
		T^{n+2L+k}h' = h_0.
	\end{displaymath}
	Since $m \leq n+2L$, we have $\theta_m(h') = \theta_m(h)$.
	It follows that 
	\begin{displaymath}
		\theta_{n+2L+k}(h') 
		= \theta_m(h')\theta_{n+2L-m+k}(T^mh') 
		= \theta_m(h)\theta_{n+2L-m+k}(T^mh') 
	\end{displaymath}
	Consequently 
	\begin{displaymath}
		g = u_1\theta_{n+2L+k}(h')u'_2,
	\end{displaymath}
	where
	\begin{displaymath}
		u'_2 = \left(\fantomB\theta_{n+2L-m+k}(T^mh')\right)^{-1}u_2.
	\end{displaymath}
	Observe that $u_1$ and $u_2$ belong to $B(L)\subset B(R)$.
	As we noticed $n \leq m + 2L$, thus $n - m +2L +k \leq 4L +K$.
	Consequently $u'_2$ belongs to $B(R)$.
	In other words $p_{n+2L+k}$ maps the element $(h',u_1,u'_2)$ of $S_{\mathfrak I}(h_0,n+2L+k) \times B(R)\times B(R)$ to $g$.
\end{proof}

%
\section{Potential and transfer operator}
%

The goal of this section is to prove the following statements.

\begin{theo}
\label{res: main theo amenability}
	Let $G$ be a group acting properly co-compactly by isometries on a hyperbolic space $X$.
	We assume that one of the following holds.
	Either
	\begin{enumerate}
		\item $X$ is the Cayley graph of $G$ with respect to a finite generating set, or
		\item $X$ is a $\operatorname{CAT}(-1)$ space.
	\end{enumerate}
	Let $H$ be a subgroup of $G$.
	We denote by $\omega_G$ and $\omega_H$ the exponential growth rates of $G$ and $H$ acting on $X$.
	The subgroup $H$ is co-amenable in $G$ if and only if $\omega_H = \omega_G$.
	In particular if $H$ is a normal subgroup of $G$, the quotient $G/H$ is amenable if and only if $\omega_H = \omega_G$.
\end{theo}

\begin{theo}
\label{res: main theo property T}
	Let $G$ be a group with Kazhdan's property (T) acting properly co-compactly by isometries on a hyperbolic space $X$.
	We assume that one of the following holds.
	Either
	\begin{enumerate}
		\item $X$ is the Cayley graph of $G$ with respect to a finite generating set, or
		\item $X$ is a $\operatorname{CAT}(-1)$ space.
	\end{enumerate}
	There exists $\epsilon > 0$ with the following property.
	Let $H$ be a subgroup of $G$.
	We denote by $\omega_G$ and $\omega_H$ the exponential growth rates of $G$ and $H$ acting on $X$.
	If $\omega_H > \omega_G - \epsilon$, then $H$ is a finite index subgroup of $G$.
\end{theo}

%
\subsection{The data}
%
\label{sec: data}

Let $G$ be a a group acting properly co-compactly by isometries on a hyperbolic space $X$.
As in the statement of Theorems~\ref{res: main theo amenability} and \ref{res: main theo property T} we consider two cases.
\begin{description}
	\item[Case 1.] The space $X$ is the Cayley graph of $G$ with respect to a finite generating set $A$.
	In this situation we denote by $\Gamma$ a copy of $X$.
	\item[Case 2.] The space $X$ is $\operatorname{CAT}(-1)$.
	In this situation we fix an arbitrary finite generating set $A$ of $G$ and denote by $\Gamma$ the Cayley graph of $G$ with respect to $A$.
\end{description}
In both cases we may assume without loss of generality that $A$ is symmetric.
As we work with two distinct metric spaces, namely the Cayley graph $\Gamma$ and the space $X$, we use this section to emphasize which objects are related to one or the other space.

\paragraph{Data related to $X$.}
The space $X$ is the one that will carry the geometric information.
We denote by $\delta$ its hyperbolicity constant.
We fix a base point $x_0 \in X$.
This allows us to identify $C_*(X)$ with the set of continuous maps vanishing at $x_0$.
We denote by $\pi_X \colon \mathfrak H(X) \twoheadrightarrow \partial X$ the projection studied in \autoref{res: projection horofunction to boundary}.
We denote by $\omega_G$ the exponential growth rate of $G$ acting on $X$.

\paragraph{Data related to $\Gamma$.}
The role of $\Gamma$ is to provide a support for coding the geodesic flow.
The space $\mathfrak H_0\subset \mathfrak H(\Gamma)$ refers to the integral horofunctions on the Cayley graph $\Gamma$.
We denote by $\pi_\Gamma \colon \mathfrak H_0 \twoheadrightarrow \partial \Gamma$ the projection coming from \autoref{res: projection integral horofunction to boundary}.
The maps $\theta \colon \mathfrak H_0 \twoheadrightarrow G$ and $T \colon \mathfrak H_0 \to \mathfrak H_0$ are the ones defined at the beginning of \autoref{sec: dyn on horoboundary}.
For simplicity we denote by $\mathfrak I$ the irreducible component of $(\mathfrak H_0, T)$ with the visibility property given by \autoref{res: irreducible component w/ visibility}.
Recall that for every $n \in \N$, the sets $S(n)$ and $B(n)$ are respectively the sphere and the ball of radius $n$, measured in $\Gamma$.

\paragraph{Comparing $\Gamma$ and $X$.}
Since $G$ acts properly co-compactly on $X$, the orbit maps $G \to X$ sending $g$ to $gx_0$ leads a $(\kappa,\ell)$-quasi-isometric embedding $f\colon \Gamma \to X$.
This map induces a homeomorphism $\partial \Gamma \to \partial X$ between the respective Gromov boundary of $\Gamma$ and $X$.
For simplicity we implicitly identify $\partial \Gamma$ and $\partial X$.

%
\subsection{Transfer operator for the irreducible component}
%
\label{sec: transfer op for irred compt}

\paragraph{Comparing horofunctions.}
The first task is to define a potential $F \colon \mathfrak H_0 \to \R_+^*$.
This potential defined on the dynamical system $(\mathfrak H_0, T)$ should reflect to geometry of $X$.
In the first case -- when $X$ is actually the Cayley graph $\Gamma$ of $X$ -- the geometry coincides with the dynamics, and we can simply take for $F$ the constant function equal to $e^{-\omega_G}$.
In the second case -- when $X$ is an arbitrary $\operatorname{CAT}(-1)$ space -- the situation is more subtle.
Indeed $\mathfrak H(X)$ does not necessarily coincide with $\mathfrak H(\Gamma)$ nor $\mathfrak H_0$.
As the space $X$ is $\operatorname{CAT}(-1)$, the set $\mathfrak H(X)$ coincides with the usual visual boundary of $X$.
More precisely the map $\pi_X \colon \mathfrak H(X) \to \partial X$ is a homeomorphism.
On the other hand the projection $\pi_\Gamma \colon \mathfrak H_0 \to \partial \Gamma$ is not always injective.
Nevertheless we are going to build a map comparing $\mathfrak H_0$ -- the horofunctions used for coding -- to $\mathfrak H(X)$ -- the horofunctions capturing the geometry of $X$.
This is the purpose of the next proposition.
Actually we develop a framework that covers both cases simultaneously.

\medskip
Recall that we identify $(\mathfrak H_0, T)$ with an appropriate subshift of finite type of $(\mathcal B^\N, \sigma)$.
This identification induces a distance on $\mathfrak H_0$ (\autoref{rem: cocycle horofunction loc cst}) as defined at the beginning of \autoref{sec: subshift finite type}.
Namely the distance between two horofunctions $h,h' \in \mathfrak H_0$ is $\dist h{h'} = e^{-n}$, where $n$ is the largest integer such that the respective images of $h$ and $h'$ in $\mathcal B^\N$ have the same first $n$ letters.
We denote by $\mathcal C(\mathfrak H_0, \C)$ for the space of continuous maps from $\mathfrak H_0$ to $\C$ while $H^\infty_\alpha(\mathfrak H_0, \C)$ stands for the space of functions with bounded $\alpha$-Hölder variations (see \autoref{sec: function spaces}).

\begin{prop}
\label{res: boundary comparison map}
	There exists a $G$-equivariant \emph{comparison map} $\mathfrak H_0 \to \mathfrak H(X)$ which we denote $h \mapsto h_X$ with the following properties.
	\begin{enumerate}
		\item \label{enu: boundary comparison map - diagram}
		The following diagram commutes
		\begin{center}
			\begin{tikzpicture}[description/.style={fill=white,inner sep=2pt},baseline=(current bounding box.center)] 
				\matrix (m) [matrix of math nodes, row sep=2em, column sep=2.5em, text height=1.5ex, text depth=0.25ex] 
				{ 
					\mathfrak H_0	&  \mathfrak H(X)\\
					\partial \Gamma	 & \partial X	\\
				}; 
				\draw[>=stealth, ->] (m-1-1) -- (m-1-2);
				\draw[>=stealth, ->] (m-2-1) -- (m-2-2);
				\draw[>=stealth, ->] (m-1-1) -- (m-2-1) node[pos=0.5, left]{$\pi_\Gamma$};
				\draw[>=stealth, ->] (m-1-2) -- (m-2-2) node[pos=0.5, right]{$\pi_X$};
			\end{tikzpicture} 
		\end{center}
		\item \label{enu: boundary comparison map - holder variations}
		There exists $\alpha \in \R_+^*$ such that the \emph{evaluation map} defined by
		\begin{displaymath}
			\begin{array}{rccc}
				\varphi \colon & \mathfrak H_0 & \to & \R \\
				& h & \to & h_X\left(\fantomB\theta(h)x_0\right)
			\end{array}
		\end{displaymath}
		belongs to $H^\infty_\alpha(\mathfrak H_0, \R)$.
	\end{enumerate}
\end{prop}

\begin{proof}
	We distinguish two cases depending whether $X$ is a Cayley graph or a $\operatorname{CAT}(-1)$ space.
	\paragraph{Case 1.}
	Assume first that $X$ is the Cayley graph of $G$ with respect to the generating set $A$.
	In this situation we defined $\Gamma$ to be exactly $X$.
	In particular $\mathfrak H_0$ can be see as a subset of $\mathfrak H(X)$.
	We simply define the comparison map $\mathfrak H_0 \to \mathfrak H(X)$ as the corresponding embedding.
	Point~\ref{enu: boundary comparison map - diagram} becomes obvious.
	It follows from the very definition of $\theta$ that $\varphi$ is constant equal to $-1$.
	Thus it belongs to $H^\infty_\alpha(\mathfrak H_0, \R)$ for every $\alpha \in \R_+^*$.
	
	\paragraph{Case 2.}
	Assume now that $X$ is a $\operatorname{CAT}(-1)$ space.
	In this situation $\Gamma$ is the Cayley graph of $G$ with respect to an arbitrary finite generating set.
	Given $h \in \mathfrak H_0$ we define $h_X$ to be the Busemann function at $\xi = \pi_\Gamma(h)$ vanishing at $x_0$.
	By construction the diagram of Point~\ref{enu: boundary comparison map - diagram} commutes.
	
	\medskip
	Let $D = D(\kappa, \ell, \delta)$ be the parameter given by the Morse lemma (\autoref{res: morse lemma}).
	Let $h, h' \in \mathfrak H_0$ such that $\dist h{h'} < 1$.
	We denote by  $n \in \N^*$ the largest integer such that $h$ and $h'$ belong to the same cylinder of length $n$, so that $\dist h{h'} = e^{-n}$.
	We write $\gamma \colon \R_+ \to \Gamma$ for the minimal $h$-gradient ray starting at $1$ and let $\xi = \gamma(\infty)$.
	Let $c \colon \R_+ \to X$ be the ray starting at $x_0$ such that $c(\infty) = \xi$.
	Recall that the map $f \colon \Gamma \to X$ induced by the orbit map is a $(\kappa, \ell)$-quasi-isometry.
	In particular $f \circ \gamma$ is a $(\kappa, \ell)$-quasi-geodesic between $x_0$ and $\xi$.
	It follows from the stability of quasi-geodesics (\autoref{res: morse lemma}) that the Hausdorff distance between $f \circ \gamma$ and $c$ is bounded above by $D$.
	In a similar way we associate to $h'$ a gradient ray $\gamma' \colon \R_+ \to \Gamma$ as well as a geodesic ray $c' \colon \R_+ \to X$.
	
	\medskip	
	Since $h$ and $h'$ belong to the same cylinder of length $n$, the paths $\gamma$ and $\gamma'$ coincide on $\intval 0n$ (see \autoref{rem: cocycle horofunction loc cst}).
	In particular, for every $k \in \intvald 0n$, the point $y_k = f\circ \gamma(k) = f\circ \gamma'(k)$ lies in the $D$ neighborhood of both $c$ and $c'$.
	As $f$ is a $(\kappa, \ell)$-quasi-isometric embedding, we also observe that
	\begin{displaymath}
		\kappa^{-1}k - \ell \leq \dist{x_0}{y_k} \leq \kappa k + \ell.
	\end{displaymath}
	It follows that there exists $t \in \R_+$, such that $\dist{c(t)}{c'(t)} \leq 2D$ and $t \geq \kappa^{-1}n - \ell - D$.
	A standard exercise of $\operatorname{CAT}(-1)$ geometry shows that 
	\begin{displaymath}
		\dist{c(t/2)}{c'(t/2)} \leq C_1e^{-\frac 12t},
	\end{displaymath}
	where $C_1$ is a parameter that only depends on $D$.
	Recall that $\dist {x_0}{y_1} \leq \kappa + \ell$.
	Another exercise of $\operatorname{CAT}(-1)$ geometry shows that 
	\begin{displaymath}
		\abs{h_X(y_1) - \left[\fantomB \dist{c(t/2)}{y_1} - t/2 \right]} \leq C_2e^{-\frac 12t}
		\quad \text{and} \quad
		\abs{h'_X(y_1) - \left[\fantomB \dist{c'(t/2)}{y_1} - t/2 \right]} \leq C_2e^{-\frac 12t},
	\end{displaymath}
	where $C_2$ only depends on $\kappa$ and $\ell$.
	Consequently
	\begin{displaymath}
		\abs{h_X(y_1) - h'_X(y_1)} \leq C_3e^{-\frac 12t} \leq C_4e^{-\frac 12\kappa^{-1} n} \leq C_4 \dist h{h'}^{\frac 12\kappa^{-1}}.
	\end{displaymath}
	where $C_3$ and $C_4$ only depends on $D$, $\kappa$ and $\ell$.
	Nevertheless $y_1$ is the point $\theta(h)x_0 = \theta(h')x_0$.
	Hence the previous inequality exactly says that 
	\begin{displaymath}
		\abs{\varphi(h) - \varphi(h')} \leq C_4 \dist h{h'}^{\frac 12\kappa^{-1}}.
	\end{displaymath}
	This inequality holds for every $h,h'$ such that $\dist h{h'} < 1$.
	Consequently $\varphi$ belongs to the space $H^\infty_\alpha(\mathfrak H_0, \R)$ where $\alpha = \kappa^{-1}/2$, which proves Point~\ref{enu: boundary comparison map - holder variations}.	
\end{proof}

\paragraph{The potential $F$.}
From now on, we fix the comparison map $\mathfrak H_0 \to \mathfrak H(X)$, $h \mapsto h_X$ given by \autoref{res: boundary comparison map}.
We keep the notations introduced in this statement.
In particular the evaluation map $\varphi \colon \mathfrak H_0 \to \R$ sending $h$ to $h_X(\theta(h)x_0)$ belongs to $H^\infty_\alpha(\mathfrak H_0, \R)$.
The potential $F \colon \mathfrak H_0 \to \R_+^*$ is the map defined by 
\begin{equation}
\label{eqn: def potential in concrete situation}
	F(h) = \exp\left(\fantomB\omega_G \varphi(h)\right) = \exp\left( \fantomA\omega_G h_X\left(\fantomB\theta(h)x_0\right)\right), \quad \forall h \in \mathfrak H_0.
\end{equation}
It directly follows from the previous proposition that $\ln F$ belongs to $H^\infty_\alpha(\mathfrak H_0, \R)$

\paragraph{Remark.}
If $X$ is simply the Cayley graph $\Gamma$ of $X$, we previously observed that the evaluation map $\varphi$ is constant equal to $-1$.
Hence the potential becomes $F(h) = e^{-\omega_G}$.

\begin{lemm}
\label{res: computing cocycle of potential}
	For every $n \in \N$, for every $h \in \mathfrak H_0$, we have
	\begin{displaymath}
		F_n(h) = \exp\left( \fantomA\omega_G h_X\left(\fantomB\theta_n(h)x_0\right)\right).
	\end{displaymath}
\end{lemm}

\begin{proof}
	Let $h \in \mathfrak H_0$.
	It is sufficient to prove that for every $n \in \N$, we have
	\begin{displaymath}
		\sum_{k = 0}^{n-1} \varphi\circ T^k(h) = h_X\left(\fantomB\theta_n(h)x_0\right).
	\end{displaymath}
	The proof is by induction on $n$.
	By convention $\theta_0(h) = 1$.
	Since $h_X$ vanishes at $x_0$ the statement obviously holds for $n = 0$.
	Assume now that the claim holds for some $n \in \N$.
	It follows from the definition of $T$ that $T^n(h) = \theta_n(h)^{-1}h$.
	As the comparison map $\mathfrak H_0 \to \mathfrak H(X)$ is $G$-equivariant we get 
	\begin{align*}
		\varphi \circ T^n(h) 
		= \left[\theta_n(h)^{-1}h_X\right]\left(\fantomB\theta \left(T^n(h)\right)x_0\right)
		& = h_X\left(\fantomB\theta_n(h)\theta  \left(T^n(h)\right)x_0\right) - h_X\left(\fantomB\theta_n(h)x_0\right) \\
		& = h_X\left(\fantomB\theta_{n+1}(h)x_0\right) - h_X\left(\fantomB\theta_n(h)x_0\right).
	\end{align*}
	The statement for $n+1$ now follows from the induction hypotheses.
\end{proof}

\begin{lemm}
\label{res: horofunction vs distance}
	There exists a constant $C\in \R_+^*$, such that for every $h \in \mathfrak H_0$, for every $n \in \N$, we have
	\begin{displaymath}
		\frac 1C
		\leq \frac{F_n(h)}{\exp\left (-\omega_G \dist{\fantomB\theta_n(h)x_0}{x_0}\right)}
		\leq C.
	\end{displaymath}
\end{lemm}

\begin{proof}
	According to \autoref{res: computing cocycle of potential}, it suffices to show that there exists $C' \in \R_+^*$ such that for every $h \in \mathfrak H_0$, for every $n \in \N$ we have
	\begin{displaymath}	
		\abs{h_X\left(\fantomB\theta_n(h)x_0\right) + \dist{\fantomB\theta_n(h)x_0}{x_0}} \leq C'.
	\end{displaymath}
	By the stability of quasi-geodesics, there exists $D \in \R_+$ such that the Hausdorff distance between two $(\kappa, \ell)$-quasi-geodesics of $X$ joining the same endpoints (possibly in $\partial X$) is at most $D$.
	Let $h \in \mathfrak H_0$.
	Let $c \colon \R_+ \to X$ be  geodesic ray between $x_0$ and $\xi = \pi_X(h_X)$.
	We write $b \colon X \to \R$ for the corresponding Busemann function vanishing at $x_0$.
	Note that $h_X$ and $b$ are two horofunctions of $X$ whose image by $\pi_X \colon \mathfrak H(X) \to \partial X$ is $\xi$.
	It follows that $\norm[\infty]{h_X - b} \leq 64\delta$ \cite[Corollary~3.8]{Coornaert:2001ff}.
	Let $\gamma \colon \R_+ \to \Gamma$ be the minimal $h$-gradient line starting at $1$.
	Observe that $f \circ \gamma$ is a $(\kappa, \ell)$ quasi-geodesic of $X$.
	Hence the Hausdorff distance between $f\circ \gamma$ and $c$ is at most $D$.
	
	\medskip
	Let $n \in \N$.
	By construction the element $\theta_n(h)$ lies on $\gamma$.
	Thus there exists $t \in \R_+$, such that $\dist{\theta_n(h)x_0}{c(t)} \leq D$.
	It follows that 
	\begin{equation}
	\label{eqn: horofunction vs distance}
		\abs{\dist{\theta_n(h)x_0}{x_0} + b(c(t))} \leq D.
	\end{equation}
	Recall that the Busemann function $b$ is a $1$-Lipschitz.
	Combined with the fact that $\norm[\infty]{h_X - b} \leq 64\delta$ we get
	\begin{displaymath}
		\abs{h_X\left(\theta_n(h)x_0\right) - b(c(t))}
		\leq D + 64 \delta.
	\end{displaymath}
	Hence (\ref{eqn: horofunction vs distance}) becomes
	\begin{displaymath}	
		\abs{h_X\left(\theta_n(h)x_0\right) + \dist{\theta_n(h)x_0}{x_0}} \leq C',
	\end{displaymath}
	where $C' =2D + 64 \delta$.
	Observes that $C'$ neither depends on $h$ of $n$, hence the proof is complete.
\end{proof}

\paragraph{The transfer operator.}
The transfer operator associated to the potential $F$ is the operator $\mathcal L \colon \mathcal C(\mathfrak I, \C) \to \mathcal C(\mathfrak I, \C)$ defined by
\begin{equation}
\label{eqn: def transfer op in concrete situation}
	\mathcal L\Phi(h_0) = \sum_{Th = h_0} F(h)\Phi(h),\quad \forall \Phi \in \mathcal C(\mathfrak I, \C), \ \forall h_0 \in \mathfrak I.
\end{equation}
Note that the restriction map $\mathcal C(\mathfrak H_0, \C) \to \mathcal C(\mathfrak I,\C)$ induces a $1$-Lipschitz map from $H^\infty_\alpha(\mathfrak H_0, \C)\to H^\infty_\alpha(\mathfrak I, \C)$.
Hence $\ln F$ restricted to $\mathfrak I$ belongs to $H^\infty_\alpha(\mathfrak I, \C)$.
As we observed in the appendix $\mathcal L$ induces a bounded operator of $H^\infty_\alpha(\mathfrak I, \C)$.
Since the system $(\mathfrak I, T)$ is irreducible, the spectral radii of $\mathcal L$ seen as an operator of $\mathcal C(\mathfrak I, \C)$ or $H^\infty_\alpha(\mathfrak I, \C)$ are the same (\autoref{res: perron-frobenius}).
We denote it by $\rho$.
According to (\ref{eqn: computing spec radius}) it can be computed as follows
\begin{equation}
\label{eqn: computing spec radius - concrete case}
	\rho = \limsup_{n \to \infty} \sqrt[n]{\norm[\infty]{\mathcal L^n\mathbb 1}}.
\end{equation}

\paragraph{Computing $\rho$.}
The goal of this section is to prove that $\rho = 1$ (\autoref{res: computing spec radius - final step}).

\begin{lemm}
\label{res: computing spec radius - step 1}
	There exists $A_1 \in \R_+^*$ such that for every $n \in \N$, we have
	\begin{displaymath}
		\norm[\infty]{\mathcal L^n\mathbb 1} \leq A_1 \sum_{g \in S(n)} e^{-\omega_G\dist{gx_0}{x_0}}.
	\end{displaymath}
\end{lemm}

\begin{proof}
	We denote by $C$ the constant given by \autoref{res: horofunction vs distance}.
	Let $n \in \N$.
	Let $h_0 \in \mathfrak I$.
	According to \autoref{res: computing cocycle of potential} we have
	\begin{displaymath}
		\mathcal L^n\mathbb 1(h_0)  
		= \sum_{T^nh = h_0} F_n(h) 
		=  \sum_{T^nh = h_0}\exp\left( \fantomA\omega_G h_X\left(\fantomB\theta_n(h)x_0\right)\right).
	\end{displaymath}
	Applying \autoref{res: horofunction vs distance} we get
	\begin{equation}
	\label{eqn: computing spec radius - step 1}
	 	   \mathcal L^n\mathbb 1(h_0) 
		\leq C\sum_{T^nh = h_0}\exp\left(\fantomA-\omega_G\dist{\fantomB\theta_n(h)x_0}{x_0}\right).
	\end{equation}
	By \autoref{res: position vs sphere - embedding} the  map $\mathfrak H_0 \to G$ sending $h$ to $\theta_n(h)$ induces an embedding of $\set{h \in \mathfrak I}{T^nh = h_0}$ into $S(n)$.
	It follows that 
	\begin{displaymath}
		\mathcal L^n\mathbb 1(h_0)  \leq  C\sum_{g \in S(n)}e^{-\omega_G\dist{gx_0}{x_0}}.
	\end{displaymath}
	This inequality holds for every $h_0 \in \mathfrak I$, thus
	\begin{displaymath}
		\norm[\infty]{\mathcal L^n\mathbb 1}  \leq C \sum_{g \in S(n)}e^{-\omega_G\dist{gx_0}{x_0}}. \qedhere
	\end{displaymath}
\end{proof}

\begin{lemm}
\label{res: computing spec radius - step 2}
	There exists $A_2 \in \R_+^*$ such that for every $n \in \N$, we have
	\begin{displaymath}
		\sum_{g \in S(n)} e^{-\omega_G\dist{gx_0}{x_0}}
		\leq A_2 \norm[\infty]{\mathcal L^n\mathbb 1}.
	\end{displaymath}
\end{lemm}
	
\begin{proof}
	We write $C \in \R_+^*$, and $R,N \in \N$, for the constants given by \autoref{res: horofunction vs distance} and \autoref{res: position vs sphere - cover} respectively.	
	Recall that the map $f \colon \Gamma \to X$ induced by the orbit map is a $(\kappa, \ell)$-quasi-isometric embedding.
	Let $h_0 \in \mathfrak I$.
	Let $n \in \N$.
	Let $k \in \N$.
	Applying \autoref{res: horofunction vs distance} we observe that
	\begin{align}
	\notag
		\sum_{T^kh = h_0} \exp \left( \fantomA- \omega_G\dist{\fantomB\theta_k(h)x_0}{x_0}\right)
		& \leq C \sum_{T^kh = h_0}\exp\left( \fantomA\omega_G h_X\left(\fantomB\theta_k(h)x_0\right)\right) \\
	\label{eqn: computing spec radius - step 2}
		& \leq C \mathcal L^k\mathbb 1(h_0).
	\end{align}
	On the other hand if $(u_1,u_2)$ belongs to $B(R)\times B(R)$, the triangle inequality tells us that 
	\begin{displaymath}
		\abs{\dist{u_1\theta_k(h)u_2x_0}{x_0} - \dist{\theta_k(h)x_0}{x_0}}
		\leq \dist{u_1x_0}{x_0} + \dist{u_2x_0}{x_0}.
	\end{displaymath}
	However the map $f \colon \Gamma \to X$ being a $(\kappa,\ell)$-quasi-isometric embedding, we get 
	\begin{displaymath}
		\abs{\dist{u_1\theta_k(h)u_2x_0}{x_0} - \dist{\theta_k(h)x_0}{x_0}}
		\leq 2(\kappa R + \ell).
	\end{displaymath}
	Summing (\ref{eqn: computing spec radius - step 2}) when $(u_1,u_2)$ runs over $B(R)\times B(R)$ and $k$ over $\intvald n{n+N}$ gives
	\begin{align*}
		\MoveEqLeft{\sum_{k = n}^{n + N}
		\sum_{(u_1,u_2) \in B(R)\times B(R)}
		\sum_{T^kh = h_0} \exp \left( \fantomA- \omega_G\dist{\fantomB u_1\theta_k(h)u_2x_0}{x_0}\right)} \\
		& \leq 
		\card{B(R)}^2 e^{2\omega_G(\kappa R + \ell)} 
		\sum_{k = n}^{n +N}
		\sum_{T^kh = h_0} \exp \left( \fantomA- \omega_G\dist{\fantomB\theta_k(h)x_0}{x_0}\right) \\
		& \leq C \card{B(R)}^2 e^{2\omega_G(\kappa R + \ell)}
		\sum_{k = n}^{n + N}\mathcal L^k\mathbb 1(h_0).
	\end{align*}
	\autoref{res: position vs sphere - cover} provides a lower bound of the triple sum in the left hand side of the inequality, leading to
	\begin{displaymath}
		\sum_{g \in S(n)} e^{-\omega_G\dist{gx_0}{x_0}}
		\leq C \card{B(R)}^2 e^{2\omega_G(\kappa R + \ell)}\sum_{k = n}^{n + N}\mathcal L^k\mathbb 1(h_0).
	\end{displaymath}
	As the potential $F$ and the function $\mathbb 1$ are positive we observe that for every $k \geq n$, 
	\begin{displaymath}
		\mathcal L^k\mathbb 1(h_0)
		\leq \norm[\infty] {\mathcal L^k\mathbb 1} \leq \norm[\infty] {\mathcal L^{k-n}} \norm[\infty] {\mathcal L^n\mathbb 1}.
	\end{displaymath}
	Hence the previous inequality becomes
	\begin{displaymath}
		\sum_{g \in S(n)} e^{-\omega_G\dist{gx_0}{x_0}}
		\leq C \card{B(R)}^2 e^{2\omega_G(\kappa R + \ell)} \left(\sum_{k = 0}^N\norm[\infty] {\mathcal L^k}\right) \norm[\infty] {\mathcal L^n\mathbb 1},
	\end{displaymath}
	which is exactly the required inequality.
\end{proof}

\begin{prop}
\label{res: computing spec radius - final step}
	The spectral radius of $\mathcal L$ is $\rho = 1$.
\end{prop}

\begin{proof}
	We form the series
	\begin{displaymath}
		\Upsilon(s) = \sum_{n = 0}^\infty e^{-sn} \norm[\infty] {\mathcal L^n\mathbb 1}.
	\end{displaymath}
	It follows from (\ref{eqn: computing spec radius - concrete case}) that the critical exponent $\Upsilon(s)$ is $\ln \rho$.
	Hence it suffices to prove this critical exponent is $0$.
	Let $s \in \R_+^*$.
	According to Lemmas~\ref{res: computing spec radius - step 1} and \ref{res: computing spec radius - step 2} there exists $A_1, A_2$ such that for every $n \in \N$
	\begin{displaymath}
		A_1 \sum_{g \in S(n)} e^{-\omega_G \dist{gx_0}{x_0}} \leq \norm[\infty]{\mathcal L^n\mathbb 1}\leq A_2 \sum_{g \in S(n)} e^{-\omega_G \dist{gx_0}{x_0}} .
	\end{displaymath}
	Multiplying theses inequalities by $e^{-sn}$ and summing over $n$ gives
	\begin{displaymath}
		A_1 \sum_{n = 0}^\infty \sum_{g \in S(n)} e^{-sn}e^{-\omega_G \dist{gx_0}{x_0}} 
		\leq \Upsilon(s) 
		\leq A_2\sum_{n = 0}^\infty \sum_{g \in S(n)} e^{-sn}e^{-\omega_G \dist{gx_0}{x_0}}.
	\end{displaymath}
	The map $f \colon \Gamma \to X$ being a $(\kappa, \ell)$-quasi-isometric embedding, for every $n \in \N$, for every $g \in S(n)$ we have
	\begin{displaymath}
		 \kappa^{-1} \left[ \dist{gx_0}{x_0} - \ell\right] \leq n \leq \kappa \left[ \dist{gx_0}{x_0} + \ell\right].
	\end{displaymath}
	Consequently
	\begin{displaymath}
		A_1e^{-s\kappa \ell}  \sum_{g \in G} e^{-(\omega_G +s\kappa)\dist{gx_0}{x_0}} 
		\leq \Upsilon(s) 
		\leq A_2e^{s\kappa^{-1} \ell}\sum_{g \in G} e^{-(\omega_G + s\kappa^{-1})\dist{gx_0}{x_0}}.
	\end{displaymath}
	This can be reformulated using the Poincaré series of $G$ as
	\begin{displaymath}
		A_1e^{-s\kappa \ell}  \mathcal P_G(\omega_G + s \kappa)
		\leq \Upsilon(s) 
		\leq A_2e^{s\kappa^{-1} \ell} \mathcal P_G(\omega_G + s \kappa^{-1}).
	\end{displaymath}
	Recall that $s \to \mathcal P_G(\omega_G + s)$ converges if $s >0$ and diverges if $s < 0$.
	It follows that the critical exponent of  $\Upsilon (s)$ equals $0$.
\end{proof}

%
\subsection{Twisted transfer operator associated to a subgroup}
%
\label{sec: twisted transfer op for sbgp}

\paragraph{Data associated to a subgroup.}
Let $H$ be a subgroup of $G$.
We denote by $Y$ the space of left cosets of $H$ in $G$, i.e. $Y = H\setminus G$.
We write $y_0$ for the image of $1$ in $Y$.
In other words $y_0$ is the coset $H$.
We denote by $\mathcal H = \ell^2(Y)$ the space of square summable functions from $Y$ to $\C$.
The group $G$ acts on $Y$ by right translations.
It induces a unitary representation $\lambda \colon G \to \mathcal U(\mathcal H)$ defined
\begin{displaymath}
	\left[\lambda(g)\phi\right]{y} = \phi(y\cdot g), \quad \forall g \in G, \ \forall \phi \in \mathcal H.
\end{displaymath}
We call $\lambda$ the \emph{the regular representation of $G$ relative to $H$}.
We denote by $\omega_H$ the exponential growth rate of $H$ \emph{acting on $X$}.

\paragraph{Twisted transfer operator.}
We denote by $\mathcal C(\mathfrak I, \mathcal H)$ the set of continuous function from $\mathfrak I$ to $\mathcal H$.
Similarly $H^\infty_\alpha(\mathfrak I, \mathcal H)$ stands for the space of functions with bounded $\alpha$-Hölder variations (see \autoref{sec: function spaces}).
As explained in this appendix, the representation $\lambda$ leads to a twisted transfer operator $\mathcal L_\lambda \colon C(\mathfrak I, \mathcal H) \to C(\mathfrak I, \mathcal H)$ defined by
\begin{displaymath}
	\mathcal L_\lambda \Phi(h_0) = \sum_{Th = h_0} F(h)\lambda(\theta(h))^{-1} \Phi(h), 
	\quad \forall \Phi \in C(\mathfrak I, \mathcal H), \ \forall h_0 \in \mathfrak I.
\end{displaymath}
This operator induces a bounded operator of $H^\infty_\alpha(\mathfrak I, \mathcal H)$ (\autoref{res: holder bounded variation inv by transfer}).
We write $\rho_\lambda$ for the spectral radius of $\mathcal L_\lambda$ seen as an operator of $H^\infty_\alpha(\mathfrak I, \mathcal H)$.

\paragraph{Computing $\rho_\lambda$.}
Our goal is to provide an estimate of $\rho_\lambda$ in terms of $\omega_G$ and $\omega_H$.
Let us first remark that $\rho_\lambda \leq \rho$ (\autoref{res: upper bound spec radius twisted}) that is in our setting $\rho_\lambda \leq 1$ (\autoref{res: computing spec radius - final step}).
We now provide a lower bound for $\rho_\lambda$.
The proof follows the same strategy as the one of \autoref{res: computing spec radius - final step}.

\begin{lemm}	
\label{res: computing twisted spec radius - step 1}
	There exist $B_2 \in \R_+^*$ and a function $\Psi \in H^\infty_\alpha(\mathfrak I, \mathcal H)$ such that for every $n \in \N$, we have
	\begin{displaymath}
		\sum_{g \in S(n)\cap H} e^{-\omega_G\dist{gx_0}{x_0}}
		\leq B_2 \norm[\infty]{\mathcal L_\lambda^n\Psi}.
	\end{displaymath}
\end{lemm}

\begin{proof}
	We write $C\in \R_+^*$, and $R,N\in \N$, for the constants given by \autoref{res: horofunction vs distance} and \autoref{res: position vs sphere - cover} respectively.
	Recall that the map $f \colon \Gamma \to X$ induced by the orbit map is a $(\kappa, \ell)$-quasi-isometric embedding.
	We denote by $Z$ the following finite subset of $Y$.
	\begin{displaymath}
		Z = \set{y_0\cdot u}{u \in B(R)}.
	\end{displaymath}
	The map $\Psi \colon \mathfrak I \to \mathcal H$ is the constant function equal to the characteristic function $\mathbb 1_Z$ of $Z$.
	One observes easily that $\Psi$ belongs to $H^\infty_\alpha(\mathfrak I, \mathcal H)$.

	\medskip
	We now fix $h_0 \in \mathfrak I$ and $n \in \N$.
	Let $k \in \N$ and $u_2 \in B(R)$.
	Let us compute $\mathcal L^k_\lambda\Psi(h_0)$ at the point $y_0\cdot u_2^{-1}$.
	By definition we have 
	\begin{displaymath}
		\left[\mathcal L_\lambda^k\Psi (h_0)\right]\left(y_0\cdot u_2^{-1} \right)
		= \sum_{T^kh = h_0} F_k(h) \left[\lambda\left( \theta_k(h)^{-1} \right)\Psi \right]\left(y_0\cdot u_2^{-1} \right)
		= \sum_{T^kh = h_0} F_k(h) \mathbb 1_Z \left(y_0\cdot u_2^{-1}\theta_k(h)^{-1}  \right).
	\end{displaymath}
	Recall that $Z = y_0 \cdot B(R)$.
	Hence the term
	\begin{displaymath}
		 \mathbb 1_Z \left(y_0\cdot u_2^{-1}\theta_k(h)^{-1}  \right)
	\end{displaymath}
	equals $1$ if there exists $u_1 \in B(R)$ such that $u_1\theta_k(h)u_2$ belongs to $H$ and zero otherwise.
	Hence 
	\begin{displaymath}
		\sum_{u_1 \in B(R)}\sum_{\substack{T^kh = h_0, \\ u_1\theta_k(h)u_2 \in H}} F_k(h)
		\leq \card{B(R)} \left[\mathcal L_\lambda^k\Psi (h_0)\right]\left(y_0\cdot u_2^{-1} \right).
	\end{displaymath}
	Recall that the potential $F$ and the vector $\mathbb 1_Z$ are non-negative.
	It follows that 
	\begin{displaymath}
	\label{eqn: computing twisted spec radius - step 1 - eqn1}
		\sum_{u_1 \in B(R)}\sum_{\substack{T^kh = h_0, \\ u_1\theta_k(h)u_2 \in H}} F_k(h)
		\leq \card{B(R)} \norm[\ell^2(Y)]{\mathcal L_\lambda^k\Psi(h_0)}
		\leq \card{B(R)} \norm[\infty]{\mathcal L_\lambda^k\Psi}.
	\end{displaymath}
	Applying \autoref{res: horofunction vs distance} to the previous inequality, we observe (as in the proof of \autoref{res: computing spec radius - step 2}) that 
	\begin{equation}
	\label{eqn: computing twisted spec radius - step 2 - eqn 2}			
		\sum_{u_1 \in B(R)}\sum_{\substack{T^kh = h_0, \\ u_1\theta_k(h)u_2 \in H}} \exp \left( \fantomA- \omega_G\dist{\fantomB\theta_k(h)x_0}{x_0}\right)
		\leq C  \card{B(R)} \norm[\infty]{\mathcal L_\lambda^k\Psi}. 
	\end{equation}
	On the other hand given $(u_1,u_2) \in B(R)\times B(R)$, the triangle inequality tells us that 
	\begin{displaymath}
		\abs{\dist{u_1\theta_k(h)u_2x_0}{x_0} - \dist{\theta_k(h)x_0}{x_0}}
		\leq \dist{u_1x_0}{x_0} + \dist{u_2x_0}{x_0}.
	\end{displaymath}
	The map $f \colon \Gamma \to X$ being a $(\kappa,\ell)$-quasi-isometric embedding, we get
	\begin{displaymath}
		\abs{\dist{u_1\theta_k(h)u_2x_0}{x_0} - \dist{\theta_k(h)x_0}{x_0}}
		\leq 2(\kappa R + \ell).
	\end{displaymath}
	Summing (\ref{eqn: computing twisted spec radius - step 2 - eqn 2}) when $u_2$ runs over $B(R)$ and $k$ over $\intvald n{n+N}$ yields
	\begin{align*}
		\MoveEqLeft \sum_{k = n}^{n + N}
		\sum_{(u_1,u_2) \in B(R)\times B(R)}
		\sum_{\substack{T^kh = h_0, \\ u_1\theta_k(h)u_2 \in H}} \exp \left( \fantomA- \omega_G\dist{\fantomB u_1\theta_k(h)u_2x_0}{x_0}\right)  \\
		& \leq   C\card{B(R)}^2 e^{2\omega_G(\kappa R + \ell) }
		\sum_{k = n}^{n + N}\norm[\infty]{\mathcal L_\lambda^k\Psi}.
	\end{align*}
	\autoref{res: position vs sphere - cover} provides a lower bound of the triple sum in the left hand side of the inequality, leading to
	\begin{displaymath}
		\sum_{g \in S(n)\cap H} e^{-\omega_G\dist{gx_0}{x_0}}
		\leq C\card{B(R)}^2 e^{2\omega_G(\kappa R + \ell) }
		\sum_{k = n}^{n + N}\norm[\infty]{\mathcal L_\lambda^k\Psi}.
	\end{displaymath}
	Observe that for every $k \geq n$ we have,
	\begin{displaymath}
		\norm[\infty] {\mathcal L\lambda^k \Psi} \leq \norm[\infty] {\mathcal L_\lambda^{k-n}} \norm[\infty] {\mathcal L_\lambda^n\Psi}.
	\end{displaymath}
	Hence the previous inequality becomes
	\begin{displaymath}
		\sum_{g \in S(n)\cap H} e^{-\omega_G\dist{gx_0}{x_0}}
		\leq C\card{B(R)}^2 e^{2\omega_G(\kappa R + \ell) } \left(\sum_{k = 0}^N\norm[\infty] {\mathcal L_\lambda^k}\right) \norm[\infty] {\mathcal L_\lambda^n\Psi},
	\end{displaymath}
	which is exactly the required inequality.	
\end{proof}

\begin{prop}
\label{res: computing twisted spec radius - final step}
	The spectral radius $\rho_\lambda$ of $\mathcal L_\lambda$ satisfies the following inequality
	\begin{displaymath}
		\rho_\lambda \geq \exp\left(\frac{\omega_H-\omega_G}\kappa\right).
	\end{displaymath}
\end{prop}

\begin{proof}
	According to \autoref{res: computing twisted spec radius - step 1} there exist $B_2 \in \R_+^*$ and a function $\Psi \in H^\infty_\alpha(\mathfrak I, \mathcal H)$ such that for every $n \in \N$ we have
	\begin{equation}
	\label{eqn: computing twisted spec radius - final step}
		\sum_{g \in S(n)\cap H} e^{-\omega_G\dist{gx_0}{x_0}}
		\leq B_2 \norm[\infty]{\mathcal L_\lambda^n\Psi}.
	\end{equation}
	Recall that the canonical map $H_\alpha^\infty(\mathfrak I, \mathcal H) \to C(\mathfrak I, \mathcal H)$ is $1$-Lipschitz (see \autoref{sec: function spaces}).
	Hence for every $n \in \N$, we have
	\begin{displaymath}
		\norm[\infty]{\mathcal L_\lambda^n\Psi} 
		\leq \norm[\infty, \alpha]{\mathcal L_\lambda^n\Psi} 
		\leq \norm[\infty, \alpha]{\mathcal L_\lambda^n} \norm[\infty, \alpha]\Psi.
	\end{displaymath}
	Hence
	\begin{displaymath}
		 \limsup_{n \to \infty} \frac 1n \ln \norm[\infty]{\mathcal L_\lambda^n\Psi} 
		 \leq  \limsup_{n \to \infty} \frac 1n \ln \norm[\infty, \alpha]{\mathcal L_\lambda^n} 
		 \leq \ln \rho_\lambda.
	\end{displaymath}
	The left hand side of the inequality can be interpreted as the critical exponent of the series
	\begin{displaymath}
		\Upsilon_H(s) = \sum_{n = 0}^\infty e^{-sn}\norm[\infty]{\mathcal L_\lambda^n\Psi}.
	\end{displaymath}
	We use (\ref{eqn: computing twisted spec radius - final step}) exactly as we did in \autoref{res: computing spec radius - final step} to prove that for every $s \in \R$,
	\begin{displaymath}
		\mathcal P_H(\omega_G + s \kappa)
		\leq B_2e^{s\kappa \ell}\Upsilon_H(s),
	\end{displaymath}
	where $\mathcal P_H$ stands for the Poincaré series of $H$.
	Recall that $s \to \mathcal P_H(s)$ diverges whenever $s < \omega_H$.
	Consequently the critical exponent of $\Upsilon_H(s)$ is bounded below by $(\omega_H- \omega_G)/\kappa$, hence the result.
\end{proof}

\begin{coro}
\label{res: computing twisted spec radius - coro}
	If $\omega_H = \omega_G$, then $\rho_\lambda = 1$.
\end{coro}

\begin{proof}
	It directly follows from the observation that $\rho_\lambda \leq 1$.
\end{proof}

\paragraph{Remark.}
The converse statement actually holds.
It is a consequence of Theorems~\ref{res: generalization stadlbauer} and \ref{res: roblin}.
Indeed if $\rho_\lambda = 1$, then the group $H$ is co-amenable in $G$ (\autoref{res: generalization stadlbauer}), hence $\omega_H = \omega_G$ (\autoref{res: roblin}).
Nevertheless we are not aware of an upper bound of the spectral radius $\rho_\lambda$ in the spirit of \autoref{res: computing twisted spec radius - final step} which would directly leads to the converse direction (and an alternative proof of Roblin's theorem).

%
\subsection{Proofs of the theorems}
%
\label{sec: proofs}

We are now in position to prove Theorems~\ref{res: main theo amenability} and \ref{res: main theo property T}.

\begin{proof}[Proof of \autoref{res: main theo amenability}]
	The Cayley graph $\Gamma$ of $G$ is defined as in \autoref{sec: data}.
	This provides a subshift of finite type $(\mathfrak H_0, T)$ as detailed in \autoref{sec: gromov coding} together with the labelling map $\theta \colon \mathfrak H_0 \to G$ defined in \autoref{sec: dyn on horoboundary}.
	We extract from this dynamical system an irreducible component $\mathfrak I$, such that the extension of $(\mathfrak I,T)$ by $\theta$ has the visibility property (\autoref{res: irreducible component w/ visibility}).
	Using the strategy developed in \autoref{sec: transfer op for irred compt} we define a potential $F \colon \mathfrak I \to \R_+^*$ which belongs to $H^\infty_\alpha(\mathfrak I, \R)$ for some $\alpha \in \R_+^*$.
	We denote by $\mathcal L \colon H^\infty_\alpha(\mathfrak I, \C) \to H^\infty_\alpha(\mathfrak I, \C)$ the corresponding transfer operator.
	Its spectral radius is $\rho = 1$ (\autoref{res: computing spec radius - final step}).
	
	\medskip
	Let $H$ be a subgroup.
	We consider the set $Y$ of left $H$-cosets of $G$.
	The group $G$ acts on $Y$ by right translations.
	If the action is amenable, then it follows from Roblin's Theorem (\autoref{res: roblin}) that $\omega_H = \omega_G$.
	Let us assume now that $\omega_H = \omega_G$.
	The action of $G$ on $Y$ induces a unitary representation $\lambda \colon G \to \mathcal U(\mathcal H)$ where $\mathcal H$ stands for $\ell^2(Y)$.
	This leads to a twisted transfer operator $\mathcal L_\lambda \colon H^\infty_\alpha(\mathfrak I, \mathcal H) \to H^\infty_\alpha(\mathfrak I, \mathcal H)$.
	Since $\omega_H = \omega_G$, \autoref{res: computing twisted spec radius - coro} tells us that the spectral radius of $\mathcal L_\lambda$ is $\rho_\lambda = 1$.
	In particular $\rho_\lambda = \rho$.
	It follows from the amenability criterion (\autoref{res: generalization stadlbauer}) that the action of $G$ on $Y$ is amenable.
\end{proof}

\begin{proof}[Proof of \autoref{res: main theo property T}]
	The dynamical system $(\mathfrak I, T)$, the extension map $\theta \colon \mathfrak I \to G$, the potential $F \in H^\infty_\alpha(\mathfrak I, \R)$ as well as the transfer operator $\mathcal L \colon H^\infty_\alpha(\mathfrak I, \C) \to H^\infty_\alpha(\mathfrak I, \C)$ are build as in the previous proof.
	In particular the spectral radius of $\mathcal L$ is $\rho = 1$.
	Let $\eta \in \R_+^*$, be the constant given by \autoref{res: Kazhdan - spectral gap}.
	
	\medskip
	Let $H$ be an infinite index subgroup of $G$.
	We consider the set $Y$ of left $H$-cosets of $G$.
	The group $G$ acts transitively on $Y$ by right translations.
	It induces a unitary representation $\lambda \colon G \to \mathcal U(\mathcal H)$ where $\mathcal H$ stands for $\ell^2(Y)$.
	This leads to a twisted transfer operator $\mathcal L_\lambda \colon H^\infty_\alpha(\mathfrak I, \mathcal H) \to H^\infty_\alpha(\mathfrak I, \mathcal H)$.
	Since $H$ has infinite index in $G$, \autoref{res: Kazhdan - spectral gap} tells us that the spectral radius $\rho_\lambda$ of $\mathcal L_\lambda$ is bounded above by $1-\eta$.
	However \autoref{res: computing twisted spec radius - final step} provides a lower bound for $\rho_\lambda$ in termes of $\omega_G$ and $\omega_H$.
	It yields $\omega_H \leq \omega_G - \epsilon$ where
	\begin{displaymath}
		\epsilon = \kappa \abs{\ln( 1- \eta)}.
	\end{displaymath}
	Note that $\epsilon$ does not depends on $H$, which completes the proof of the theorem.
\end{proof}

\pagebreak

\appendix

%
\section{An extension of Kesten's criterion}
%

As explained in the introduction, this appendix is deeply inspired by the work of Stadlbauer.
We prove a variation of his amenability criterion.
Our approach makes an explicit use of representation theory and operator algebra, which were somehow hidden in \cite{Stadlbauer:2013dg}.
In addition it provides precise estimates that can be use to analyse groups with Kazhdan's property (T) (see \autoref{res: Kazhdan - spectral gap}).
Similar results were also obtained by Dougall in \cite{Dougall:2017va}.

\medskip
In this section $(\Sigma,\sigma)$ is a subshift of finite type of the alphabet $\mathcal A$.
We use the same notations as in \autoref{sec: subshift finite type}.

%
\subsection{Function spaces}
%
\label{sec: function spaces}
In order to study the dynamical system $(\Sigma,\sigma)$ and its extensions, we define a family of function spaces.
To that end, we fix a Banach space $(E, \normV)$.
We endow the space $\mathcal C(\Sigma,E)$ of continuous function $\Phi \colon \Sigma \to E$ with the norm $\normV[\infty]$ defined by
\begin{displaymath}
	\norm[\infty]\Phi = \sup_{x \in \Sigma} \norm{\Phi(x)}.
\end{displaymath}
For our purpose, it will be rather convenient to work with smooth functions. 
Given $\alpha >0$ we measure the \emph{$\alpha$-Hölder variations} of $\Phi$ by the quantity
\begin{displaymath}
	\Delta_{\alpha}(\Phi) = \sup_{x \neq y} \frac{\norm{\Phi(x) - \Phi(y)}}{\dist xy^\alpha}.
\end{displaymath}
We define the norm $\normV[\infty,\alpha]$ of $\Phi$ by
\begin{displaymath}
	\norm[\infty,\alpha]\Phi = \norm[\infty] \Phi + \Delta_\alpha(\Phi).
\end{displaymath}

\begin{defi}[Functions with bounded variations]
\label{def: Holder spaces}
	Let $\alpha >0$.
	The space $H^\infty_\alpha(\Sigma,E)$ is the set of all maps $\Phi \colon \Sigma \to E$ satisfying $\norm[\infty, \alpha] \Phi < \infty$.
\end{defi}

It is a standard exercise to prove that $H^\infty_\alpha(\Sigma,E)$ is a Banach space.
Moreover the canonical map $H^\infty_\alpha(\Sigma, E) \to \mathcal C(\Sigma,E)$ is a $1$-Lischitz embedding.
Sometimes it is more convenient to focus on the \emph{local} $\alpha$-Hölder variations of a function.
Given $r \in \R_+^*$ and $\Phi \in \mathcal C(\Sigma,E)$ we let
\begin{displaymath}
	\Delta_{\alpha,r}(\Phi) = \sup_{\substack{\dist xy < r, \\x \neq y}} \frac{\norm{\Phi(x) - \Phi(y)}}{\dist xy^\alpha}.
\end{displaymath}
The next lemma explains how $\Delta_{\alpha, r}(\Phi)$ depends on $r$. 

\begin{lemm}
\label{res: control local variation}
	Let $\alpha > 0$ and $r \in \R_+^*$.
	For every $\Phi \in \mathcal C(\Sigma,E)$ we have
	\begin{displaymath}
		\Delta_\alpha(\Phi) - 2r^{-\alpha}\norm[\infty]\Phi 
		\leq \Delta_{\alpha,r}(\Phi)
		\leq \Delta_\alpha(\Phi).
	\end{displaymath}
\end{lemm}

\begin{rema}
	In particular $\Delta_\alpha(\Phi)$ is finite if and only if so is $\Delta_{\alpha,r}(\Phi)$ for some (hence any) $r \in \R_+^*$.
\end{rema}

\begin{proof}
	Let $\Phi \in \mathcal C(\Sigma,E)$.
	The second inequality is obvious.
	Let us focus on the first one.
	Let $x,y \in \Sigma$.
	If $\dist xy < r$, then by definition we have
	\begin{displaymath}
		\norm{\Phi(x) - \Phi(y)} \leq \Delta_{\alpha,r}(\Phi) \dist xy^\alpha.
	\end{displaymath}
	On the other hand if $\dist xy \geq r$, the triangle inequality yields
	\begin{displaymath}
		\norm{\Phi(x) - \Phi(y)} \leq 2 \norm[\infty]\Phi \leq 2r^{-\alpha}\norm[\infty]\Phi \dist xy^\alpha.
	\end{displaymath}
	Consequently for every distinct $x,y \in \Sigma$ we have
	\begin{displaymath}
		\frac{\norm{\Phi(x) - \Phi(y)}}{\dist xy^\alpha} \leq \Delta_{\alpha,r}(\Phi) + 2r^{-\alpha}\norm[\infty]\Phi. \qedhere
	\end{displaymath}
\end{proof}

The next lemmas are straightforward.
Their proof is left to the reader.

\begin{lemm}
\label{res: ln of Hölder bounded}
	Let $\Phi \in \mathcal C(\Sigma,\R)$ such that $\Phi(x) > 0$,  for every $x \in \Sigma$.
	The map $\Phi$ belongs to $H^\infty_\alpha(\Sigma,\C)$ if and only if so does $\ln \Phi$.
\end{lemm}

\begin{lemm}
\label{res: multiplication by scalar map}
	Let $f \in H^\infty(\Sigma,\R)$ and $\Phi \in H^\infty_\alpha(\Sigma,E)$.
	The pointwise product function $f\Phi$ belongs to $H^\infty_\alpha(\Sigma,E)$.
	Moreover 
	\begin{displaymath}
		\norm[\infty,\alpha] {f\Phi}
		\leq \norm[\infty, \alpha]f \norm[\infty, \alpha] \Phi.
	\end{displaymath}
\end{lemm}

%
\subsection{Ruelle's Perron Frobenius Theorem}
%

\paragraph{Transfer operator.}
We fix a potential $F \colon \Sigma \to \R_+^*$ and assume that there exists $\alpha > 0$ such that $\ln F \in H^\infty_\alpha(\Sigma,\R)$.
For every $n \in \N$, for every $x \in \Sigma$, we let
\begin{displaymath}
	F_n(x) = F(x)F(\sigma x)\cdots F\left(\sigma^{n-1}x\right).
\end{displaymath}
By convention $F_0 = \mathbb1$.
To such a potential we associate a transfer operator $\mathcal L \colon \mathcal C(\Sigma,\C) \to \mathcal C(\Sigma,\C)$ defined by
\begin{displaymath}
	\mathcal L\Phi(x) = \sum_{\sigma y = x} F(y)\Phi(y) =  \sum_{a \in \mathcal A} \mathbb1_{\sigma[a]}(x) F(ax) \Phi(ax).
\end{displaymath}
One checks easily that the powers of $\mathcal L$ are given by the following formula
\begin{displaymath}
	\mathcal L^n\Phi(x) = \sum_{\sigma^ny = x} F_n(y) \Phi(y) = \sum_{w \in \mathcal W^n} \mathbb1_{\sigma^n[w]}(x) F_n(wx) \Phi(wx).
\end{displaymath}
It is a standard fact that $\mathcal L$ defines a bounded operator of both $\mathcal C(\Sigma,\C)$ and $H^\infty_\alpha(\Sigma,\C)$ \cite[Section~XII.2]{Hennion:2001fm}.
We write $\rho_\infty$ and $\rho$ for the spectral radius of $\mathcal L$ seen as an operator of $\mathcal C(\Sigma,\C)$ and $H^\infty_\alpha(\Sigma,\C)$ respectively.
One observes easily (see for instance \cite[Section~XII.2]{Hennion:2001fm}) that 
\begin{equation}
\label{eqn: computing spec radius}
	\rho_\infty  = \lim_{n \to \infty} \sqrt[n]{\norm[\infty]{\mathcal L^n\mathbb 1}}.
\end{equation}
Recall that by the Riesz representation theorem, the dual space of $\mathcal C(\Sigma, \R)$ can be identified with the set of measures on $\Sigma$.
We write $\mathcal L^*$ for the dual operator of $\mathcal L$.
From now on, unless mentioned otherwise, we see $\mathcal L$ as an operator on $H^\infty_\alpha(\Sigma,\C)$ rather than $\mathcal C(\Sigma,\C)$.
For a proof of the following version of the Ruelle Perron-Frobenius Theorem we refer the reader to \cite[Theorem~XII.6]{Hennion:2001fm} or \cite[Theorem~1.5]{Baladi:2000foa}.

\begin{theo}[Ruelle's Perron-Frobenius Theorem]
\label{res: perron-frobenius}
	If $(\Sigma,\sigma)$ is topologically transitive, then the following holds
	\begin{enumerate}
		\item $\rho = \rho_\infty$ is positive.
		\item There exists a probability measure $\mu$ on $\Sigma$ whose support is $\Sigma$ such that $\mathcal L^* \mu = \rho \mu$.
		\item $\rho$ is an eigenvalue of $\mathcal L$; the corresponding eigenspace has dimension $1$; it is spanned by  a function $h \in H^\infty_\alpha(\Sigma,\C)$ such that $h(x) > 0$, for every $x \in \Sigma$ and 
		\begin{displaymath}
			\int_\Sigma hd\mu = 1.
		\end{displaymath}
		\item $\mathcal L$ has only finitely many eigenvalue of modulus $\rho$; the corresponding eigenspaces are finite dimensional; the rest of the spectrum of $\mathcal L$ is included in a disc of radius strictly less than $\rho$.
	\end{enumerate}
\end{theo}

%
\subsection{Twisted transfer operator}
%
\label{sec: twisted transfer operator}

Let $G$ be a finitely generated group.
We fix a locally constant map $\theta \colon \Sigma \to G$.
For every $n \in \N$, for every $x \in \Sigma$ we write
\begin{displaymath}
	\theta_n(x) = \theta(x) \theta\left(\sigma x\right) \cdots \theta\left(\sigma^{n-1}x\right).
\end{displaymath}
By convention $\theta_0$ is the constant map sending $x$ to the identity $1 \in G$.
We use this data to produce an extension $(\Sigma_\theta, \sigma_\theta)$ of $(\Sigma,\sigma)$ as follows.
We let $\Sigma_\theta = \Sigma \times G$ and define $\sigma_\theta \colon \Sigma_\theta \to \Sigma_\theta$ by 
\begin{displaymath}
	\sigma_\theta(x,g) = (\sigma x, g\theta(x)).
\end{displaymath}
Recall that the extension $(\Sigma_\theta,\sigma_\theta)$ has the \emph{visibility property} if there exists a finite subset $U$ of $G$ such that for every $g \in G$, there exists two elements $u_1,u_2 \in U$, a point $x \in \Sigma$, and  an integer $n \in \N$, satisfying $g = u_1\theta_n(x)u_2$.

\medskip
We now fix a Banach space $(E, \normV)$.
We denote by $\mathcal B(E)$ the space of bounded operators on $E$ endowed with the operator norm, while $\isom E$ stands for the set of \emph{linear} isometries of $E$.
Let $\lambda \colon G \to \isom E$ be a homomorphism.
As the $\theta \colon \Sigma \to G$ is locally constant the composition $\lambda \circ \theta \colon \Sigma \to \mathcal B(E)$ has $\alpha$-Hölder bounded variations.
For simplicity we make the following abuse of notations: given $x \in \Sigma$ and $n \in \N$, we write $\lambda(x)$ for $\lambda\circ \theta(x)$ and $\lambda_n(x)$ for $\lambda\circ \theta_n(x)$.
In particular we have
\begin{displaymath}
	\lambda_n(x) = \lambda(x)\lambda\left(\sigma x\right)\cdots \lambda\left(\sigma^{n-1}x\right).
\end{displaymath}
The representation $\lambda$ allows us to define a \emph{twisted transfer operator} $\mathcal L_\lambda \colon \mathcal C(\Sigma, E) \to \mathcal C(\Sigma, E)$ as follows
\begin{displaymath}
	\mathcal L_\lambda \Phi (x) = \sum_{\sigma y = x}F(y)\lambda(y)^{-1}\Phi(y) = \sum_{a \in \mathcal A} \mathbb1_{\sigma[a]}(x)F(ax)\lambda(ax)^{-1}\Phi(ax).
\end{displaymath}
A standard computation shows that the $n$-th power of $\mathcal L_\lambda$ is given by
\begin{displaymath}
	\mathcal L_\lambda^n\Phi(x) = \sum_{\sigma^ny = x} F_n(y)\lambda_n(y)^{-1}\Phi(y) = \sum_{w \in \mathcal W^n} \mathbb1_{\sigma^n[w]}(x) F_n(wx) \lambda_n(wx)^{-1}\Phi(wx).
\end{displaymath}
If $\lambda$ is the trivial representation of $G$ in $\C$, we recover the usual Ruelle Perron-Frobenius operator defined in the previous section.

\begin{lemm}
\label{res: transfer bounded op in infty norm}
	The operator $\mathcal L_\lambda \colon \mathcal C(\Sigma, E) \to \mathcal C(\Sigma, E)$ is bounded. 
	More precisely for every $n \in \N$, we have
	\begin{displaymath}
		\norm[\infty]{\mathcal L_\lambda^n} \leq \norm[\infty]{\mathcal L^n \mathbb 1}.
	\end{displaymath}
\end{lemm}

\begin{proof}
	Let $n \in \N$.
	Let $\Phi \in \mathcal C(\Sigma,E)$.
	It follows from the triangle inequality that for every $x \in \Sigma$, 
	\begin{displaymath}
		\norm{\mathcal L_\lambda^n\Phi(x)} 
		\leq \sum_{w \in \mathcal W^n} \mathbb 1_{\sigma^n[w]}(x)F(wx) \norm{\lambda_n(wx)^{-1}\Phi(wx)}
		\leq \mathcal L^n\mathbb 1(x) \norm[\infty] \Phi.
	\end{displaymath}
	Hence $\mathcal L_\lambda^n$ is a bounded operator and its operator norm is at most $\norm[\infty]{\mathcal L^n \mathbb 1}$.
\end{proof}

\begin{prop}
\label{res: holder bounded variation inv by transfer}
	For every $n \in \N$, there exists $C_n\in \R_+^*$ such that for every linear representation $\lambda \colon G \to \isom E$ into a Banach space $(E, \normV)$, for every $\Phi \in H^\infty_\alpha(\Sigma,E)$ we have 
	\begin{displaymath}
		\Delta_\alpha(\mathcal L_\lambda^n \Phi) \leq e^{-n\alpha}\norm[\infty]{\mathcal L^n\mathbb 1}\Delta_\alpha(\Phi) + C_n \norm[\infty]\Phi.
	\end{displaymath}
\end{prop}

\begin{proof}
	As $(\Sigma,\sigma)$ is a subshift of finite type, there exists $r >0$ with the following property: for every $x,y \in \Sigma$, if $\dist xy < r$, then for every $n \in \N$, for every $w\in \mathcal W^n$, $\mathbb1_{\sigma^n[w]}(x) = \mathbb1_{\sigma^n[w]}(y)$.
	Said differently the words that can be added in front of $x$ or $y$ are the same.
	We will take advantage of this fact to estimate the \emph{local} Hölder variations of the twisted transfer operator.
	Recall that $\theta \colon \Sigma \to G$ is locally constant.
	Hence there exists $m \in \N$ such that $\theta$ is constant on every cylinder of length $m$.
	
	\medskip
	Let $\lambda \colon G \to \isom E$ be a linear representation.
	Let $n \in \N$.
	Let $\Phi \in H^\infty_\alpha(\Sigma,E)$.
	Let $x,y \in \Sigma$ be two distinct points such that $\dist xy <r$.
	We fix $w \in \mathcal W^n$ such that $\mathbb1_{\sigma^n[w]}(x)= \mathbb1_{\sigma^n[w]}(y)$ equals $1$.
	A standard computation tells us that
	\begin{align*}
		F_n(wx)\lambda_n(wx)^{-1}\Phi(wx) - F_n(wy)\lambda_n(wy)^{-1}\Phi(wy)
		&= F_n(wx)\lambda_n(wx)^{-1}\left(\Phi(wx)-\Phi(wy)\right) \\
		&+ F_n(wx)\left(\lambda_n(wx)^{-1}-\lambda_n(wy)^{-1}\right)\Phi(wy) \\
		&+ \left(F_n(wx)-F_n(wy)\right)\lambda_n(wy)^{-1}\Phi(wy).
	\end{align*}
	Recall that the image of $\lambda$ is contained in the isometry group of $E$.
	On the other hand we observe that $d(wx,wy) = e^{-n}d(x,y)$.
	Hence the triangle inequality yields
	\begin{align*}
		\MoveEqLeft{\norm{F_n(wx)\lambda_n(wx)^{-1}\Phi(wx) - F_n(wy)\lambda_n(wy)^{-1}\Phi(wy)}}\\
		&\leq  \left( \fantomB F_n(wx)\Delta_\alpha(\Phi)
		+ F_n(wx)\Delta_\alpha(\lambda_n)\norm[\infty]\Phi
		+ \Delta_\alpha(F_n)\norm[\infty]\Phi \right)  e^{-n\alpha}\dist xy^\alpha.
	\end{align*}
	This inequality holds for every $w \in \mathcal W^n$ such that $\mathbb1_{\sigma^n[w]}(x)$ -- which equals $\mathbb1_{\sigma^n[w]}(y)$  -- does not vanish.
	We sum these inequalities to get
	\begin{displaymath}
		\frac{\norm{\mathcal L_\lambda^n\Phi(x) - \mathcal L_\lambda^n\Phi(y)}}{\dist xy^\alpha}
		\leq e^{-n\alpha}\norm[\infty]{\mathcal L^n\mathbb 1}\Delta_\alpha(\Phi)  + e^{-n\alpha}\left(\fantomB\card{\mathcal W^n}\Delta_\alpha(F_n)+  \norm[\infty]{\mathcal L^n\mathbb 1}\Delta_\alpha(\lambda_n)\right)\norm[\infty]\Phi.
	\end{displaymath}
	Recall that $\theta$ is constant on any cylinder of length $m$.
	Hence $\theta_n$ is constant on any cylinder of length $n+m$.
	It follows from \autoref{res: control local variation} that $\Delta_\alpha(\lambda_n) \leq 2e^{(n+m)\alpha}$.
	Hence
	\begin{displaymath}
		\frac{\norm{\mathcal L_\lambda^n\Phi(x) - \mathcal L_\lambda^n\Phi(y)}}{\dist xy^\alpha}
		\leq e^{-n\alpha}\norm[\infty]{\mathcal L^n\mathbb 1}\Delta_\alpha(\Phi)  + \left(\fantomB\, e^{-n\alpha}\card{\mathcal W^n}\Delta_\alpha(F_n)+  2e^{m\alpha}\norm[\infty]{\mathcal L^n\mathbb 1}\right)\norm[\infty]\Phi.
	\end{displaymath}
	This estimation holds for every distinct $x,y \in \Sigma$ satisfying $\dist xy <r$, hence the right hand side of the last inequality is an upper bound of $\Delta_{\alpha,r}(\mathcal L_\lambda^n\Phi)$.
	Combined with \autoref{res: control local variation} we get that 
	\begin{displaymath}
		\Delta_\alpha(\mathcal L_\lambda\Phi) \leq e^{-n\alpha}\norm[\infty]{\mathcal L^n\mathbb 1}\Delta_\alpha(\Phi) + C_n \norm[\infty]\Phi,
	\end{displaymath}
	where 
	\begin{displaymath}
		C_n = e^{-n\alpha}\card{\mathcal W^n}\Delta_\alpha(F_n)+  2e^{m\alpha}\norm[\infty]{\mathcal L^n\mathbb 1} + 2r^{-\alpha}.
	\end{displaymath}
	Observe that the parameter $C_n$ neither depends on $\lambda$ nor on $\Phi$, thus the proof of the proposition is complete.
\end{proof}

Combined with \autoref{res: transfer bounded op in infty norm}, the previous proposition tells us that for every $n \in \N$, for every $\Phi \in H^\infty_\alpha(\Sigma,E)$ we have 
\begin{equation}
\label{eqn: upper bounded holder norm twisted op}
	\norm[\infty,\alpha]{\mathcal L_\lambda^n \Phi} \leq e^{-n\alpha}\norm[\infty]{\mathcal L^n\mathbb 1}\norm[\infty, \alpha]\Phi + \left(C_n + \norm[\infty]{\mathcal L^n\mathbb 1}\right) \norm[\infty]\Phi.
\end{equation}
Hence we can view $\mathcal L_\lambda$ as a bounded operator on $H^\infty_\alpha(\Sigma,E)$.
This is the point of view that we will adopt in the remainder of this appendix.
In particular we denote by $\norm[\infty, \alpha]{\mathcal L_\lambda}$ its corresponding operator norm and $\rho_\lambda$ its spectral radius.
Recall that $\rho_\infty$ stands for the spectral radius of $\mathcal L$ seen as an operator of $\mathcal C(\Sigma,\C)$.

\begin{coro}
\label{res: upper bound spec radius twisted}
	The spectral radius $\rho_\lambda$ of $\mathcal L_\lambda$ is bounded above by $\rho_\infty$.
\end{coro}

\begin{proof}
	Let $\beta > \ln \rho_\infty$.
	It follows from (\ref{eqn: computing spec radius}) that there exists $n_0\in \N$, such that for every $n \geq n_0$.
	\begin{displaymath}
		\norm[\infty]{\mathcal L^n\mathbb 1} \leq e^{n\beta}.
	\end{displaymath}
	We now fix $n \geq n_0$.
	According to \autoref{res: holder bounded variation inv by transfer} there exists $C \in \R_+$ such that for every $\Phi \in H^\infty_\alpha(\Sigma, E)$ we have
	\begin{displaymath}
		\norm[\infty,\alpha]{\mathcal L_\lambda^n \Phi} 
		\leq e^{n(\beta-\alpha)} \norm[\infty, \alpha]\Phi + C\norm[\infty]\Phi.
	\end{displaymath}
	Let $m \in \N$. 
	We are going to estimate $\norm[\infty, \alpha]{\mathcal L_\lambda^{nm}}$.
	To that end we choose $\Phi \in H^\infty_\alpha(\Sigma, E)$ and $k \in \intvald 0{m-1}$.
	Applying the previous inequality to $\mathcal L_\lambda^{kn}\Phi$ we get
	\begin{displaymath}
		\norm[\infty,\alpha]{\mathcal L_\lambda^{(k+1)n} \Phi}
		\leq e^{n(\beta-\alpha)} \norm[\infty, \alpha]{\mathcal L_\lambda^{kn}\Phi} + C\norm[\infty]{\mathcal L_\lambda^{kn}\Phi}.
	\end{displaymath}
	which combined with \autoref{res: transfer bounded op in infty norm} becomes
	\begin{displaymath}
		\norm[\infty,\alpha]{\mathcal L_\lambda^{(k+1)n} \Phi}
		\leq e^{n(\beta - \alpha)} \norm[\infty, \alpha]{\mathcal L_\lambda^{kn}\Phi} + Ce^{kn\beta} \norm[\infty]{\Phi}.
	\end{displaymath}
	We multiply these inequalities by $e^{-(k+1)n(\beta - \alpha)}$ and sum them when $k$ runs over $\intvald 0{m-1}$. 
	It gives
	\begin{displaymath}
		\norm[\infty, \alpha]{\mathcal L_\lambda^{nm}\Phi}
		\leq e^{nm(\beta - \alpha)}\norm[\infty, \alpha]\Phi 
		+ e^{nm\beta} \frac{Ce^{-n(\beta -\alpha)}}{e^{n\alpha}-1} \norm[\infty]\Phi.
	\end{displaymath}
	Recall that $\norm[\infty]{\Phi} \leq \norm[\infty,\alpha]\Phi$.
	So we have proved that there exists $D_n \in \R_+^*$ such that for every $m \in \N$, for every $\Phi \in H^\infty(\Sigma,E)$
	\begin{displaymath}
		\norm[\infty, \alpha]{\mathcal L_\lambda^{nm}\Phi}
		\leq  e^{nm\beta} D_n\norm[\infty,\alpha]\Phi.
	\end{displaymath}
	Hence for every $m \in \N$, we get
	\begin{displaymath}
		\norm[\infty, \alpha]{\mathcal L_\lambda^{nm}}
		\leq  e^{nm\beta}  D_n.
	\end{displaymath}
	Passing to the limit when $m$ tends to infinity we obtain
	\begin{displaymath}
		\ln\rho_\lambda 
		= \lim_{m \to \infty} \frac 1{nm} \ln \left(\norm[\infty, \alpha]{\mathcal L_\lambda^{nm}}\right)
		\leq \beta.
	\end{displaymath}
	This inequality holds for every $\beta \in \R$ such that $\beta > \ln \rho_\infty$, thus $\rho_\lambda \leq \rho_\infty$.
\end{proof}

%
\subsection{Renormalization}
%
\label{sec: changing potential}

In this section we study how the transfer operators are affected when replacing the potential $F$ by a rescaled and/or an homologous potential.
This will allow us later to assume that $\mathcal L$ has spectral radius $1$ and fixes the constant map $\mathbb 1$, which will considerably lighten the notations.

\paragraph{Rescaling a homologous potential.}
Assume that the system $(\Sigma,\sigma)$ is topologically transitive.
Let $h \in H^\infty_\alpha(\Sigma,\C)$ be the positive eigenvector given by \autoref{res: perron-frobenius}.
We define a new potential $F' \colon \Sigma \to \R_+^*$ by
\begin{equation}
\label{eqn: def homologue potential}
	F'(x) = \frac 1\rho \cdot\frac {h(x)}{h\circ \sigma(x)}F(x), 
\end{equation}
so that 
\begin{displaymath}
	\ln F'(x) = - \ln \rho + \ln\left(h(x)\right) - \ln \left( h\circ \sigma(x)\right)+ \ln F(x).
\end{displaymath}
We know that $\ln F$ belongs to $H^\infty_\alpha(\Sigma,\C)$.
The same holds for $h$.
It follows that $\ln F'$ belongs to $H^\infty_\alpha(\Sigma,\C)$ (\autoref{res: ln of Hölder bounded}).
In particular, $F'$ satisfies the same assumptions as $F$.
We write $\mathcal L' \colon H^\infty_\alpha(\Sigma,\C) \to H^\infty_\alpha(\Sigma,\C)$ and $\mathcal L'_\lambda \colon H^\infty_\alpha(\Sigma,E) \to H^\infty_\alpha(\Sigma,E)$ for the corresponding usual and twisted transfer operators.

\paragraph{Comparing the operators.}
Let $\Phi \in H^\infty_\alpha(\Sigma, E)$.
We observe that for every $x \in \Sigma$
\begin{displaymath}
	\mathcal L'_\lambda \Phi(x) 
	= \sum_{\sigma y = x}\frac{h(y)}{\rho h(x)}F(y)\lambda(y)^{-1}\Phi(y)
	= \frac 1{\rho h(x)} \mathcal L_\lambda(h\Phi)(x).
\end{displaymath}
In particular, since $h$ is a positive eigenvector of $\mathcal L$ for the eigenvalue $\rho$, we get that $\mathcal L' \mathbb 1 = \mathbb 1$.
It follows from (\ref{eqn: computing spec radius}) combined with \autoref{res: perron-frobenius} that the spectral radius of $\mathcal L'$ is $1$.
Let $\Phi \in H^\infty_\alpha(\Sigma,E)$.
A proof by induction shows that for every $n \in \N$, 
\begin{equation}
\label{eqn: changing potential iterated operator}
	\rho^n h {\mathcal L'_\lambda}^n(\Phi) = \mathcal L_\lambda^n(h\Phi).
\end{equation}
Since $h$ is a positive continuous map on a compact set, there exists $m, M \in \R_+^*$ such that for every $x \in \Sigma$, we have $m \leq h(x) \leq M$.
Hence $1/h$ also belongs to $H^\infty_\alpha(\Sigma, \C)$.
Combining (\ref{eqn: changing potential iterated operator}) with \autoref{res: multiplication by scalar map} we observe that there exists $A_1, A_2 \in \R_+^*$ such that for every $n \in \N$,
\begin{displaymath}
	A_1 \norm[\infty,\alpha]{\mathcal L_\lambda^n}
	\leq \rho^n\norm[\infty,\alpha]{\mathcal {L'_\lambda}^n}
	\leq A_2\norm[\infty,\alpha]{\mathcal L_\lambda^n}.
\end{displaymath}
Hence the spectral radius of $\mathcal L'_\lambda$ is $\rho_\lambda/\rho$.

%
\subsection{Invariant and almost invariant vectors}
%

In this section we suppose that the system $(\Sigma,\sigma)$ is topologically transitive.
In addition we assume that $\mathcal L$ has spectral radius $1$ and fixes $\mathbb 1$.
Under these assumptions, the operator norm of $\mathcal L_\lambda \colon \mathcal C(\Sigma,E) \to \mathcal C(\Sigma,E)$ is at most $1$ (\autoref{res: transfer bounded op in infty norm}).
Similarly \autoref{res: holder bounded variation inv by transfer} yields

\begin{prop}
\label{res: holder bounded variation inv by transfer - simplified version}
	There exists $C_1 \in \R_+^*$ such that for every linear representation $\lambda \colon G \to \isom E$ into a Banach space $(E, \normV)$, for every $\Phi \in H^\infty_\alpha(\Sigma,E)$ we have 
	\begin{equation}
	\label{eqn: holder bounded variation inv by transfer - simplified version}
		\Delta_\alpha(\mathcal L_\lambda \Phi) \leq e^{-\alpha}\Delta_\alpha(\Phi) + C_1\norm[\infty]\Phi,
	\end{equation}
\end{prop}

We would like to understand the behavior $\mathcal L_\lambda$ when $1$ is a spectral value.
To that end we study invariant and almost invariant vectors of the twisted transfer operator.
Given $\epsilon \in \R_+^*$ we say that $\Phi \in H^\infty_\alpha(\Sigma,E)$ is an \emph{$\epsilon$-invariant vector} if
\begin{displaymath}
	\norm[\infty,\alpha]{\mathcal L_\lambda \Phi - \Phi } < \epsilon\norm[\infty, \alpha] \Phi.
\end{displaymath}

\paragraph{Almost invariant vectors.}
We start with some preliminary properties of almost invariants vectors for the twisted transfer operator.

\begin{prop}
\label{res: almost invariant vector - holder variations}
		There exists $D_2, \eta \in \R_+^*$ with the following property.
		For every linear representation $\lambda \colon G \to \isom E$ into a Banach space $(E, \normV)$, if $\Phi \in H^\infty_\alpha(\Sigma,E)$ is an $\eta$-invariant vector of $\mathcal L_\lambda$, then $\Delta_\alpha(\Phi) \leq D_2 \norm[\infty] \Phi$.
\end{prop}

\begin{proof}
	Let $C_1$ be the parameter given by \autoref{res: holder bounded variation inv by transfer - simplified version}.
	We fix $\eta \in \R_+^*$ such that $\eta <1-e^{-\alpha} $.
	Let $\lambda \colon G \to \isom E$ be a linear representation into a Banach space $(E, \normV)$.
	Let $\Phi \in H^\infty_\alpha(\Sigma,E)$ be an $\eta$-invariant vector of $\mathcal L_\lambda$.
	The triangle inequality combined with (\ref{eqn: holder bounded variation inv by transfer - simplified version}) yields
	\begin{displaymath}
		\Delta_\alpha(\Phi) 
		\leq \Delta_\alpha(\mathcal L_\lambda \Phi) + \Delta_\alpha(\mathcal L_\lambda \Phi - \Phi)
		\leq e^{-\alpha} \Delta_\alpha(\Phi) + C_1 \norm[\infty] \Phi + \eta \norm[\infty, \alpha] \Phi.
	\end{displaymath}
	Hence 
	\begin{displaymath}	
		\Delta_\alpha(\Phi) \leq \frac{C_1 + \eta}{1 - \left(e^{-\alpha} + \eta\right)} \norm[\infty] \Phi. \qedhere
	\end{displaymath}
\end{proof}

\paragraph{Invariants vectors.}
We now detail the behavior of invariant vectors for the twisted transfer operator.

\begin{lemm}
\label{res: eigenvector cst norm}
	Let $\Phi \in H^\infty_\alpha(\Sigma,E)$.
	If $\mathcal L_\lambda \Phi = \Phi$, then the map $\Sigma \to \R$ sending $x$ to $\norm{\Phi(x)}$ is constant.
\end{lemm}

\begin{proof}
	We denote by $\Psi \colon \Sigma \to \R$ the map defined by $\Psi(x) = \norm{\Phi(x)}$.
	It is an element of $H^\infty_\alpha(\Sigma,\R)$.
	Let $x_0\in \Sigma$ such that 
	\begin{displaymath}
		\Psi(x_0) = \sup_{x \in \Sigma} \Psi(x).
	\end{displaymath}
	Such a point exists as $\Psi$ is a continuous function on a compact set.
	Let $x \in \Sigma$.
	Let $\epsilon \in \R_+^*$.
	Since the system $(\Sigma,\sigma)$ is topologically transitive, there exists $n \in \N$ and $w_0 \in \mathcal W^n$ such that $w_0x_0$ belongs to $\Sigma$ and $\dist x{w_0x_0} \leq \epsilon$.
	According to the triangle inequality we have
	\begin{displaymath}
		\Psi(x_0) \leq \mathcal L^n\Psi(x_0) \leq \sum_{w \in \mathcal W^n}\mathbb1_{\sigma^n[w]}(x_0)F_n(wx_0)\Psi(wx_0).
	\end{displaymath}
	Recall that $\mathcal L^n\mathbb 1 = \mathbb 1$.
	Hence the right hand side is a convex combination of terms of the form $\Psi(wx_0)$, all of them being bounded above by $\Psi(x_0)$.
	Consequently $\Psi(w_0x_0) = \Psi(x_0)$.
	Since $\Psi$ has bounded variations we get
	\begin{displaymath}
		\norm{\Psi(x)- \Psi(x_0)} 
		= \norm{\Psi(x)- \Psi(w_0x_0)}
		\leq \Delta_\alpha(\Psi)d(x,w_0x_0)^\alpha \leq \Delta_\alpha(\Psi)\epsilon^\alpha.
	\end{displaymath}
	This inequality holds for every sufficiently small positive $\epsilon$ hence $\Psi(x) = \Psi(x_0)$.
	This proves that $\Psi$ is a constant function equal to $\Psi(x_0)$.
\end{proof}

\begin{lemm}
\label{res: eigenvector cocycle relation}
	Assume that $E$ is strictly convex.
	Let $\Phi \in H^\infty_\alpha(\Sigma,E)$.
	If $\mathcal L_\lambda \Phi = \Phi$, then for every $x \in \Sigma$, for every $n \in \N$, we have 
	\begin{displaymath}
		\lambda_n(x) \Phi(\sigma^nx) = \Phi(x).
	\end{displaymath}
\end{lemm}

\begin{proof}
	It follows from \autoref{res: eigenvector cst norm} that for every $x \in \Sigma$, we have $\norm{\Phi(x)} = \norm[\infty]\Phi$.
	Let $x \in \Sigma$ and $n \in \N$.
	We let $z = \sigma^nx$.
	Observe that 
	\begin{displaymath}
		\Phi(z) = \mathcal L_\lambda^n\Phi(z) = \sum_{\sigma^ny = z}F_n(y)\lambda_n(y)^{-1}\Phi(y).
	\end{displaymath}
	Recall that $\mathcal L^n\mathbb 1 = \mathbb 1$.
	Hence the right hand side is a convex combination of vectors of the form $\lambda_n(y)^{-1}\Phi(y)$.
	Since the image of $\lambda$ is contained in the isometry group of $E$, their norm is the same as the one of $\Phi(z)$, namely $\norm[\infty]\Phi$.
	The space $E$ being strictly convex, we get that $\Phi(z) = \lambda_n(y)^{-1}\Phi(y)$, for every $y \in \sigma^{-n}(\{z\})$.
	This holds in particular for $y = x$, hence the result.
\end{proof}

\paragraph{From the transfer operator to the representation.}
The goal is now to prove that if $\mathcal L$ and $\mathcal L_\lambda$ have the same spectral radius, then $\lambda$ admits almost invariant vectors. 
We first cover the case when $1$ is an \emph{eigenvalue} of $\mathcal L_\lambda$ (\autoref{res: inv vect rep - baby case}).
In this situation we combine a convexity argument taking place in $E$ with the visibility property to prove the existence of a non-zero vector $\phi_0 \in E$ that is fixed by a finite index subgroup $G_0$ of $G$.
The second step deals with the general situation, i.e. when $\mathcal L_\lambda$ admits almost invariant vectors (\autoref{res: inv vect rep - op vs rep}).
The proof of this proposition is by contradiction.
Negating the statement provides a family of counterexamples.
Then, using an ultra-limit argument, we are able two produce a new twisted transfer operator for which $1$ is an eigenvalue, therefore reducing the general case to the previous one.

\begin{prop}
\label{res: inv vect rep - baby case}
	We assume that the extension of $(\Sigma, \sigma)$ by $\theta$ has the visibility property.
	There exists a finite index subgroup $G_0$ of $G$ such that for every representation $\lambda \colon G \to \isom E$ into a uniformly convex Banach space $(E, \normV)$ the following holds.
	If $1$ is an eigenvalue of the twisted transfer operator $\mathcal L_\lambda$, then the representation $\lambda$ restricted to $G_0$ has a non-zero invariant vector $\phi_0 \in E$.
\end{prop}

\begin{proof}
	We start the proof by introducing a few auxiliary objects that will lead to the definition of $G_0$.
	The most important point is that these objects do not involve the representation of $G$.
	Let $D_2 \in \R$ be the constant given by \autoref{res: almost invariant vector - holder variations}.
	We fix an integer $m \in \N$ such that $e^{-m\alpha}D_2 < 1/2$.
	Up to increasing the value of $m$ we can assume that $\theta \colon \Sigma \to G$ is constant when restricted to any cylinder of length $m$.
	Since $(\Sigma,\sigma)$ is an irreducible subshift of finite type, there exists a finite subset $\mathcal W_0 \subset \mathcal W$ with the following property: for every $w,w' \in \mathcal W$, there exists $w_0 \in \mathcal W_0$ such that $ww_0w'$ is admissible.
	We denote by $N$ the length of the longest word in $\mathcal W_0$.
	As $\theta$ is locally constant, the set
	\begin{displaymath}
		U = \set{\theta_k(x)}{x \in \Sigma,\ k \in \intvald 0{m+N}}
	\end{displaymath}
	is finite.
	It follows from the visibility property, there exists a finite subset $U'$ of $G$ with the following property:
	for every $g \in G$, there exists two elements $u'_1,u'_2 \in U'$, a point $x \in \Sigma$, and  an integer $n \in \N$, satisfying $g = u'_1\theta_n(x)u'_2$.
	Finally, we let
	\begin{displaymath}
		K = \card U^2\card{U'}^2.
	\end{displaymath}
	Since $G$ is finitely generated, it has only finitely many finite index subgroups whose index does not exceed $K$.
	We define $G_0$ as the intersection of these subgroups.
	It is a finite index subgroup of $G$.

	\medskip
	Let us now fix a representation $\lambda \colon G \to \isom E$ into a uniformly convex Banach space $(E, \normV)$  such that $1$ is an eigenvalue of the twisted transfer operator $\mathcal L_\lambda$.
	We choose a non-zero eigenvector $\Phi \in H^\infty_\alpha(\Sigma,E)$ of $\mathcal L_\lambda$, i.e. $\mathcal L_\lambda \Phi = \Phi$.
	It follows that the map $x \to \norm{\Phi(x)}$ is constant (\autoref{res: eigenvector cst norm}) and $\Delta_\alpha(\Phi) \leq D_2 \norm[\infty]\Phi$ (\autoref{res: almost invariant vector - holder variations}).
	
	\medskip
	We write $v_1, \dots , v_\ell$ for the collection of admissible words of length $m$.
	Let $i \in \intvald 1\ell$.
	We denote by $C_i$ the closure of the convex hull of $\Phi([v_i])$.
	Since $E$ is uniformly convex, the zero vector $0 \in E$ admits a unique projection on $C_i$ that we denote by $\phi_i$.
	It follows from our choice of $m$, that for every $x,x' \in [v_i]$, we have
	\begin{displaymath}
		\norm{\Phi(x)- \Phi(x')} \leq \Delta_\alpha(\Phi)e^{-m\alpha} < \frac 12 \norm{\Phi(x)}.
	\end{displaymath}
	Consequently $C_i$ does not contains the zero vector and therefore $\phi_i$ is non-zero.
	We define $H_i$ as the pre-image by $\lambda \colon G \to \isom E$ of the stabilizer of $\phi_i$.
	We are going to show that $H_i$ is a finite index subgroup of $G$.
	To that end we start by proving the following lemma.
	
	\begin{lemm}
	\label{res: inv vect rep - baby case weak}
		For every $x \in \Sigma$, for every $n \in \N$, there exists two elements $u_1,u_2 \in U$ such that  $u_1\theta_n(x)u_2\in H_i$.
	\end{lemm}
	
	\begin{proof}
		Let $x \in \Sigma$ and $n \in \N$.
		For simplicity we let $g = \theta_n(x)$.
		Let $w$ be the prefix of length $n + m$ of $x$.
		Since $\theta \colon \Sigma \to G$ is constant when restricted on any cylinder of length $m$, we note that $\theta_n(x') = \theta_n(x)$ for every $x' \in [w]$.
		By the very definition of $\mathcal W_0$, there exists $w_1, w_2 \in \mathcal W_0$ such that $v_iw_1ww_2v_i$ is admissible.
		We denote by $p$ and $q$ the length of $v_iw_1$ and $ww_2$ respectively and observe that $p \leq m +N$ and $q \leq n + m + N$.
		
		\medskip
		We now claim that there exist $u_1,u_2\in U$ such that for every $y_1 \in [v_i]$, there exists $y_2 \in [v_i]$ satisfying
		\begin{displaymath}
			\lambda(u_1gu_2)\Phi(y_1) = \Phi(y_2).
		\end{displaymath}
		Choose $y_1 \in [v_i]$.
		It follows from our choice of $w_1$ and $w_2$ that $y_2 = v_iw_1ww_2y_1$ and $z = ww_2y_1$ are two well defined points of $\Sigma$.
		Moreover $z = \sigma^py_2$ and $y_1 = \sigma^qz$.
		According to \autoref{res: eigenvector cocycle relation} we have
		\begin{displaymath}
			\lambda_{p+q}(y_2)\Phi(y_1) = \lambda_{p+q}(y_2)\Phi(\sigma^{p+q}y_2) = \Phi(y_2).
		\end{displaymath}
		Note that $y_2$ belongs to $[v_i]$.
		Hence it suffices to show that $\theta_{p+q}(y_2)$ can be written $u_1gu_2$ where $u_1$ and $u_2$ do not depend on $y_1$.
		The cocycle relation that $\theta_{p+q}$ satisfies yields
		\begin{displaymath}
			\theta_{p+q}(y_2) = \theta_p(y_2)\theta_n(\sigma^py_2)\theta_{q-n}(\sigma^{p+n}y_2)
			= \theta_p(y_2)\theta_n(z)\theta_{q-n}(\sigma^nz).
		\end{displaymath}
		By construction $z$ belongs to $[w]$, hence $\theta_n(z) = \theta_n(x) = g$.
		Observe that $p$ and $q-n$ are bounded above by $m+N$.
		Hence $u_1 = \theta_p(y_2)$ and $u_2= \theta_{q-n}(\sigma^nz)$ belong to $U$.
		The proof of the claim will be complete if we can prove that $u_1$ and $u_2$ do not depend on $y_1$.
		Consider $y'_1$ another element of $[v_i]$ and let as previously $y'_2 = v_iw_1ww_2y'_1$ and $z = ww_2y'$.
		We observe that $z$ and $z'$ belong to the same cylinder, namely $[ww_2v_i]$ whose length is bounded below by $q+m$.
		Hence $\sigma^nz$ and $\sigma^nz'$ coincide on the first $q-n+m$ letters.
		Since $\theta \colon \Sigma \to G$ is constant on any cylinder of length $m$, we get $\theta_{q-n}(\sigma^nz) = \theta_{q-n}(\sigma^nz')$.
		The same argument shows that $\theta_p(y_2) = \theta_p(y'_2)$, which completes the proof of our claim.
		
		\medskip
		It follows from the claim that $\lambda(u_1gu_2)$ maps $C_i$ into itself as well.
		In particular, $\lambda(u_1gu_2)\phi_i$ belongs to $C_i$.
		However $\lambda(u_1gu_2)$ being a linear isometry we have $\norm{\lambda(u_1gu_2)\phi_i} = \norm{\phi_i}$.
		Recall that we defined $\phi_i$ as the projection of $0$ onto $C_i$.
		By unicity of the projection ($E$ is uniformly convex) we get $\lambda(u_1gu_2)\phi_i = \phi_i$.
		In other words $u_1gu_2 = u_1\theta_n(x)u_2$ belongs to $H_i$ which completes the proof of the lemma.	
	\end{proof}

	As the extension of $(\Sigma,\sigma)$ by $\theta$ satisfies the visibility property, the previous lemma has the following consequence: 	
	\begin{displaymath}
		G 
		= \bigcup_{u_1,u_2 \in U, u'_1,u'_2 \in U'} u'_1u_1^{-1}H_iu_2^{-1}u'_2
		= \bigcup_{u_1,u_2 \in U, u'_1,u'_2 \in U'} \left[\left(u'_1u_1^{-1}\right)H_i\left(u'_1u_1^{-1}\right)^{-1}\right] \left(u'_1u_1^{-1}u_2^{-1}u'_2\right).
	\end{displaymath}
	In other words $G$ can be covered by finitely many cosets of conjugates of $H_i$.
	Moreover the number of these cosets is bounded above by
	\begin{displaymath}
		K = \card U^2\card{U'}^2.
	\end{displaymath}
	According to \cite[Lemma~4.1]{Neumann:1954wx}, $H_i$ is a finite index subgroup of $G$ and $[G:H_i] \leq K$.
	It follows from its definition that $G_0$ is a subgroup of $H_i$, thus the representation $\lambda$ restricted to $G_0$ has a non-zero invariant vector, namely $\phi_i$. 
\end{proof}

\begin{defi}
\label{def: unif cvx sequence of banach}
	We say that a collection $\mathcal E$ of Banach spaces is \emph{uniformly convex} if for every $\epsilon >0$ there exits $\eta>0$ such that for every space $(E, \normV)$ of $\mathcal E$, for every unit vectors $\phi, \phi' \in E$, if $\norm{\phi - \phi'} \geq \epsilon$, then $\norm{\phi + \phi'} \leq 2(1 - \eta)$.
\end{defi}

\paragraph{Remark.}
This definition not only asks that each space $E \in \mathcal E$ is uniformly convex, but also that the parameters quantifying their rotundity work for all spaces simultaneously.

\begin{defi}
	Let $\lambda \colon G \to \isom E$ be a representation of $G$ into a Banach space.
	Let $S$ be a finite subset of $G$ and $\epsilon \in \R_+^*$.
	A vector $\phi \in E$ is \emph{$(S,\epsilon)$-invariant} (with respect to $\lambda$) if 
	\begin{displaymath}
		\sup_{s \in S} \norm{\lambda(s)\phi - \phi} < \epsilon \norm{\phi}.
	\end{displaymath}
	The representation $\lambda \colon G \to \isom E$ \emph{almost has invariant vectors} if for every finite subset $S$ of $G$, for every $\epsilon \in \R_+^*$ there exists an $(S, \epsilon)$-invariant vector.
\end{defi}

\begin{prop}
\label{res: inv vect rep - op vs rep}
	We assume that the extension of $(\Sigma, \sigma)$ by $\theta$ has the visibility property.
	There exists a finite index subgroup $G_0$ of $G$ with the following property.
	Let $\mathcal E$ be a uniformly convex collection of Banach spaces.
	For every finite subset $S_0$ of $G_0$, for every $\epsilon \in \R_+^*$, there exists $\eta \in \R_+^*$, such that the following holds.
	Let $\lambda \colon G \to \isom E$ be a representation of $G$ into a Banach space $(E,\normV)$ of $\mathcal E$.
	If the twisted transfer operator $\mathcal L_\lambda$ has an $\eta$-invariant vector, then the representation $\lambda$ admits an $(S_0, \epsilon)$-invariant vector.
\end{prop}

Before giving the proof, we recall some useful material regarding ultra-limit of Banach spaces.

\paragraph{Ultra-limit of Banach spaces.}
Let $\omega \colon \mathcal P(\N) \to \{0,1\}$ be a \emph{non-principal ultra-filter}, i.e. a finitely additive map which vanishes on any finite subset of $\N$ and such that $\omega(\N) = 1$.
We say that a property $\mathcal P_n$ is true \oas if $\omega(\set{n \in \N}{\mathcal P_n\ \text{holds}}) = 1$.
A real sequence $(u_n)$ is \oeb if there exists $M \in \R$ such that $\abs{u_n} \leq M$ \oas.
We say that the $\omega$-limit of $(u_n)$ is $\ell \in \R$ and write $\limo u_n = \ell$, if for every $\epsilon > 0$, we have $\dist{u_n}{\ell} < \epsilon$ \oas.
Any real sequence which is \oeb admits an $\omega$-limit \cite{Bou71}.

\medskip
Let $(E_n, \normV)$ be a sequence of Banach spaces.
We define the ultra-product $\Pi_\omega E_n$ as the set
\begin{displaymath}
	\Pi_\omega E_n = 
	\set{(\phi_n) \in \Pi_\N E_n}{ \norm{\phi_n} \text{\oeb}}.
\end{displaymath}
We endow $\Pi_\omega E_n$ with the following equivalence relation: $(\phi_n)\equiv(\phi'_n)$ if $\limo \norm{\phi_n - \phi'_n} = 0$.

\begin{defi}
\label{def: ultralimit of banach}
	The $\omega$-limit of the sequence $(E_n, \normV)$ that we denote $\limo E_n$ is the quotient of $\Pi_\omega E_n$ by the equivalence relation $\equiv$.
\end{defi}

If $(\phi_n)$ is an element of $\Pi_\omega E_n$ we write $\limo \phi_n$ for its image in $\limo E_n$.
The set $\limo E_n$ naturally comes with a structure of vector space characterized as follows.
\begin{enumerate}
	\item $(\limo \phi_n) + (\limo \phi'_n) = \limo (\phi_n + \phi'_n)$, for every $(\phi_n), (\phi'_n) \in \Pi_\omega E_n$;
	\item $c (\limo \phi_n) = \limo (c\phi_n)$, for every $(\phi_n) \in \Pi_\omega E_n$ and $c \in \C$.
\end{enumerate}
The space $\limo E_n$ carries also a norm characterized by 
\begin{displaymath}
	\norm{\limo \phi_n} = \limo \norm{\phi_n},\quad  \forall (\phi_n) \in \Pi_\omega E_n.
\end{displaymath}
One can check that $(\limo E_n, \normV)$ is a Banach space.
The next lemma is a straightforward exercise.

\begin{lemm}
\label{res: limit of unif cvx banach}
	If $(E_n, \normV)$ is a uniformly convex collection of Banach spaces, then $(\limo E_n, \normV)$ is uniformly convex.
\end{lemm}

\begin{proof}[Proof of \autoref{res: inv vect rep - op vs rep}]
	We denote by $G_0$ the finite index subgroup of $G$ given by \autoref{res: inv vect rep - baby case}.
	We denote by $D_2$ and $\eta$ the parameters given by \autoref{res: almost invariant vector - holder variations}.
	Let $\mathcal E$ be a uniformly convex collection of Banach spaces.

	\paragraph{A family of counterexamples.}
	Let $S_0$ be a finite subset of $G_0$ and $\epsilon \in \R_+^*$.
	Let $(\eta_n)$ be a sequence of positive real numbers converging to zero.
	Assume that the proposition is false.
	It means that for every $n \in \N$, there exists a representation $\lambda_n \colon G \to \isom{E_n}$ where $(E_n, \normV)$ belongs to $\mathcal E$, with the following properties:
	\begin{enumerate}
		\item the twisted transfer operator $\mathcal L_{\lambda_n}$ has an $\eta_n$-invariant vector $\Phi_n\in H^\infty_\alpha(\Sigma,E_n)$;
		\item the representation $\lambda_n$ does not admit any $(S_0, \epsilon)$-invariant vector.
	\end{enumerate}
	In the remainder of the proof,  we write for simplicity $\mathcal L_n$ instead of $\mathcal L_{\lambda_n}$.	
	
	\paragraph{Almost invariant vectors for the twisted operator.}
	Without loss of generality we can assume that $\norm[\infty]{\Phi_n} = 1$, for every $n \in \N$.
	Recall that $D_2,\eta \in \R_+^*$ are the parameter given by \autoref{res: almost invariant vector - holder variations}.
	Since $(\eta_n)$ converges to $0$, up to throwing away the first terms of the sequence, we can assume that $\eta_n \leq \eta$ for every $n \in \N$.
	It follows now from \autoref{res: almost invariant vector - holder variations} that for every $n \in \N$, we have
	\begin{equation}
	\label{eqn: inv vect rep - gal case - holder}
		\Delta_\alpha(\Phi_n) \leq D_2 \norm[\infty]{\Phi_n} \leq D_2.
	\end{equation}
	We now fix $k \in \N$ such that $e^{-k\alpha} D_2 < 1/2$.
	We write $v_1, \dots , v_\ell$ for the collection of all admissible words of length $k$.
	Up to passing to a subsequence we may assume that there exists $i \in \intvald 1\ell$ such that the map $x \to \norm{\Phi_n(x)}$ achieved its maximum in $[v_i]$.
	By reordering the elements $v_1, \dots, v_\ell$ we can actually assume that $i = 1$.
	In particular it follows from (\ref{eqn: inv vect rep - gal case - holder}) and our choice of $k$, that for every $n \in \N$, for every $x \in [v_1]$, we have $\norm{\Phi_n(x)} \geq 1/2$.
	
	\paragraph{The Banach space $E_\infty$.}
	We now fix a non-principal ultra-filter $\omega$.
	The limit space $E_\infty = \limo E_n$ is a uniformly convex Banach space (\autoref{res: limit of unif cvx banach}).
	The next step is to define a representation $\lambda_\infty \colon G \to \isom {E_\infty}$.
	Given $g \in G$ and a vector $\phi = \limo \phi_n$ of $E_\infty$, we let
	\begin{displaymath}
		\lambda_\infty(g) \phi = \limo \left[\lambda_n(g)\phi_n\right].
	\end{displaymath}
	One easily checks that $\lambda_\infty(g)$ is a well-defined linear isometry of $E_\infty$.
	Moreover that map $\lambda_\infty \colon G \to \isom {E_\infty}$ obtained in this way is a homomorphism.
	In particular, one can consider the twisted transfer operator
	\begin{displaymath}
		\mathcal L_{\lambda_\infty} \colon H^\infty_\alpha(\Sigma,E_\infty) \to H^\infty_\alpha(\Sigma, E_\infty).
	\end{displaymath}
	that for simplicity we denote $\mathcal L_\infty$.

	\paragraph{The eigenvector $\Phi_\infty$.}
	We now use the sequence $(\Phi_n)$ to produce an eigenvector of $\mathcal L_\infty$.
	Recall that for every $n \in \N$, for every $x \in \Sigma$, we have $\norm{\Phi_n(x)} \leq 1$.
	Hence we can define a map $\Phi_\infty \colon \Sigma \to E_\infty$ as follows
	\begin{displaymath}
		\Phi_\infty(x) = \limo \Phi_n(x), \quad \forall x \in \Sigma.
	\end{displaymath}
	Note that $\Phi_\infty$ is bounded. 
	More precisely, $\norm[\infty]{\Phi_\infty} \leq 1$.
	Recall that for every $\Delta_\alpha(\Phi_n) \leq D_2$ for every $n \in \N$.
	It directly follows that $\Delta_\alpha(\Phi_\infty) \leq D_2$.
	Hence $\Phi_\infty$ belongs to $H^\infty_\alpha(\Sigma,E_\infty)$.
	We also observe that $\Phi_\infty$ is non-trivial.
	Indeed by construction $\norm{\Phi_n(x)} \geq 1/2$, for every $n \in \N$, for every $x \in [v_1]$.
	It follows that $\Phi_\infty$ restricted to $[v_1]$ does not vanish.
	Finally, we claim that $\mathcal L_{\lambda_\infty} \Phi_\infty = \Phi_\infty$.
	Let $x \in \Sigma$.
	Since $\mathcal L\mathbb 1 = \mathbb 1$ we can write 
	\begin{displaymath}
		\mathcal L_n\Phi_n(x) - \Phi_n(x)
		= \sum_{\sigma y=x}  F(y) \left[\fantomB\;\lambda_n(y)^{-1}\Phi_n(y) - \Phi_n(x) \right], \quad \forall n \in \N
	\end{displaymath}
	and
	\begin{displaymath}
		\mathcal L_\infty\Phi_{\infty}(x) - \Phi_\infty(x)
		= \sum_{\sigma y=x} F(y) \left[\fantomB\;\lambda_\infty(y)^{-1}\Phi_\infty(y) - \Phi_\infty(x) \right].
	\end{displaymath}
	It follows from the definition of $\lambda_\infty$ and $\Phi_\infty$ that 
	\begin{displaymath}
		\mathcal L_\infty\Phi_{\infty}(x) - \Phi_\infty(x) = \limo \left[\fantomB\; \mathcal L_n\Phi_n(x) - \Phi_n(x)\right] = 0,
	\end{displaymath}
	which completes the proof of our claim.

	\paragraph{Almost invariant vector for $\lambda_n$.}
	The previous discussion shows that $1$ is an eigenvalue of $\mathcal L_\infty \colon H^\infty_\alpha(\Sigma, E_\infty) \to H^\infty_\alpha(\Sigma, E_\infty)$.
	It follows from \autoref{res: inv vect rep - baby case} that the limit representation $\lambda_\infty \colon G \to \isom{E_\infty}$ restricted to the finite index subgroup $G_0$ admits a non-zero invariant $\phi_\infty$.
	Such a vector can be written $\phi_\infty = \limo \phi_n$, where $(\phi_n) \in \Pi_\omega E_n$.
	Since $\phi_\infty$ is non zero, we can assume without loss of generality that $\norm{\phi_n} = 1$, for every $n \in \N$.
	Since $S_0$ is contained in $G_0$, for every $g_0 \in S_0$,  we have
	\begin{displaymath}
		\limo \left[\lambda_n(g_0)\phi_n\right] 
		= \lambda_\infty (g_0)\phi_\infty
		= \phi_\infty
		= \limo \phi_n.
	\end{displaymath}
	The set $S_0$ being finite, the vector $\phi_n$ is an $(S_0,\epsilon)$-invariant vector (with respect to $\lambda_n$) \oas.
	This contradicts our initial assumption and completes the proof of the proposition.
\end{proof}

\begin{coro}
\label{res: inv vect rep - gal case}
	We assume that the extension of $(\Sigma, \sigma)$ by $\theta$ has the visibility property.
	There exists a finite index subgroup $G_0$ of $G$ with the following property.
	Let $\mathcal E$ be a uniformly convex collection of Banach spaces.
	For every finite subset $S_0$ of $G_0$, for every $\epsilon \in \R_+^*$ there exists $\eta \in \R_+^*$ such that the following holds.
	Let $\lambda \colon G \to \isom E$ be a representation of $G$ into a Banach space $(E,\normV)$ of $\mathcal E$.
	If the spectral radius $\rho_\lambda$ of the twisted transfer operator $\mathcal L_\lambda$ satisfies $\rho_\lambda > 1 - \eta$, then there exists $\beta \in \R$ such that the representation $e^{i\beta}\lambda$ admits an $(S_0, \epsilon)$-invariant vector.
\end{coro}

\begin{proof}
	Let	$\lambda \colon G \to \isom E$ be a representation of $G$ into a Banach space $(E,\normV)$.
	Let $\eta >0$.
	We claim that if $\rho_\lambda > 1-\eta$, then there exists $\beta \in \R$, such that the operator $e^{i\beta}\mathcal L_\lambda$ admits an $\eta$-invariant vector.
	Since $\rho_\lambda$ is the spectral radius of $\mathcal L_\lambda$, there exists $\beta \in \R$, such that $\rho_\lambda e^{-i\beta}$ is a point in the boundary of $\spec{\mathcal L_\lambda}$.
	According to \cite[Proposition~6.7]{Conway:1985gw}, there exists $\Phi \in H^\infty_\alpha(\Sigma,E)$ such that 
	\begin{equation}
	\label{eqn: main technical theorem - almost eigenvectors}
		\norm[\infty,\alpha]{\mathcal L_\lambda\Phi - \rho_\lambda e^{-i\beta}\Phi} < (\rho_\lambda -1+\eta) \norm[\infty, \alpha]\Phi.
	\end{equation}
	Combined with the triangle inequality, it yields
	\begin{align*}
		\norm[\infty,\alpha]{e^{i\beta}\mathcal L_\lambda\Phi - \Phi}
		& \leq  \norm[\infty,\alpha]{\mathcal L_\lambda\Phi_n - \rho_\lambda e^{-i\beta}\Phi} + \norm[\infty,\alpha]{\rho_\lambda\Phi -\Phi } \\
		& <  (\rho_\lambda -1+\eta) \norm[\infty, \alpha]\Phi + (1-\rho_\lambda) \norm[\infty, \alpha]\Phi.
	\end{align*}
	Hence $\norm[\infty,\alpha]{e^{i\beta}\mathcal L_\lambda\Phi - \Phi} < \eta  \norm[\infty, \alpha]\Phi$, which completes the proof of our claim.
	Observe that the operator $e^{i\beta}\mathcal L_\lambda$ can be seen as the twisted transfer operator $\mathcal L_{\lambda'}$ associated to the representation $\lambda' \colon G \to \isom E$ defined by $\lambda'(g)= e^{i\beta}\lambda(g)$.
	The corollary is now a direct consequence of \autoref{res: inv vect rep - op vs rep}
\end{proof}

%
\subsection{Amenability and Kazhdan property (T)}
%

In this section we focus on representations induced by a group actions.
Let $Y$ be a set.
Let $\mathcal H = \ell^2(Y)$ be the set of functions $\phi \colon Y \to \C$ which are square summable.
It carries a natural structure of Hilbert space.
A vector $\phi \in \mathcal H$ is \emph{non-negative}, if $\phi(y) \in \R_+$ for every $y \in Y$.
Given any vector $\phi \in \mathcal H$, we defined its \emph{modulus} to be the vector $\abs{\phi} \in \mathcal H$ defined by $\abs{\phi}(y) = \abs{\phi(y)}$ for every $y \in Y$.
Observe that $\norm{ \abs \phi} = \norm {\phi}$.

\medskip
Let $G$ be a group acting on $Y$.
The action of $G$ induces a unitary representation $\lambda \colon G \to \mathcal U(\mathcal H)$ defined as follows: for every $g \in G$, for every $\phi \in \mathcal H$, 
\begin{displaymath}	
	\left[\lambda(g)\phi\right](y) = \phi(g^{-1}y), \quad \forall y \in Y.
\end{displaymath}
Observe that for every $\phi \in \mathcal H$, for every $g \in G$ we have $\abs{\lambda(g)\phi} = \lambda(g)\abs{\phi}$.

\begin{lemm}
\label{res: modulus vs inv vectors}
	Let $Y$ be a metric space endowed with an action of $G$.
	Let $\mathcal H = \ell^2(Y)$ and $\lambda \colon G \to \mathcal U(\mathcal H)$ be the unitary representation induced by the action of $G$.
	Let $\beta \in [0, 2\pi)$.
	Let $S$ be a finite subset of $G$ and $\epsilon \in \R_+^*$.
	If $\phi \in \mathcal H$ is $(S,\epsilon)$-invariant with respect to $e^{i\beta}\lambda$, then $\abs{\phi}$ is  \emph{non-negative} and $(S,\epsilon)$-invariant with respect to $\lambda$.
\end{lemm}

\begin{proof}
	Given any two vectors $\phi_1,\phi_2 \in \mathcal H$, one checks easily that their modulus satisfies
	\begin{displaymath}
		\norm{\abs{\phi_1} - \abs{\phi_2}} \leq \norm{\phi_1 -\phi_2}.
	\end{displaymath}
	Let $\phi \in \mathcal H$, be an $(S,\epsilon)$-invariant with respect to $e^{i\beta}\lambda$
	Combining our various observations on modulus vector, we get that for every $g \in S$,
	\begin{displaymath}
		\norm{\lambda(g)\abs{\phi} - \abs{\phi}}
		= \norm{\abs{e^{i\beta}\lambda(g)\phi} - \abs{\phi}}
		\leq \norm{e^{i\beta}\lambda(g)\phi - \phi} 
		\leq \epsilon \norm \phi
		= \epsilon \norm{\abs \phi}.
	\end{displaymath}
	Hence $\abs{\phi}$ is a non-negative $(S,\epsilon)$-invariant with respect to $\lambda$
\end{proof}

The main result of this section is the following statement.

\begin{theo}
\label{res: main technical theorem}
	Let $(\Sigma,\sigma)$ be an irreducible subshift of finite type.
	Let $F \colon \Sigma \to \R_+^*$ be a potential with $\alpha$-bounded Hölder variations for some $\alpha \in \R_+^*$.
	Let $\mathcal L$ be the corresponding transfer operator and $\rho$ its spectral radius.
	Let $G$ be a finitely generated group and $\theta \colon \Sigma \to G$ be a locally constant map.
	We assume that the corresponding extension $(\Sigma_\theta,\sigma_\theta)$ has the visibility property.	
	For every finite subset $S$ of $G$ and every $\epsilon \in \R_+^*$ there exists $\eta \in \R_+^*$ with the following property.
	
	\medskip
	Let $Y$ be a set endowed with an action of $G$ and $\lambda \colon G \to \mathcal U(\mathcal H)$ be the induced unitary representation, where $\mathcal H = \ell^2(Y)$.
	Let $\rho_\lambda$ be the spectral radius of the twisted transfer operator $\mathcal L_\lambda \colon H^\infty_\alpha(\Sigma, \mathcal H) \to H^\infty_\alpha(\Sigma,\mathcal H)$.
	If $\rho_\lambda > (1 - \eta)\rho$, then the representation $\lambda$ admits an $(S,\epsilon)$ invariant vector.
\end{theo}

\begin{proof}[Proof of \autoref{res: main technical theorem}]
	The strategy of the proof is the following. 
	First we renormalize the potential so that we can assume that the spectral radius of $\mathcal L$ is $\rho = 1$ and $\mathbb 1$ is an invariant vector of $\mathcal L$.
	Applying \autoref{res: inv vect rep - gal case} we get a finite index subgroup $G_0$ of $G$ such that the representation $\lambda$ when restricted to $G_0$ admits a certain almost invariant vector $\phi$.
	We finally take advantage of the structure of $\ell^2(Y)$ to average the orbit of $\phi$, and thus get an almost invariant vector with respect to $\lambda$.

	\paragraph{Renormalization of the potential.}
	We start with a reduction argument: we claim that without loss of generality we can assume that $\mathcal L$ has spectral radius $1$ and fixes $\mathbb 1$.
	Assume indeed that the result has be proved in this context and let us explain how to deduce the general case.
	Let $h \in H^\infty_\alpha(\Sigma, \C)$ be the positive eigenvector of $\mathcal L$ given by the Ruelle Perron-Frobenius Theorem (\autoref{res: perron-frobenius}).
	Following the strategy of \autoref{sec: changing potential} we define a new potential $F' \colon \Sigma \to \R_+^*$ by 
	\begin{displaymath}
		F'(x) = \frac 1\rho \cdot \frac{h(x)}{h\circ \sigma(x)} F(x).
	\end{displaymath}
	We write $\mathcal L'$ for the corresponding transfer operator.
	As we observed $\mathcal L'$ has spectral radius $1$ and fixes $\mathbb 1$.
	It follows from our assumption that we can apply~\autoref{res: main technical theorem} to this operator.
	
	\medskip
	Let $S$ be a finite subset of $G$ and $\epsilon \in \R_+^*$.
	Let $\eta \in \R_+^*$ be the parameter given by \autoref{res: main technical theorem} (with the additional assumption that the transfer operator has spectral radius $1$ and $\mathbb 1$ as an eigenvector) applied to the potential $F'$.
	Suppose now that $Y$ is a space endowed with an action of $G$ and denote by $\lambda \colon G \to \mathcal U(\mathcal H)$ the induced unitary representation, where $\mathcal H = \ell^2(Y)$.
	Assume that the spectral radius of $\mathcal L_\lambda$ satisfies $\rho_\lambda > (1-\eta)\rho$.
	It follows from the discussion of \autoref{sec: changing potential} that the spectral radius $\rho'_\lambda$ of $\mathcal L'_\lambda$ satisfies $\rho'_\lambda = \rho_\lambda/\rho$.
	In particular $\rho'_\lambda > 1 - \eta$.
	\autoref{res: main technical theorem} applied to the potential $F'$ tells us that $\lambda$ admits an $(S,\epsilon)$ invariant vector, which completes the proof of our claim.
	
	\paragraph{Finite index subgroup with almost invariant vectors.}
	From now on we assume that $\mathcal L$ has spectral radius $1$ and fixes $\mathbb 1$.
	Let $S$ be a finite subset of $G$ and $\epsilon \in \R_+^*$.
	We denote by $G_0$ the finite index subgroup of $G$ given by \autoref{res: inv vect rep - gal case}.
	We denote by $u_1, \dots, u_m$ a set of representatives of $G/G_0$.
	For every $g \in S$, there exists a permutation $\sigma_g \colon \intvald1m \to \intvald 1m$ such that for all $i \in \intvald 1m$, we have
	\begin{displaymath}
		u_{\sigma_g(i)}^{-1}gu_i \in G_0.
	\end{displaymath}
	We now define a finite subset $S_0$ of $G_0$ as
	\begin{displaymath}
		S_0 = \set{u_{\sigma_g(i)}^{-1}gu_i}{g \in S,\ i \in \intvald 1m}.
	\end{displaymath}
	Note that the set of all Hilbert spaces is a uniformly convex collection of Banach spaces.
	According to \autoref{res: inv vect rep - gal case}, there exists $\eta \in \R_+^*$ with the following property.
	Let $\lambda \colon G \to \mathcal U(\mathcal H)$ be a unitary representation of $G$.
	If the spectral radius of the twisted transfer operator $\mathcal L_\lambda$ is larger than $1 - \eta$, then there exists $\beta \in \R$ such that the representation  $e^{i\beta}\lambda$ admits an $(S_0, \epsilon/\sqrt m)$-invariant vector.

	\paragraph{Representation induced by an action.}
	Let $Y$ be a set endowed with an action of $G$ and $\lambda \colon G \to \mathcal U(\mathcal H)$ be the induced unitary representation, where $\mathcal H = \ell^2(Y)$.
	Assume that the spectral radius $\rho_\lambda$ of the twisted transfer operator $\mathcal L_\lambda$ is at least $1 - \eta$.
	According to the very definition of $\eta$, there exists $\beta \in \R$, such that the representation $e^{i\beta}\lambda$ admits an $(S_0, \epsilon/\sqrt m)$-invariant vector $\phi$.
	It follows from \autoref{res: modulus vs inv vectors} that $\abs \phi$ is an $(S_0, \epsilon/\sqrt m)$-invariant vector with respect to the representation $\lambda$.
	We now let
	\begin{displaymath}
		\bar \phi = \frac 1m\sum_{i = 1}^m \lambda_n(u_i)\abs{\phi}.
	\end{displaymath}
	Let $g \in S$.
	The computation yields
	\begin{displaymath}
		m\lambda(g)\bar \phi 
		= \sum_{i =1}^m \lambda\left(gu_i\right)\abs{\phi} 
		= \sum_{i =1}^m \lambda\left(u_{\sigma_g(i)}\right)\lambda\left(u_{\sigma_g(i)}^{-1}gu_i\right)\abs{\phi}. 
	\end{displaymath}
	On the other hand, reindexing the sum defining $\bar \phi$ gives
	\begin{displaymath}
		m\bar \phi = \sum_{i =1}^m \lambda\left(u_{\sigma_g(i)}\right)\abs{\phi}
	\end{displaymath}
	Recall that for every $i \in \intvald 1m$, the element $u_{\sigma_g(i)}^{-1}gu_i$ belongs to $S_0$.
	Thus the triangle inequality yields
	\begin{displaymath}
		\norm{\lambda(g)\bar \phi - \bar \phi}
		\leq \frac 1m\sum_{i = 1}^m \norm{\lambda\left(u_{\sigma_g(i)}^{-1}gu_i\right)\abs{\phi} 
 - \abs{\phi}}
 < \frac \epsilon{\sqrt m}.
	\end{displaymath}
	This inequality holds for every $g \in G$.
	Observe that $\bar \phi$ is obtained by averaging non-negative vectors of $\mathcal H$ all of them having norm $1$.
	It follows that the norm of $\bar \phi$ is bounded below by $1/\sqrt m$.
	Hence $\bar \phi$ is an $(S,\epsilon)$-invariant vector with respect to $\lambda$.
\end{proof}

\paragraph{Amenability.}
There are numerous equivalent definition of amenability.
The one that is the most adapted four our purpose can be formulated in terms of regular representation.

\begin{defi}
\label{def: amenability}
	The action of $G$ on $Y$ is amenable if and only if the representation $\lambda \colon G \to \mathcal U(\mathcal H)$ admits almost invariant vectors.
	The group $G$ is \emph{amenable} if its action on itself is amenable.
\end{defi}

\begin{theo}[Amenability criterion]
\label{res: generalization stadlbauer}
	Let $(\Sigma,\sigma)$ be an irreducible subshift of finite type.
	Let $F \colon \Sigma \to \R_+^*$ be a potential with $\alpha$-bounded Hölder variations for some $\alpha \in \R_+^*$.
	Let $\mathcal L$ be the corresponding transfer operator and $\rho$ its spectral radius.
	Let $G$ be a finitely generated group and $\theta \colon \Sigma \to G$ be a locally constant map.
	We assume that the corresponding extension $(\Sigma_\theta,\sigma_\theta)$ has the visibility property.	
	Let $Y$ be a set endowed with an action of $G$ and $\lambda \colon G \to \mathcal U(\mathcal H)$ be the induced unitary representation, where $\mathcal H = \ell^2(Y)$.
	Let $\rho_\lambda$ be the spectral radius of the twisted transfer operator $\mathcal L_\lambda \colon H^\infty_\alpha(\Sigma, \mathcal H) \to H^\infty_\alpha(\Sigma,\mathcal H)$ defined by 
	\begin{displaymath}
		\mathcal L_\lambda \Phi (x) = \sum_{\sigma y = x}F(y)\lambda(y)^{-1}\Phi(y).
	\end{displaymath}
	The following statements are equivalent.
	\begin{enumerate}
		\item The action of $G$ on $Y$ is amenable.
		\item $\rho$ belongs to $\spec{\mathcal L_\lambda}$.
		\item $\rho_\lambda = \rho$.
	\end{enumerate}
\end{theo}

\begin{proof}
	Reasoning as in the beginning of the proof of \autoref{res: main technical theorem} we observe that without loss of generality we can assume that $\mathcal L$ has spectral radius $1$ and fixes $\mathbb 1$.
	We start with (ii)$\Rightarrow$(iii)
	Recall that $\rho_\lambda \leq 1$ (\autoref{res: upper bound spec radius twisted}).
	Hence if $1$ belongs to $\spec{\mathcal L_\lambda}$, then $\rho_\lambda = 1$.
	We now focus on (iii)$\Rightarrow$(i).
	Assume that $\rho_\lambda = 1$.
	It follows from \autoref{res: main technical theorem} that $\lambda$ almost admits invariants vectors.
	Hence the action of $G$ on $Y$ is amenable.
	We are left to prove (i)$\Rightarrow$(ii).
	Assume that the action of $G$ on $Y$ is amenable.
	According to \autoref{res: upper bound spec radius twisted} it is sufficient to prove that $\rho_\lambda \geq 1$.
	Let $n \in \N$.
	Let $\epsilon > 0$.
	Since $\theta \colon \Sigma \to G$ is locally constant, the set 
	\begin{displaymath}
		S = \set{\theta_n(x)}{x \in \Sigma}
	\end{displaymath}
	is finite.
	Since the action of $G$ is amenable, there exists an $(S, \epsilon)$-invariant vector $\phi \in \mathcal H\setminus\{0\}$.
	Without loss of generality we can assume that $\norm \phi = 1$.
	We define a map $\Phi \colon \Sigma \to \mathcal H$ by letting $\Phi(x) = \phi$, for every $x \in \Sigma$.
	Obviously $\norm[\infty]{\Phi} = 1$ and $\Delta_\alpha(\Phi) = 0$, hence $\Phi$ belongs to $H^\infty_\alpha(\Sigma,\mathcal H)$.
	Using the fact that $\mathcal L \mathbb 1 = \mathbb 1$, we can write for every $x \in \Sigma$,
	\begin{displaymath}
		\norm{\mathcal L_\lambda^n \Phi(x) - \Phi (x)}
		\leq \sum_{\sigma^ny = x} F_n(y) \norm{\lambda_n(y)^{-1}\Phi(y) - \Phi(x)}
		=  \sum_{\sigma^ny = x} F_n(y) \norm{\lambda_n(y)\phi - \phi}
		< \epsilon.
	\end{displaymath}
	This proves that $\norm[\infty]{\mathcal L_\lambda^n \Phi - \Phi} < \epsilon$.
	In particular, we get
	\begin{displaymath}
		\norm[\infty, \alpha]{\mathcal L_\lambda^n\Phi}
		\geq \norm[\infty]{\mathcal L_\lambda^n\Phi}
		> \norm[\infty]\Phi - \epsilon
		\geq 1 - \epsilon.
	\end{displaymath}
	Recall that $\norm[\infty, \alpha]\Phi = 1$.
	Thus we have proved that the norm of $\mathcal L_\lambda^n$ -- see as an operator of $H^\infty_\alpha(\Sigma,\mathcal H)$ -- is larger than $1 - \epsilon$.
	This holds for every $\epsilon > 0$.
	Hence for every $n \in \N$, we have
	\begin{displaymath}
		\norm[\infty, \alpha]{\mathcal L_\lambda^n} \geq 1.
	\end{displaymath}
	Consequently
	\begin{displaymath}
		\rho_\lambda 
		= \lim_{n \to \infty} \sqrt[n]{\norm[\infty, \alpha]{\mathcal L_\lambda^n}}
		\geq 1. \qedhere
	\end{displaymath}
\end{proof}

\paragraph{Kazhdan property (T).}
Let us recall first the definition of property (T).

\begin{defi}
\label{def: property T}
	A discrete group $G$ has \emph{Kazhdan property} (T), if there exists a finite subset $S$ of $G$ and $\epsilon \in \R_+^*$ with the following property.
Any unitary representation $\pi \colon G \to \mathcal U(\mathcal H)$ into a Hilbert space which admits $(S,\epsilon)$-invariant vectors has a non-zero invariant vector.
	Such a pair $(S,\epsilon)$ is called a Kazhdan pair.
\end{defi}

Let us also recall the following useful statement.
\begin{lemm}
\label{res: characterization l2 finite}
	Assume that the action of $G$ on $Y$ is transitive.
	The set $Y$ is finite if and only if the representation $\lambda \colon G \to \mathcal U(\mathcal H)$ admits a non-zero invariant vector.
\end{lemm}

\begin{theo}
\label{res: Kazhdan - spectral gap}
	Let $(\Sigma,\sigma)$ be an irreducible subshift of finite type.
	Let $F \colon \Sigma \to \R_+^*$ be a potential with $\alpha$-bounded Hölder variations for some $\alpha \in \R_+^*$.
	Let $\mathcal L$ be the corresponding transfer operator and $\rho$ its spectral radius.
	Let $G$ be a finitely generated group with Kazhdan property (T) and $\theta \colon \Sigma \to G$ be a locally constant map.
	We assume that the corresponding extension $(\Sigma_\theta,\sigma_\theta)$ has the visibility property.	
	There exists $\eta >0$ with the following property.
	
	\medskip
	Let $Y$ be an infinite set endowed with an transitive action of $G$ and $\lambda \colon G \to \mathcal U(\mathcal H)$ be the induced unitary representation, where $\mathcal H = \ell^2(Y)$.	
	Let $\rho_\lambda$ be the spectral radius of the twisted transfer operator $\mathcal L_\lambda \colon H^\infty_\alpha(\Sigma, \mathcal H) \to H^\infty_\alpha(\Sigma,\mathcal H)$ defined by 
	\begin{displaymath}
		\mathcal L_\lambda \Phi (x) = \sum_{\sigma y = x}F(y)\lambda(y)^{-1}\Phi(y).
	\end{displaymath}
	Then $\rho_\lambda \leq (1-\eta)\rho$.
\end{theo}

\begin{proof}
	Let $(S,\epsilon)$ be a Kazhdan pair of $G$. 
	Let $\eta > 0$ be the constant given by \autoref{res: main technical theorem}.
	Let $Y$ be a set endowed with an transitive action of $G$ and $\lambda \colon G \to \mathcal U(\mathcal H)$ be the induced unitary representation, where $\mathcal H = \ell^2(Y)$.	
	Let $\mathcal L_\lambda$ be the corresponding twisted transfer operator.
	Assume that $\rho_\lambda > (1-\eta)\rho$.
	It follows from \autoref{res: main technical theorem} that $\lambda$ has an $(S,\epsilon)$-invariant vector.
	Since $(S,\epsilon)$ is a Kazhdan pair, it follows that $\lambda$ has a non-zero invariant vector.
	However the action of $G$ on $Y$ is transitive.
	Hence $Y$ is finite (\autoref{res: characterization l2 finite}).
\end{proof}

%
\section{Roblin's theorem}
%
\label{sec: roblin}
In this appendix, we provide a proof of Roblin's Theorem.
Note that the statement below does not require $H$ to be a normal subgroup.
The proof is probably well-known from the specialists in the field, however we did not find it in the literature.
It relies on a rather simple counting argument in a hyperbolic space.

\begin{theo}[compare with Roblin {\cite[Théorème~2.2.2]{Roblin:2005fn}}]
\label{res: roblin}
	Let $G$ be a group acting properly co-compactly on a hyperbolic space $X$.
	Let $H$ be a subgroup of $G$.
	We denote by $\omega_G$ and $\omega_H$ the exponential growth rates of $G$ and $H$ acting on $X$.
	If $H$ is co-amenable in $G$, then $\omega_H = \omega_G$.
\end{theo}

\medskip
Let $G$ be a group acting properly co-compactly on a hyperbolic space.
Let $\omega_G$ be the exponential growth rate of $G$ acting on $X$.
We fix a base point $o \in X$.
Given $r\in \R_+$ we define the \emph{ball} of radius $r$ to be 
\begin{displaymath}
	B(r) = \set{g \in G}{\dist {go}o \leq r}.
\end{displaymath}
Coornaert \cite{Coo93} proved that there exists $C_1 \in \R_+^*$ such that for every $r \in \R_+$,
\begin{equation}
\label{res: counting coornaert}
	e^{\omega_Gr} \leq \card{B(r)} \leq C_1 e^{\omega_Gr}.
\end{equation}
Let $\delta \in \R_+$ be the hyperbolicity constant of $X$.
Up to increasing the value of $\delta$ we can always assume that the following holds:
\begin{enumerate}
	\item The diameter of $X/G$ is at most $\delta$.
In particular, for every $x,y \in X$, there exists $g \in G$ such that $\dist x{gy} \leq \delta$.
	\item $1- C_1e^{-\omega_G\delta} > 0$.
\end{enumerate}
For every $r \in \R_+$, we denote by $S(r) = B(r) \setminus B(r-\delta)$ the \emph{sphere} of radius $r$.

\begin{lemm}
\label{res: counting element with initial segment}	
	There exists $C_2 \in \R_+^*$ with the following property.
	Given $\ell, r \in \R_+$ and $x \in X$, we denote by $U$ the set of elements $g \in B(\ell)$ such that $\gro{go}xo \geq r$.
	The cardinality of $U$ is bounded above by
	\begin{displaymath}
		\card U \leq C_2e^{\omega_G(\ell - r)}.
	\end{displaymath}
\end{lemm}

\begin{proof}
	We fix a geodesic $\geo ox$ from $o$ to $x$ and write $y$ for the point of $\geo ox$ at distance $r$ from $o$.
	According to our choice of $\delta$ that there exists $h \in G$ such that $\dist y{ho} \leq \delta$.
	Note that $\dist{ho}o \geq r - \delta$.
	Let $g \in U$.
	It follows from the four point inequality (\ref{eqn: hyperbolicity condition 1}) that 
	\begin{displaymath}
		\gro{go}o{ho} \leq \gro{go}oy + \delta \leq 2\delta
	\end{displaymath}
	Consequently
	\begin{displaymath}
		\dist{h^{-1}go}o = \dist{go}{ho} = \dist{go}o - \dist{ho}o + 2 \gro{go}o{ho} \leq \ell - r + 5\delta.
	\end{displaymath}
	Thus $h^{-1}U$ is contained in $B(\ell - r + 5 \delta)$ and the result follows from (\ref{res: counting coornaert}).
\end{proof}

Let $\ell \in \R_+$.
We denote by $\mu_\ell$ the probability measure on $G$ which is uniformly distributed on $S(\ell)$.
It follows from (\ref{res: counting coornaert}) that for every $g \in S(\ell)$ we have
\begin{equation}
\label{eqn : first estimate proba}
	\frac 1{C_1}e^{-\omega_G \ell} \leq \mu_\ell(g) \leq \frac 1{C_3}e^{-\omega_G \ell},
\end{equation}
where $C_3 = 1 - C_1e^{-\omega_G\delta}$.
Our first task is to provide an estimate for the $n$-th convolution product of $\mu_\ell$.
Later we will let $\ell$ tends to infinity.
Thus we will be particularly careful to control these estimates in terms of $\ell$.
More precisely we are going to prove the following statement.

\begin{prop}
\label{res: proba convolution}
	There exists $D \in \R_+^*$ such that for every $\ell \in \R_+$, for every $n \in \N$, for every $g \in G$, we have
	\begin{displaymath}
		\mu_\ell^{\ast n} (g) \leq D^n \left(\frac\ell\delta +1\right)^n \exp\left(-\omega_G\frac{n\ell + \dist{go}o}2\right).
	\end{displaymath}
\end{prop}

In order to prove this proposition, we introduce the following sets that will allow us to track the orbits of the random walk.
For every $\ell \in \R_+$, for every $n \in \N$, for every $g \in G$, we let
\begin{displaymath}
	\mathcal O_\ell(g,n) = \set{(u_1, \dots, u_n) \in S(\ell)^n}{u_1 \cdots u_n = g}.
\end{displaymath}
Note that if $\dist{go}o > n \ell$, then $\mathcal O_\ell(g,n)$ is empty.
We adopt the convention that a product of elements of $G$ indexed by the empty set is trivial.
It follows that  $\mathcal O_\ell(g,0)$ is empty if $g$ is non trivial and reduced to a single element (the empty tuple) if $g = 1$.

\begin{lemm}
\label{res: counting orbits}
	There exists $D_0 \in \R_+^*$ such that for every $\ell \in \R_+$, for every $n \in \N$, for every $g \in G$, we have
	\begin{displaymath}
		\card{\mathcal O_\ell(g,n)} \leq D_0^n \left(\frac\ell\delta +1\right)^n \exp\left(\omega_G\frac {n\ell - \dist{go}o}2\right).
	\end{displaymath}
\end{lemm}

\begin{proof}
	Let $C_2$ be the constant given by \autoref{res: counting element with initial segment}.
	We let 
	\begin{displaymath}
		D_0 = C_2e^{2\omega_G\delta}.
	\end{displaymath}
	Let $\ell \in \R_+$.
	We are going to prove the result by induction on $n$.
	If $n = 0$, it follows from our convention that for every $g \in G$, the set $\mathcal O_\ell(g,0)$ contains at most $1$ element, hence the result.
	Assume now that the statement holds for some $n \in \N$.
	Let $g \in G$.
	For every element $u = (u_1, \dots, u_{n+1})$ of $\mathcal O_\ell(g,n+1)$ we let $g_u = u_1\cdots u_n = gu_{n+1}^{-1}$ (according to our convention $g_u$ is trivial is $n = 0$).
	For every $k \in \N$ such that $k\delta \leq \ell$, we denote by $P_k$ the set of elements $u \in \mathcal O_\ell(g,n+1)$ such that 
	\begin{displaymath}
		k \delta \leq \gro{g_uo}o{go} < (k+1)\delta.
	\end{displaymath}
	Note that if $u = (u_1, \dots, u_{n+1})$ is an element of $\mathcal O_\ell(g,n+1)$, then $\gro{g_uo}o{go} \leq \dist{u_{n+1}o}o \leq \ell$.
	Hence the collection $(P_k)$ forms a partition of $\mathcal O_\ell(g,n+1)$.
	We are now going to estimate the cardinality of each of these sets.
	
	\medskip
	Let $k \in \N$ such that $k \delta \leq \ell$.
	We write $U_k$ for the image of $P_k$ by the projection $P_k \to S(\ell)$ sending $(u_1, \dots, u_{n+1})$ to $u_{n+1}$.
	It follows from \autoref{res: counting element with initial segment} that 
	\begin{displaymath}
		\card{U_k} \leq C_2 e^{\omega_G(\ell - k\delta)} \leq D_0e^{-\omega_G (k+2)\delta}e^{\omega_G \ell} .
	\end{displaymath}
	Let $u_{n+1}$ be an element of $U_k$ and $u = (u_1, \dots, u_{n+1})$ a pre-image of $u_{n+1}$ in $P_k$.
	By definition $(u_1, \dots, u_n)$ is an element of $\mathcal O_\ell(gu_{n+1}^{-1},n)$, whose cardinality can be bounded from above using the induction hypotheses.
	It follows that
	\begin{equation}
	\label{eqn: counting orbits - 1}
		\begin{split}
			\card{P_k}
			& \leq \sum_{u_{n+1} \in U_k} \card{\mathcal O_\ell(gu_{n+1}^{-1},n)} \\
			& \leq \sum_{u_{n+1} \in U_k} D_0^n \left(\frac\ell\delta +1\right)^n \exp\left(\omega_G\frac {n\ell - \dist{gu_{n+1}^{-1}o}o}2\right).
		\end{split}
	\end{equation}
	Observe that for any $u_{n+1} \in U_k$ we have $\gro{gu_{n+1}^{-1}o}o{go} < (k+1)\delta$
	Hence 
	\begin{displaymath}
		\dist{gu_{n+1}^{-1}o}o \geq \dist{go}o + \dist{u_{n+1}o}o - 2 \gro{gu_{n+1}^{-1}o}o{go} \geq \dist{go}o + \ell - 2(k+2)\delta.
	\end{displaymath}
	Consequently (\ref{eqn: counting orbits - 1}) becomes
	\begin{align*}
	\label{eqn: counting orbits - 2}
		\card{P_k}
		& \leq \sum_{u_{n+1} \in U_k} e^{\omega_G(k+2)\delta}D_0^n \left(\frac\ell\delta +1\right)^n \exp\left(\omega_G\frac {(n-1)\ell - \dist{go}o}2\right) \\
		& \leq \card{U_k} e^{\omega_G(k+2)\delta}D_0^n \left(\frac\ell\delta +1\right)^n \exp\left(\omega_G\frac {(n-1)\ell - \dist{go}o}2\right) .
	\end{align*}
	We now use the above estimate of $\card{U_k}$ to get
	\begin{displaymath}
		\card{P_k} \leq D_0^{n+1} \left(\frac\ell\delta +1\right)^n \exp\left(\omega_G\frac {(n+1)\ell - \dist{go}o}2\right). 
	\end{displaymath}
	Note that this estimate does not depends on $k$.
	Moreover there are at most $\ell/\delta +1$ integer $k \in \N$ such that $k\delta \leq \ell$.
	Since $(P_k)$ forms a partition of $\mathcal O_\ell(g,n+1)$ we obtain
	\begin{displaymath}
		\card{\mathcal O_\ell(g,n+1)}
		\leq \sum_{k\delta \leq \ell} \card{P_k}
		\leq D_0^{n+1} \left(\frac\ell\delta +1\right)^{n+1} \exp\left(\omega_G\frac {(n+1)\ell - \dist{go}o}2\right).
	\end{displaymath}
	Hence the statement holds for $n+1$, which completes the proof of the proposition.
\end{proof}

\begin{proof}[Proof of \autoref{res: proba convolution}]
	We denote by $C_3$ and $D_0$ the constants given by (\ref{eqn : first estimate proba}) and \autoref{res: counting orbits} respectively and let $D = D_0/C_3$.
	Let $\ell \in \R_+$.
	Let $n \in \N$ and $g \in G$.
	It follows from the definition of the convolution that 
	\begin{displaymath}
		\mu_\ell^{\ast n}(g) = \sum_{(u_1, \dots, u_n) \in O_\ell(g,n)} \mu_\ell(u_1)\cdots \mu_\ell(u_n).
	\end{displaymath}
	Combining (\ref{eqn : first estimate proba}) and \autoref{res: counting orbits}, the previous equality becomes
	\begin{align*}
		\mu_\ell^{\ast n}(g)
		& \leq \left(\frac {D_0}{C_3}\right)^n \left(\frac\ell\delta +1\right)^n \exp\left(-\omega_G\frac{n\ell + \dist{go}o}2\right) \\
		& \leq D^n \left(\frac\ell\delta +1\right)^n \exp\left(-\omega_G\frac{n\ell + \dist{go}o}2\right). \qedhere
	\end{align*}
\end{proof}

We now fix a subgroup $H$ of $G$ and write $\omega_H$ for the exponential growth rate of $H$ acting on $X$.
We denote by $Y$ the set of left $H$-cosets in $G$.
The group $G$ acts on $Y$ by right translations.
We write $\mathcal H = \ell^2(Y)$ for the set of square summable functions from $Y$ to $\C$ and $\lambda \colon G \to \mathcal U( \mathcal H)$ for the regular representation of $G$ relative to $H$.
Given $\ell \in \R_+$, we consider the random walk on $Y$ associated to the probability measure $\mu_\ell$.
Said differently for every $y \in Y$ and $g \in G$ the probability of going from $y$ to $y\cdot g$ is $\mu_\ell (g)$.
Let $y_0$ be the point of $Y$ corresponding to $H$.
Note that its stabilizer is exactly $H$.
Hence the probability $p_\ell(n)$ that after $n$-step, the random walk starting to $y_0$ goes back to $y_0$ is exactly
\begin{displaymath}
	p_\ell(n) = \mu_\ell^{\ast n} (H).
\end{displaymath}
We associate to this random walk a Markov operator $M_\ell$ on $\mathcal H$.
\begin{displaymath}
	M_\ell \phi = \sum_{g \in G} \mu_\ell (g) \lambda(g) \phi, \quad \forall \phi \in \mathcal H.
\end{displaymath}
Since $\mu_\ell$ is symmetric, $M_\ell$ is a self-adjoint operator.
It follows that its spectral radius $\rho_\ell$ can be computed as follows -- see for instance \cite[Lemma~10.1]{Woe00}.
\begin{displaymath}
	\rho_\ell 
	= \limsup_{n \to \infty} \sqrt[n]{p_\ell(n)}
	= \limsup_{n \to \infty} \sqrt[n]{\mu_\ell^{\ast n}(H)}.
\end{displaymath}
The next proposition relates the spectral radius $\rho_\ell$ to the critical exponents $\omega_H$ and $\omega_G$.

\begin{prop}
\label{res: roblin - maj spec rad}
	The growth rates of $H$ and $G$ acting on $X$ satisfy the following inequality
	\begin{displaymath}
		\ln \rho_\infty \leq \max \left\{ - \frac 12 \omega_G, \omega_H - \omega_G \right\},
	\end{displaymath}
	where $\rho_\infty = \limsup_{\ell \to \infty} \sqrt[\ell]{\rho_\ell}$.
\end{prop}

\begin{proof}
	We fix $\epsilon \in \R_+^*$ such that if $\omega_H < \omega_G/2$, then $\omega_H + \epsilon < \omega_G/2$.
	It follows from the definition of exponential growth rate that there exists $A \in \R_+^*$ such that for every $r \in \R_+$,
	\begin{equation}
	\label{eqn: roblin - maj spec rad}
		\card{H \cap S(r)} \leq \card{H \cap B(r)} \leq A e^{(\omega_H + \epsilon) r}.
	\end{equation}
	We write $D$ for the constant given by \autoref{res: proba convolution}.
	Let $\ell \in \R_+$ and $n \in \N$.
	Our first task is to bound $\mu_\ell^{\ast n}(H)$ from above.
	To that end we partition $H$ according to the length of its elements.
	\begin{displaymath}
		\mu_\ell^{\ast n} (H)
		= \sum_{k \in \N} \sum_{h \in H \cap S(k\delta)} \mu_\ell^{\ast n}(h).
	\end{displaymath}
	Note that if $n\ell > k\delta$, then the probabilities $\mu_\ell^{\ast n}(h)$ vanish.
	Using \autoref{res: proba convolution}, we get
	\begin{displaymath}
		\mu_\ell^{\ast n} (H)
		\leq \sum_{k \delta \leq n\ell} \card{H \cap S(k \delta)} D^n \left(\frac\ell\delta +1\right)^n \exp\left(-\omega_G\frac{n\ell + (k-1)\delta}2\right).
	\end{displaymath}
	Combined with (\ref{eqn: roblin - maj spec rad}) it yields
	\begin{equation}
	\label{eqn: roblin - maj spec rad - pre aristote}
		\mu_\ell^{\ast n} (H) \leq Ae^{\frac 12\omega_G\delta} D^n \left(\frac \ell \delta +1\right)^n e^{-\frac12\omega_G n\ell} \sum_{k \delta \leq n\ell}e^{\left(\omega_H + \epsilon-\frac 12\omega_G\right)k\delta}.
	\end{equation}
	We now distinguish two cases.
	Assume first that $\omega_H < \omega_G$.
	It follows from our choice of $\epsilon$ that $\omega_H + \epsilon-\omega_G/2 < 0$.
	Hence a (rather brutal !) majoration in (\ref{eqn: roblin - maj spec rad - pre aristote}) gives
	\begin{displaymath}
		\mu_\ell^{\ast n} (H) \leq Ae^{\frac 12\omega_G\delta} D^n \left(\frac \ell \delta +1\right)^n \left( \frac {n\ell}{\delta} +1\right) e^{-\frac12\omega_G n\ell}.
	\end{displaymath}
	Consequently 
	\begin{displaymath}
		\ln\rho_\ell \leq \ln D + \ln \left(\frac \ell \delta +1\right) - \frac 12 \omega_G\ell.
	\end{displaymath}
	This inequality holds for every $\ell \in \R_+$.
	Consequently
	\begin{displaymath}
		\ln \rho_\infty 
		= \limsup_{\ell \to \infty} \frac 1\ell \ln \rho_\ell
		\leq  - \frac 12 \omega_G,
	\end{displaymath}
	which completes the first case.
	Assume now that $\omega_H \geq \omega_G/2$.
	Computing the sum in (\ref{eqn: roblin - maj spec rad - pre aristote}) we get
	\begin{displaymath}
		\mu_\ell^{\ast n} (H) \leq Ae^{\frac 12\omega_G\delta} D^n \left(\frac \ell \delta +1\right)^n e^{-\frac12\omega_G n\ell} 
		\frac{e^{\left(\omega_H + \epsilon-\frac 12\omega_G\right)(n\ell + \delta)}-1}{e^{\left(\omega_H + \epsilon-\frac 12\omega_G\right)\delta}-1}.
	\end{displaymath}
	Consequently
	\begin{displaymath}
		\ln \rho_\ell 
		\leq \ln D + \ln\left(\frac \ell \delta +1\right) + (\omega_H + \epsilon -\omega_G) \ell.
	\end{displaymath}
	Since this inequality holds for every $\ell \in \R_+$, we get $\ln \rho_\infty \leq \omega_H + \epsilon - \omega_G$.
	This last inequality holds for every $\epsilon > 0$, hence the result.
\end{proof}

\begin{proof}[Proof of \autoref{res: roblin}]
	Assume now that $H$ is co-amenable in $G$.
	According to Kesten's criterion, the spectral radius of any of the Markov operator $M_\ell$ is $1$.
	It follows from \autoref{res: roblin - maj spec rad} that $\omega_H \geq \omega_G$.
	The other inequality is obvious.
\end{proof}

\paragraph{Remark.}
The exact same strategy can be used to provide a lower bound for $\mu_\ell^{\ast n}$ of the same kind than the one given in \autoref{res: proba convolution}.
This leads to the more general version of \autoref{res: roblin - maj spec rad}

\begin{prop}
\label{res: roblin - equ spec rad}
	The limit $\rho_\infty = \lim_{\ell \to \infty}\sqrt[\ell]{\rho_\ell}$ exists. 
	Moreover
	\begin{displaymath}
		\ln \rho_\infty = \max \left\{ - \frac 12 \omega_G, \omega_H - \omega_G \right\}.
	\end{displaymath}
\end{prop}

\noindent
\emph{R\'emi Coulon} \\
Univ Rennes, CNRS \\
IRMAR - UMR 6625 \\
 F-35000 Rennes, France\\
\texttt{remi.coulon@univ-rennes1.fr} \\
\texttt{http://rcoulon.perso.math.cnrs.fr}

\medskip
\noindent
\emph{Fran\c coise Dal'Bo} \\
Univ Rennes, CNRS \\
IRMAR - UMR 6625 \\
 F-35000 Rennes, France\\
\texttt{francoise.dalbo@univ-rennes1.fr} \\
\texttt{https://perso.univ-rennes1.fr/francoise.dalbo}

\medskip
\noindent
\emph{Andrea Sambusetti} \\
Dipartimento di Matematica \\
Sapienza Universit\`a di Roma \\
P.le Aldo Moro 5, 00185 Roma, Italy \\
\texttt{sambuset@mat.uniroma1.it} \\
\texttt{http://www1.mat.uniroma1.it/people/sambusetti/andreas\_webpage/home.html}

\end{document}